\documentclass[a4paper]{amsart}

\usepackage[utf8]{inputenc}
\usepackage[T1]{fontenc}
\usepackage{lmodern, enumerate}
\usepackage{amssymb,amsxtra}
\usepackage[all]{xy}
\usepackage{nicefrac,mathtools}
\usepackage{microtype}
\usepackage{amscd}
\usepackage{xcolor}
\usepackage[pdftitle={Exotic Baum-Connes conjecture},
 pdfauthor={Alcides Buss, Siegfried Echterhoff, Rufus Willet},
 pdfsubject={Mathematics}
]{hyperref}
\usepackage[lite]{amsrefs}
\newcommand*{\MRref}[2]{ \href{http://www.ams.org/mathscinet-getitem?mr=#1}{MR #1}}
\newcommand*{\arxiv}[1]{\href{http://www.arxiv.org/abs/#1}{arXiv: #1}}

\numberwithin{equation}{section}
\theoremstyle{plain}
\newtheorem{theorem}[equation]{Theorem}
\newtheorem{lemma}[equation]{Lemma}
\newtheorem{proposition}[equation]{Proposition}
\newtheorem{corollary}[equation]{Corollary}
\theoremstyle{definition}
\newtheorem{definition}[equation]{Definition}
\theoremstyle{remark}
\newtheorem{remark}[equation]{Remark}
\newtheorem{example}[equation]{Example}
\newtheorem{question}[equation]{Question}

\DeclareMathOperator{\Aut}{Aut}% automorphism group
\DeclareMathOperator{\cspn}{\overline{span}}

\DeclareMathOperator{\Ind}{Ind}

\renewcommand{\top}{{\operatorname{top}}}

\renewcommand{\H}{\mathcal H}
\newcommand*{\nb}{\nobreakdash}
\newcommand*{\Star}{\(^*\)\nobreakdash-}
\newcommand*{\dd}{\,d}

\newcommand{\cB}{\mathcal L} % -> changed notation here to be consistent with our other use: \Lb below...

\newcommand*{\C}{\mathbb C}

\newcommand*{\Z}{\mathbb Z}
\newcommand*{\Real}{\mathbb R}
\newcommand*{\N}{\mathbb N}

\newcommand*{\Lb}{\mathcal L}%adjointable operators on a Hilbert module
\newcommand*{\K}{\mathcal K}% Compact operators
\newcommand*{\M}{\mathcal{M}} % multiplier algebra

\newcommand*{\Mat}{\mathbb{M}} % matrices

\newcommand*{\Cst}{C^*} % C*-algebra
\newcommand*{\cont}{C}% continuous functions
\newcommand*{\contz}{\cont_0}%continuous functions vanishing at infinity
\newcommand*{\contc}{\cont_c}%continuous functions with compact support
\newcommand*{\id}{\textup{id}}%identity
\newcommand*{\Ad}{\textup{Ad}}%conjugation by a unitary

\newcommand*{\U}{\mathcal U}% unitaries
\newcommand*{\E}{\mathcal E}% Hilbert modules
\newcommand*{\EE}{\mathbb E}% Integral expectation
\newcommand*{\defeq}{\mathrel{\vcentcolon=}}% used for definitions
\newcommand*{\congto}{\xrightarrow\sim}

\newcommand*{\braket}[2]{\langle#1\!\mid\!#2\rangle}

\newcommand*{\sbe}{\subseteq} % inclusion
\newcommand*{\F}{\mathcal F}
\newcommand*{\cstar}{\texorpdfstring{$C^*$\nobreakdash-\hspace{0pt}}{*-}}
\newcommand*{\into}{\hookrightarrow}
\newcommand*{\onto}{\twoheadrightarrow}
\newcommand*{\red}{\mathrm{r}}
\renewcommand*{\max}{\mathrm{max}}
\newcommand*{\un}{\mathrm{u}}

\newcommand{\hotimes}{\hat\otimes}
\newcommand{\eps}{\varepsilon}

\newcommand*{\Cor}{\mathfrak{Corr}}% category of C*-algebras with homomorphisms for trivial groups
\newcommand*{\Free}{\mathbb{F}}

\newcommand{\ie}{\emph{i.e.}}
\newcommand{\eg}{\emph{e.g.}}
\newcommand{\cf}{\emph{cf.}}

\newcommand{\as}{\operatorname{as}}
\newcommand{\jas}{\operatorname{J-as}}

\newcommand{\ind}{\operatorname{ind}}

\newcommand{\dach}{{\!\widehat{\ \ }}}

\newcommand{\Manoa}{M\=anoa}

\newcommand{\Hawaii}{Hawai\kern.05em`\kern.05em\relax i}

\begin{document}
\title[Exotic Baum-Connes conjecture]{Exotic crossed products and the \\ Baum-Connes conjecture}

\author{Alcides Buss}
\email{alcides@mtm.ufsc.br}
\address{Departamento de Matem\'atica\\
 Universidade Federal de Santa Catarina\\
 88.040-900 Florian\'opolis-SC\\
 Brazil}

\author{Siegfried Echterhoff}
\email{echters@math.uni-muenster.de}
\address{Mathematisches Institut\\
 Westf\"alische Wilhelms-Universit\"at M\"un\-ster\\
 Einsteinstr.\ 62\\
 48149 M\"unster\\
 Germany}

\author{Rufus Willett}
\email{rufus@math.hawaii.edu}
\address{Mathematics Department\\
 University of \Hawaii~at \Manoa\\
Keller 401A \\
2565 McCarthy Mall \\
 Honolulu\\
 HI 96822\\
USA}

\begin{abstract}
We study general properties of exotic crossed-product functors and characterise those which extend to functors on  equivariant \cstar{}algebra categories based on correspondences.
We show that every such functor allows the construction of a descent in $KK$-theory and we use
this to show that all crossed products by correspondence  functors of $K$-amenable groups are $KK$-equivalent.
 We also show that for second countable groups the minimal exact Morita compatible crossed-product functor
 used in the new formulation of the Baum-Connes conjecture by Baum, Guentner and Willett (\cite{Baum-Guentner-Willett})
extends to correspondences when restricted to separable $G$\nb-\cstar{}algebras.
It therefore allows a descent in $KK$-theory for separable systems.
\end{abstract}


\thanks{Supported by Deutsche Forschungsgemeinschaft  (SFB 878, Groups, Geometry \& Actions), by CNPq/CAPES -- Brazil, and by the US NSF (DMS 1401126).}


\maketitle

\section{Introduction}
The concept of a crossed product $A\rtimes_\alpha G$ by an action $\alpha:G\to \Aut(A)$ of a locally compact group $G$ on a \cstar{}Algebra
$A$ by *-automorphisms plays a very important role in the fields of Operator Algebras and Noncommutative Geometry. Classically, there
are two crossed-product constructions which have been used in the literature: the universal (or maximal) crossed product $A\rtimes_{\alpha,\un}G$
and the reduced (or minimal/spatial) crossed product $A\rtimes_{\alpha,\red}G$.  The universal crossed product satisfies a universal property for
covariant representations $(\pi, u)$ of the underlying system $(A,G,\alpha)$ in such a way that every such representation integrates to
a unique representation $\pi\rtimes u$ of $A\rtimes_{\alpha,\un}G$, while the reduced crossed product is the image of $A\rtimes_{\alpha,\un}G$
under the integrated form of the regular covariant representation of  $(A,G,\alpha)$.
In case $A=\C$ we recover the construction of the maximal and reduced group algebras
$C^*(G)=\C\rtimes_{\un}G$ and $C_\red^*(G)=\C\rtimes_{\red}G$ of $G$, respectively.

The constructions of maximal and reduced crossed product extend to functors between categories with several sorts of morphisms (like equivariant
\Star{}homomorphisms or correspondences) and each one of these functors has its own special features. For instance the universal crossed-product functor always preserves short exact sequences of equivariant homomorphisms, while the reduced one only does if the group $G$ is exact. On the other hand, the reduced crossed product always takes embeddings to embeddings, while the full one only does if the group $G$ is amenable.

More recently (\eg, see \cite{Brown-Guentner:New_completions, Kaliszewski-Landstad-Quigg:Exotic, Kaliszewski-Landstad-Quigg:Exotic-coactions,
Baum-Guentner-Willett, Buss-Echterhoff:Exotic_GFPA, Okayasu:Free-group}),
 there has been a growing interest in the study of exotic (or intermediate) group \cstar{}algebras $C_\mu^*(G)$ and crossed products
 $A\rtimes_{\alpha,\mu}G$, where the word ``intermediate'' indicates that they ``lie'' between the universal and reduced group algebras
or crossed products. To be more precise: both classical crossed products can be obtained as completions of the convolution algebra
$C_c(G,A)$ with respect  to the universal norm $\|\cdot\|_{\un}$ or the reduced norm $\|\cdot\|_{\red}\leq\|\cdot\|_\un$, respectively. The intermediate crossed products
$A\rtimes_{\alpha,\mu}G$ are completions of $C_c(G,A)$ with respect to a \cstar{}norm $\|\cdot\|_\mu$ such that
$$\| f \|_{\red}\leq \|f\|_{\mu}\leq \| f\|_{\un}$$
for all $f\in C_c(G,A)$. This is equivalent to saying that the identity on $C_c(G,A)$ induces canonical surjective *-homomorphisms
$$A\rtimes_{\alpha,\un}G\onto A\rtimes_{\alpha,\mu}G \onto A\rtimes_{\alpha,\red}G.$$
The main goal of this paper is to study functors $(A,\alpha)\mapsto A\rtimes_{\alpha,\mu}G$ from the category of $G$-\cstar{}algebras
into the category of \cstar{}algebras  which assign to
each $G$-algebra $(A,\alpha)$ an intermediate crossed product $A\rtimes_{\alpha,\mu}G$.

There are several reasons why researchers became interested in intermediate crossed products and group algebras. On one side intermediate
group algebras provide interesting new examples of  \cstar{}algebras attached to certain representation theoretic properties of the group. Important examples
are the group algebras $C_{E_p}^*(G)$ attached to the unitary representations of $G$ which have a dense set of matrix coefficients  in $L^p(G)$, for $1\leq p \leq \infty$,
or group algebras attached to other growth conditions on their matrix coefficients.
It is shown in \cite{Brown-Guentner:New_completions}*{Proposition 2.11} that for every discrete group $G$ and $1\leq p\leq 2$
we have $C_{E_p}^*(G)=C_\red^*(G)$. But if
$G=\Free_2$ (or any discrete group which contains $\Free_2$) then Okayasu shows in \cite{Okayasu:Free-group} that
all $C_{E_p}^*(G)$ are different for $2\leq p\leq \infty$. Very recently, a similar result has been shown by Wiersma for $\mathrm{SL}(2,\mathbb R)$ (\cite{Wiersma}).
As one important outcome of the results of this paper we shall see that for every $K$-amenable group $G$, like $G=\Free_2$ or $G=\mathrm{SL}(2,\mathbb R)$, all these different group algebras are $KK$-equivalent. On the other hand, it has been already observed in \cite{Baum-Guentner-Willett}*{Example 6.4} that there are exotic group algebra completions $C^*_\mu(G)$ of $G=\Free_2$ which do not have the same $K$-theory as $C^*_r(G)$ or $C^*(G)$. This means that the $K$\nb-theory of such algebras depends not only on the structure of the group $G$, but also on the structure of the completion $C^*_{\mu}(G)$, or more precisely, on the structure of the crossed-product functor $A\mapsto A\rtimes_\mu G$. Our results will help us to understand which properties a crossed-product functor should have in order to behave well with $K$-theory and other constructions.

This point also brings us to another important reason for the growing interest on intermediate crossed products, which is motivated by a very recent
new formulation of the Baum-Connes conjecture due to Baum, Guentner, and Willett in \cite{Baum-Guentner-Willett}.
Recall that the Baum-Connes conjecture predicted that the $K$-theory of a reduced crossed product $A\rtimes_{\alpha,\red}G$
can be computed with the help of a canonical assembly map
$$\as_{(A,G)}^{\red}:K_*^{\top}(G;A)\to K_*(A\rtimes_{\alpha,\red}G),$$
in which $K_*^{\top}(G;A)$, the topological $K$-theory of $G$ with coefficient $A$, can be computed (at least
in principle) by more classical topological methods. The original Baum-Connes conjecture stated  that this assembly
map should always be an isomorphism. But in \cite{Higson-Lafforgue-Skandalis} Higson, Lafforgue and Skandalis
provided counter examples for the conjecture which are based on a construction of non-exact groups
due to Gromov and others. Recall that a group is called {\em exact} if every short exact sequence of $G$-algebras descends
to a short exact sequence of their {\em reduced} crossed products.

In \cite{Baum-Guentner-Willett} it is shown
that for every  group $G$, there exists a crossed-product functor $(A,\alpha)\mapsto A\rtimes_{\alpha,\E}G$ which is minimal
among all exact and (in a certain sense) Morita compatible crossed-product functors for $G$.
Using the assembly map for the universal crossed product $A\rtimes_{\alpha,\un} G$ together with the canonical quotient map, we may construct
an assembly map
\begin{equation}\label{eq-as}
\as_{(A,G)}^{\mu}:K_*^{\top}(G;A)\to K_*(A\rtimes_{\alpha,\mu}G)
\end{equation}
for every intermediate crossed-product functor $\rtimes_{\mu}$.
Using $E$-theory, Baum, Guentner and Willett show that all known counter-examples to the conjecture disappear
if we replace  reduced crossed products by  $\E$-crossed products in the formulation of the conjecture.
 If $G$ is exact, the $\E$-crossed product will always be the reduced one, hence the new formulation of the conjecture
 coincides with the classical one for such groups.

Note that the authors of \cite{Baum-Guentner-Willett} use $E$-theory instead of $KK$-theory since the proof requires a
direct construction of the assembly map, which in the world of $KK$-theory would involve a descent homomorphism
$$J_G^\E:KK^G(A,B)\to KK(A\rtimes_{\E}G, B\rtimes_{\E}G)$$
and it was not clear whether such descent exists (due to exactness and Morita compatibility, it exists in $E$-theory).
It is one consequence of our results (see \S\ref{sec:minimal})
that a $KK$-descent does exist for $\rtimes_{\E}$ at least if we restrict to separable systems.

Indeed, in \S\ref{sec-prop} we give a systematic study of crossed-product functors $\rtimes_\mu$ which
are functorial for equivariant correspondences. This means that
for every pair of $G$-algebras $A,B$ and for every $G$-equivariant Hilbert $A-B$-bimodule
$\F$ in the sense of Kasparov, there is a canonical crossed-product construction $\F\rtimes_{\mu}G$
as a Hilbert $A\rtimes_{\mu}G-B\rtimes_{\mu}G$  bimodule. We call such functors {\em correspondence
functors}. We show in \S\ref{sec-prop} that there are a number of equivalent conditions which characterise
correspondence functors. For instance, it turns out  (see Theorem \ref{thm-corr}) that being a correspondence functor is
equivalent to the property that functoriality extends to $G$-equivariant completely positive maps.
But maybe  the most convenient of the conditions which characterise correspondence functors
is the {\em projection property} which requires that
for every $G$-algebra
$A$ and every $G$-invariant projection $p\in \M(A)$, the descent
$pAp\rtimes_{\mu}G\to A\rtimes_{\mu}G$ of the inclusion $pAp\into A$ is faithful (see Theorem \ref{thm-corr}).
This property is easily checked in special situations and we observe that many natural
examples of crossed-product functors, including all Kaliszewski-Landstad-Quigg functors
(or KLQ-functors) attached  to weak*-closed ideals in the Fourier-Stieltjes algebra $B(G)$ (see \ref{sec:KLQ} for details of the construction)
 are correspondence functors and that correspondence functors  are stable under certain
manipulations which construct new functors out of old ones.
In particular, for every given family of functors $\{\rtimes_{\mu_i}:i\in I\}$ there is a construction
of an infimum $\rtimes_{\inf\mu}$ of this family given in \cite{Baum-Guentner-Willett},
and it is easy to check that if all $\rtimes_{\mu_i}$ satisfy the projection property (\ie\ are correspondence functors), then so does $\rtimes_{\inf\mu}$. It follows from this and the fact (shown in \cite{Baum-Guentner-Willett}) that exactness is preserved by taking the infimum of a family of functors that there exists a {\em minimal exact correspondence functor} $\rtimes_{\E_\Cor}$
for each given locally compact group $G$. We show in Proposition \ref{prop-Morita} that this functor
coincides with the minimal exact Morita compatible functor $\rtimes_{\E}$ of \cite{Baum-Guentner-Willett}
if $G$ is second countable and if we restrict the functor to separable $G$-algebras, as one almost
always does
if one discusses the Baum-Connes conjecture.

In \S\ref{sec:descent} we prove that every correspondence functor will allow the construction of
a descent
$$J_G^\mu:KK^G(A,B)\to KK(A\rtimes_{\mu}G, B\rtimes_{\mu}G)$$
in equivariant $KK$-theory, which now allows one to use the full force of equivariant $KK$-theory in
the study of the new formulation of the Baum-Connes conjecture in \cite{Baum-Guentner-Willett}.
Moreover, using this descent together with the deep results of  \cite{Higson-Kasparov:Etheory} we can show that the
assembly map (\ref{eq-as}) is an isomorphism whenever $G$ is a-T-menable (or, slightly more general, if $G$ has a $\gamma$-element
equal to $1_G\in KK^G(\C,\C)$ in Kasparov's sense) and if $\rtimes_\mu$ is a correspondence functor on $G$.
By a result of Tu in \cite{Tu:Kamenable}, all such groups are $K$-amenable, and we show in Theorem \ref{thm-Kamenable}
that for every $K$-amenable group $G$ and every correspondence crossed-product functor $\rtimes_\mu$ on $G$
the quotient maps
$$A\rtimes_{\un}G\onto A\rtimes_{\mu}G\onto A\rtimes_\red G$$
are $KK$-equivalences (the proof closely follows the line of arguments given by Julg and Valette in
\cite{Julg-Valette:Kamenable}*{Proposition 3.4}, but the main ideas are due to Cuntz in \cite{Cuntz:Kamenable}).
In particular, since all $L^p$-group algebras $C_{E_p}^*(G)$ for $p\in [1,\infty]$ are group algebras corresponding to
KLQ-functors, which are correspondence functors, they all have isomorphic $K$-theory for
a given fixed $K$-amenable group $G$. In particular this holds for interesting groups like $G={\Free_2}$ or $G=\mathrm{SL}(2,\mathbb R)$, where $C_{E_p}^*(G)$ are all different for $p\in [2,\infty]$.

The outline of the paper is given as follows: after this introduction we start in \S\ref{sec:preliminaries} with some preliminaries
on intermediate crossed-product functors and the discussion of various known examples of such functors.  In \S\ref{sec:id} we isolate properties governing how crossed products behave with respect to ideals and multipliers that will be useful later. In \S\ref{sec-prop} we give many useful characterisations of correspondence functors: the main result here is Theorem \ref{thm-corr}.
In \S\ref{sec:duality} we show that crossed products by correspondence functors $\rtimes_\mu$ always admit
dual coactions. In particular it follows that the group algebras
 $C_\mu^*(G):=\C\rtimes_{\mu}G$ corresponding to any correspondence functor carries a dual coaction
 of $G$ and hence corresponds to a translation invariant  weak-* closed ideal $E\subseteq B(G)$ as described in
 \S\ref{sec:id}. We also prove a version of Imai-Takai duality for the dual crossed-product. Then, in \S\ref{sec:descent} we show that every correspondence functor allows
a $KK$-descent and, as an application,  we use this to show
the isomorphism of the assembly map for
correspondence functors on a-T-menable groups and the results on $K$-amenability. In \S\ref{sec:lp} we give a short discussion on
the $L^p$-group algebras and applications of our results to such algebras. In \S\ref{sec:minimal} we
show that being a correspondence functor passes to the infimum of a family of correspondence functors
and we study the relation of the minimal exact correspondence functor to the functor $\rtimes_\E$
constructed in \cite{Baum-Guentner-Willett}.  Finally in \S\ref{sec:questions} we finish with some remarks and open questions.

Part of the work on this paper took place during visits of the second author to the Federal University of Santa Catarina, and of the third author to the Westf\"{a}lische Wilhelms-Universit\"{a}t, M\"{u}nster.  We would like to thank these institutions for their hospitality.

\label{sec:introduction}

\section{Preliminaries on crossed-product functors}\label{sec:preliminaries}
 Let $(B,G, \beta)$ be a \cstar{}dynamical system.
By a covariant representation $(\pi, u)$ into the multiplier algebra $\M(D)$ of some \cstar{}algebra
$D$, we understand a \Star{}ho\-mo\-mor\-phism $\pi:B\to \M(D)$ together with a strictly continuous
\Star{}homomorphism $u: G\to \U\M(D)$ such that $\pi(\beta_s(b))=u_s\pi(b)u_s^*$.
If $\E$ is a Hilbert module (or Hilbert space) then a covariant representation on $\E$ will be the
same as a covariant representation into $\M(\K(\E))\cong \Lb(\E)$.
We say that a covariant representation is {\em nondegenerate} if $\pi$  is nondegenerate in the sense
that $\pi(B)D=D$.

If $(B,G,\beta)$ is a \cstar{}dynamical system, then
 $\contc(G,B)$ becomes a \Star{}algebra with respect to the usual convolution and involution:
$$f*g(t)\defeq\int_G f(s)\beta_s(g(s^{-1}t))\dd{s}\quad\mbox{and}\quad f^*(t)\defeq \Delta(t)^{-1}\beta_t(f(t^{-1}))^*.$$
The full (or universal) crossed product $B\rtimes_{\beta,\un} G$ is the completion of $\contc(G,B)$ with respect to the universal norm
 $\|f\|_\un:=\sup_{(\pi,u)}\|\pi\rtimes u(f)\|$ in which
$(\pi,u)$ runs through all covariant representations of $(B,G,\beta)$ and
$$\pi\rtimes u(f)=\int_G\pi(f(s))u_s\,ds$$
denotes the integrated form of a covariant representation $(\pi,u)$.
By definition of this norm,
each integrated form $\pi\rtimes u$ extends uniquely to a \Star{}representation of
$B\rtimes_{\beta,\un}G$ and this extension process gives a one-to-one correspondence
between the nondegenerate covariant \Star{}representations $(\pi,u)$ of $(B,G,\beta)$ and
the nondegenerate \Star{}representations of $B\rtimes_{\beta,\un}G$.

There is a canonical  representation $(\iota_B,\iota_G)$ of $(B,G, \beta)$ into
$\M(B\rtimes_{\beta,\un}G)$ which is determined by
$$\big(\iota_B(b)\cdot f\big)(t)= bf(t),\quad \text{and}\quad
\big(\iota_G(s)\cdot f\big)(t)=\beta_s(f(s^{-1}t)), $$
for all $b\in B, s\in G$, and $f\in C_c(G,B)$. Then, for a given nondegenerate
\Star{}homomorphism  $\Phi:B\rtimes_{\beta,\un}G\to \M(D)$ we have
$\Phi=\pi\rtimes u$ for
$$\pi=\Phi\circ i_B\quad\text{and}\quad u=\Phi\circ i_G.$$
In this sense we get the identity $B\rtimes_{\beta,\un} G=i_B\rtimes i_G(B\rtimes_{\beta,\un} G)$.

If $B=\C$, then we recover the full group algebra $C^*(G)=\C\rtimes_{\id,\un}  G$ and since every
nondegenerate covariant representation of $(\C,G,\id)$ is of the form $(1, u)$ for some unitary representation $u$ of $G$,
this gives the well-known correspondence between unitary representations of $G$ and
nondegenerate representations of $C^*(G)$. In this case we denote the canonical
representation of $G$ into $\U\M(C^*(G))$ by $u_G$.

For a locally compact group $G$, let $\lambda_G: G\to \U(L^2(G))$ denote the left regular representation given
by $(\lambda_G(s)\xi)(t)=\xi(s^{-1}t)$ and let $M: C_0(G)\to {\cB}(L^2(G))$ denote the representation
of $C_0(G)$ by multiplication operators. Then, for each system $(B,G,\beta)$, there is a canonical
covariant homomorphism $(\Lambda_B, \Lambda_G)$ into $\M(B\otimes \K(L^2(G)))$ given as follows:
Let $\tilde\beta: B\to C_b(G,B)\subseteq \M(B\otimes C_0(G))$ be the \Star{}homomorphism which sends $b\in B$
to $\big(s\mapsto \beta_{s^{-1}}(b)\big)\in C_b(G,B)$. Then
\begin{equation}\label{eq-reg-rep}
(\Lambda_B, \Lambda_G):=\big((\id_B\otimes M)\circ \tilde\beta, 1_B\otimes \lambda_G).
\end{equation}
The {\em reduced crossed product} $B\rtimes_{\beta,\red}G$ is defined as the image of $B\rtimes_{\beta,\un}G$
under the integrated form $\Lambda:=\Lambda_B\rtimes \Lambda_G$.  Note that $\Lambda$ is always faithful
on the dense subalgebra $C_c(G,B)$ of $B\rtimes_{\beta,\un}G$, so that we may regard $B\rtimes_{\beta,\red}G$
also as the completion of $C_c(G,B)$ given by the norm $\|f\|_\red=\|\Lambda(f)\|$ for $f\in C_c(G,B)$.
If $B=\C$, we get $\C\rtimes_{\red}G=\lambda_G(C^*(G))=C_\red^*(G)$, the \emph{reduced group \cstar{}algebra} of $G$.
\medskip

\subsection{Exotic crossed products}
We shall now consider general crossed-prod\-uct functors $(B,\beta)\mapsto B\rtimes_{\beta, \mu}G$ such that
the $\mu$-crossed product $B\rtimes_{\beta,\mu}G$ can be obtained as some \cstar{}completion of the
convolution algebra $C_c(G,B)$. We always want to require that the $\mu$-crossed products
lie {\em between} the full and reduced crossed products in the sense that the identity map on $C_c(G,B)$
induces surjective \Star{}homomorphisms
$$B\rtimes_{\beta,\un}G\onto B\rtimes_{\beta,\mu}G\onto B\rtimes_{\beta,\red}G$$
for all $G$\nb-algebras $(B,\beta)$.
Such crossed-product functors have been studied quite recently by several authors, and we
shall later recall
the most prominent constructions as discussed in \cite{Brown-Guentner:New_completions, Kaliszewski-Landstad-Quigg:Exotic, Baum-Guentner-Willett, Buss-Echterhoff:Exotic_GFPA, Buss-Echterhoff:Imprimitivity}.
Note that one basic requirement will be that this construction is functorial for $G$\nb-equivariant
\Star{}homomorphisms, \ie, whenever we have a $G$\nb-equivariant \Star{}homomorphism $\phi:A\to B$
between two $G$\nb-algebras $A$ and $B$,  there will be a \Star{}homomorphism
$\phi\rtimes_\mu G: A\rtimes_{\mu}G\to B\rtimes_\mu G$ given on the dense subalgebra
 $C_c(G,A)$ by  $f\mapsto \phi\circ f$.

 For any crossed-product functor $(A,\alpha)\mapsto A\rtimes_{\alpha,\mu}G$ there
is a canonical covariant homomorphism  $(i_A^\mu, i_G^\mu)$ of  $(A,G,\alpha)$ into
$\M(A\rtimes_{\alpha,\mu}G)$ given by the composition of the canonical covariant
homomorphism $(i_A,i_G)$ into the universal crossed product followed by the
quotient map $q_{A,\mu}:A\rtimes_{\alpha,\un}G\to A\rtimes_{\alpha,\mu}G$.
If we compose it with the quotient map $q_{\red}^\mu:A\rtimes_{\alpha,\mu}G\to A\rtimes_{\alpha,\red}G$,
this will give the inclusions $(i_A^\red, i_G^\red)$ into the reduced crossed product
$A\rtimes_{\alpha,\red}G$. Since the latter are known to be injective on $A$ and $G$,
the same is true $i_A^\mu$ and $i_G^\mu$.
Moreover, since $i_A^\mu:A\to \M(A\rtimes_{\alpha,\mu}G)$ is nondegenerate, it uniquely extends
to $\M(A)$ (compare with \cite[Lemma 4.2]{Baum-Guentner-Willett}).

Let us briefly discuss some particular examples of exotic crossed-product functors which
have been introduced in the recent literature (\eg, see the appendix of \cite{Baum-Guentner-Willett}).
For the discussion recall that if $\pi:A\to C$  and $\rho:A\to D$ are two \Star{}homomorphisms of a given \cstar{}algebra
$A$ into \cstar{}algebras $C$ and $D$, then $\pi$ is said to be {\em weakly contained} in $\rho$, denoted $\pi{\preceq} \rho$,  if
$\ker\rho\subseteq \ker\pi$, or, equivalently, if  $\|\pi(a)\|\leq \|\rho(a)\|$ for all $a\in A$.
The  homomorphisms $\pi$ and $\rho$
 are called {\em weakly equivalent}, denoted $\pi\approx \rho$ if their kernels coincide, \ie, if $\pi\preceq\rho$ and $\rho\preceq\pi$.
Similarly, if $\{\rho_i: i\in I\}$ is any collection of \Star{}homomorphisms, we say that $\pi$ is weakly contained
in this collection, if $\|\pi(a)\|\leq\sup_{i\in I}\|\rho_i(a)\|$. One easily checks that in case of \Star{}representations
on Hilbert spaces, this coincides with the notion of weak containment as introduced by Fell in {\cite{Fell:Weak_Containment}}. If $G$ is a group and $u,v$ are unitary representations of $G$, we write $u\preceq v$ if and only if this holds for their integrated forms on $C^*(G)$.

\subsection{Brown-Guentner crossed products.}\label{sec:BG}
The Brown-Guentner crossed products (or BG-crossed products for short) have been introduced by Brown and Guentner
in \cite{Brown-Guentner:New_completions} and can be described as follows: Let $G$ be a locally compact group and fix any unitary representation $v$ of $G$ which weakly contains the regular representation $\lambda_G$.
Then $C_v^*(G):=v(C^*(G))$ is
a \cstar{}algebra which lies between the maximal and reduced group algebras of $G$.  Clearly, any such algebra is of this form for some $v$ and is called an \emph{exotic group algebra} of $G$. Suppose now that $(B,G,\beta)$ is a system.
Let  $$I_{\beta,v}:=\cap\{\ker(\pi\rtimes u): u\preceq v\}\subseteq B\rtimes_{\beta,\un}G.$$
Then  $B\rtimes_{\beta, v_{BG}}G:=(B\rtimes_{\beta,\un}G)/I_{\beta, v}$ is called the {\em BG-crossed product}
corresponding to $v$. Note that if $B=\C$, we recover the exotic group algebra $C_v^*(G)$ up to isomorphism.

It is very easy to check that this always defines a crossed-product functor.
Note that by the definition of the notion of weak containment, the BG-crossed product only depends on the ideal
$\ker v\subseteq C^*(G)$, \ie, on the exotic group algebra $C_v^*(G)\cong C^*(G)/\ker v$.
Notice that for the trivial coefficient $B=\C$, we get $\C\rtimes_{v_{BG}}G=C^*_v(G)$.
We refer to \cite{Brown-Guentner:New_completions, Baum-Guentner-Willett} for more information on BG-crossed-product functors.

\begin{example}[{\cf \cite{Baum-Guentner-Willett}*{Lemma A.6}}]\label{ex-BG}
Let us consider the BG-crossed product corresponding to the regular representation $v=\lambda_G$.
Although the corresponding group algebra $C_v^*(G)$ is the reduced group algebra,
the BG-functor corresponding to $\lambda_G$  {\em does not} coincide with the reduced crossed-product functor
if $G$ is not amenable. To see this
consider the action $\alpha=\Ad\lambda_G$ of $G$ on $\K(L^2(G))$.
Then the BG-crossed product $\K(L^2(G))\rtimes_{\alpha, \lambda_{BG}}G$
is the full crossed product $\K(L^2(G))\rtimes_{\alpha, \un}G\cong C^*(G)\otimes \K(L^2(G))$ since
$(1_{C^*(G)}\otimes \id_{\K})\rtimes(u_G\otimes \lambda)$
is an isomorphism between $\K(L^2(G))\rtimes_{\alpha, \un}G$ and $C^*(G)\otimes \K(L^2(G))$ and $u_G\otimes\lambda\approx\lambda$.
It therefore differs from $\K(L^2(G))\rtimes_{\alpha,\red}G\cong C_\red^*(G)\otimes\K(L^2(G))$ if $G$ is not amenable.
Moreover, since $\alpha$ is Morita equivalent to the trivial action $\id$ on $\C$ via $(L^2(G), \lambda)$ and since $\C\rtimes_{\lambda_{BG}}G=C_\red^*(G)$,
we see that BG-crossed products do not preserve Morita equivalence!
\end{example}

\subsection{Kaliszewski-Landstad-Quigg crossed products.}\label{sec:KLQ}
The Kaliszewski-Land\-stad-Quigg crossed products (or KLQ-crossed products for short)
have been introduced by Kaliszewski, Landstad and Quigg
in \cite{Kaliszewski-Landstad-Quigg:Exotic}, and discussed since then in several papers (\eg, see
\cite{Buss-Echterhoff:Exotic_GFPA, Buss-Echterhoff:Imprimitivity, Baum-Guentner-Willett, Kaliszewski-Landstad-Quigg:Exotic-coactions}). The easiest way to introduce them is the following.
Choose any unitary representation $v$ of $G$ and  consider the universal covariant homomorphism
$(i_B, i_G)$ from $(B,G,\beta)$ into $\M(B\rtimes_{\beta,\un}G)$. Consider the covariant representation
$(i_B\otimes 1_{C_v^*(G)}, {i}_G\otimes v)$ of $(B, G,\beta)$ into $\M(B\rtimes_{\beta,\un}G\otimes C_v^*(G))$.
Let $J_{\beta, v}:=\ker\big((i_B\otimes 1_{C_v^*(G)})\rtimes ({i}_G\otimes v)\big)\subseteq B\rtimes_{\beta,\un}G$.
Then the KLQ-crossed product corresponding to $v$ is defined as
$$B\rtimes_{\beta, v_{KLQ}}G:=(B\rtimes_{\beta,\un}G)/J_{\beta,v}.$$

A different characterisation of these KLQ-crossed products can be obtained by considering the dual coaction
$$\widehat\beta:= (i_B\otimes 1_{C^*(G)})\rtimes (i_G\otimes u_G): B\rtimes_{\beta, \un}G\to
\M(B\rtimes_{\beta, \un}G\otimes C^*(G)).$$
It follows right from the definition that  $J_{\beta, v}=\ker(\id_{B\rtimes G}\otimes v)\circ \widehat\beta$. Moreover,
by the properties of minimal tensor products, this kernel only depends on the kernel $\ker v\subseteq C^*(G)$, so that
the KLQ-crossed product also only depends on the quotient $C_v^*(G)\cong C^*(G)/\ker v$.

Let us see what happens if we apply this to $B=\C$. Using
$\C\rtimes_\un G\cong C^*(G)$, the construction gives $\C\rtimes_{v_{KLQ}}G=C^*(G)/\ker(u_G\otimes v)$
which equals $C_v^*(G)=C^*(G)/\ker(v)$ if and only if $v$ is weakly equivalent to $u_G\otimes v$. In this case we say
that $v$ {\em absorbs the universal representation of $G$}. Since $\widehat\beta=(i_B\otimes 1_{C^*(G)})\rtimes (i_G\otimes u_G)$
is faithful on $B\rtimes_{\alpha,\un}G$, \ie, it is weakly equivalent to $i_B\rtimes i_G$,  it follows
that $(i_B\otimes 1_{C_v^*(G)})\rtimes ({i}_G\otimes v)$ is weakly equivalent to
$(i_B\otimes 1_{C^*(G)}\otimes 1_{C_v^*(G)})\rtimes(i_G\otimes u_G\otimes v)$.
This follows from the properties of the minimal tensor product together with the fact that for any covariant representation $(\pi,U)$ of $(B,G,\beta)$, we have $(\pi\otimes 1)\rtimes (U\otimes v)=((\pi\rtimes U)\otimes v)\circ\widehat\beta$.
But this implies that
$J_{\beta,v}=J_{\beta, u_G\otimes v}$, hence
$$B\rtimes_{\beta, v_{KLQ}}G\cong B\rtimes_{\beta, (u_G\otimes v)_{KLQ}}G.$$
Thus one can always restrict attention to representations $v$ which absorb $u_G$ when considering KLQ-crossed products.

It is proved in \cite{Buss-Echterhoff:Exotic_GFPA} that  $(B,G,\beta)\mapsto B\rtimes_{\beta,v_{KLQ}}G$ is a crossed-product functor. In \cite{Buss-Echterhoff:Imprimitivity} it is proved that  functoriality even extends to (equivariant) correspondences. We will have much more to say about this property in \S\ref{sec-prop} below.

Another way to look at KLQ-functors is via the one-to-one correspondence between ideals in
$C^*(G)$ and $G$\nb-invariant (for left and right translation) weak{*-}closed subspaces of the Fourier-Stieltjes
algebra $B(G)$, which is the algebra of all bounded continuous functions on $G$ which
can be realised as matrix coefficients $s\mapsto \braket{U_s\xi}{\eta}$ of strongly continuous unitary Hilbert-space $G$\nb-representations $U$.
Then $B(G)$ identifies with the Banach-space dual $\Cst(G)^*$ if we map the function $s\mapsto\braket{U_s\xi}{\eta}$
 to the linear functional on $\Cst(G)$ given by $x\mapsto \braket{U(x)\xi}{\eta}$, where we use the same letter for $U$ and its integrated form.
It is explained in \cite{Kaliszewski-Landstad-Quigg:Exotic}  that the one-to-one
correspondence between closed ideals in $C^*(G)$ and $G$\nb-invariant weak{*-}closed subspaces $E$ in $B(G)$
is given via $E\mapsto I_E:={\ }^\perp E=\{x\in C^*(G): f(x)=0\;\forall f\in E\}$. Hence,
we obtain all quotients of $C^*(G)$ as $C_E^*(G):=C^*(G)/I_E$ for some $E$.
The ideal $I_E$ will be contained in the kernel of the regular representation $\lambda_G$ if and only if
$E$ contains the Fourier-algebra $A(G)$ consisting of matrix coefficients of $\lambda_G$.
Moreover, the quotient map $v_E: C^*(G)\to C_E^*(G)$ absorbs the universal representation $u_G$ if and
only if $E$ is an ideal in $B(G)$.  As a consequence of this discussion we get:

\begin{proposition} \label{prop-BG-KLQ}
Let $G$ be a locally compact group. Then
\begin{enumerate}
\item there is a one-to-one correspondence between BG-crossed-product functors and
$G$\nb-invariant weak{*-}closed subspaces $E$ of $B(G)$ which contain the Fourier algebra $A(G)$
given by sending $E$ to the BG-functor associated to the quotient map $v_E:C^*(G)\to C_E^*(G)$.
\item Similarly, there is a one-to-one correspondence between KLQ-crossed-product functors
and weak{*-}closed ideals $E$ in $B(G)$ given by sending $E$ to the KLQ-functor associated
to the quotient map $v_E:C^*(G)\to C_E^*(G)$.
\end{enumerate}
\end{proposition}

As for BG-crossed-product functors, we get $\C\rtimes_{v_{KLQ}}G=C^*_v(G)$, the exotic group \cstar{}algebra for the trivial coefficient algebra $B=\C$
if $v$ absorbs the universal representation $u_G$, that is, if $v=v_E$ for some $G$-invariant ideal in $B(G)$ as above.
This is not true if $E$ is not an ideal (compare also with \cite{Buss-Echterhoff:Imprimitivity}*{Proposition~2.2}).

\begin{example}\label{ex-lp}
For  each $p\in [2,\infty]$ we define
 $E_p$ to be the weak{*-}closure of $B(G)\cap L^p(G)$, which is an ideal in $B(G)$.
 Let us write  $(B,\beta)\mapsto B\rtimes_{\beta, p_{KLQ}}G$ for the corresponding KLQ-functor.
 It then follows from \cite{Okayasu:Free-group} that for any discrete group $G$ which contains the free group
${\Free_2}$ with two generators, the group algebras $C^*_{E_p}(G)$ are all different.
 Since $\C\rtimes_{p_{KLQ}}G=C^*_{E_p}(G)$ it follows from this that for such groups $G$
the KLQ-crossed-product functors associated to different $p$ are also different.
Note that for $p=2$ we just get the reduced crossed-product functor and for $p=\infty$ we get
the universal crossed-product functor as the associated KLQ-functors.

The corresponding BG-crossed-product functors
 $(B,\beta)\mapsto B\rtimes_{\beta,p_{BG}}G$ will also be different for $2\leq p\leq \infty$
 because we also have $\C\rtimes_{p_{BG}}G=C^*_{E_p}(G)$. But Example \ref{ex-BG} shows that these
 functors produce different results (and have different properties) for non-trivial coefficients.

\end{example}

\subsection{New functors from old functors.} \label{subsec-tensor}
It is shown in \cite{Baum-Guentner-Willett} that given any crossed-product functor $(B,\beta)\mapsto B\rtimes_{\beta,\mu}G$ and any unital $G$\nb-\cstar{}algebra $(D,\delta)$, one can use $D$ to construct a new functor $\rtimes_{\mu^\nu_D}$ by defining
$$B\rtimes_{\beta, \mu^\nu_D}G:=j_B\rtimes_\mu G(B\rtimes_{\beta,\mu}G)\subseteq (B\otimes_\nu D)\rtimes_{\beta\otimes\delta, \mu}G,$$
where $\otimes_\nu$  either denotes the minimal or the maximal tensor product, and $j_B:B\to B\otimes_{\nu}D, j_B(b)=b\otimes 1_D$ denotes the canonical imbedding.
A particularly interesting example is obtained in the case where $G$ is discrete by taking $D=\ell^\infty(G)$ and
$\rtimes_\mu=\rtimes_{\un}$, the universal crossed product. For  general locally compact groups this could be replaced by
the algebra $UC_b(G)$ of uniformly continuous bounded functions on $G$: we will discuss this further in Section \S\ref{sec:cub}.
For discrete $G$,  this functor will coincide with the reduced crossed-product functor if and only if $G$ is exact.
Later, in Corollary \ref{cor-tensor}, we shall extend the definition of $\rtimes_{\mu^\nu_D}$ to non-unital \cstar{}algebras $D$.

\subsection{A general source for (counter) examples}\label{subsec:counter}

The following general construction is perhaps a little unnatural.  It is, however, a very useful way to produce `counterexamples': examples of crossed-product functors that do not have good properties.

Let $\mathcal{S}$ be a collection of $G$\nb-algebras, and for any given $G$\nb-algebra $(A,\alpha)$, let $\Phi(A,\mathcal{S})$ denote the class of $G$\nb-equivariant \Star{}homomorphisms from $A$ to an element of $\mathcal{S}$.  Now define $A\rtimes_{\alpha,\mathcal{S}}G$ to be the completion of $C_c(G,A)$ for the norm
$$
\|f\|:=\max\Big\{\|f\|_{A\rtimes_{\alpha,\red}G}\quad,\quad\sup_{(\phi:A\to B)\in\Phi(A,\mathcal{S})}\|\phi\circ f\|_{B\rtimes_\un G}\Big\};
$$
this is a supremum of \cstar{}algebra norms, so a \cstar{}algebra norm, and it is clearly between the reduced and maximal completions of $C_c(G,A)$.

\begin{lemma}\label{oddfun}
The collection of completions $A\rtimes_{\mathcal{S}}G$ defines a crossed-product functor.
\end{lemma}
\begin{proof}
This follows immediately from the fact that if $\psi:A\to B$ is any $G$\nb-equivariant \Star{}homomorphism and $\phi:B\to C$ is an element of $\Phi(B,\mathcal{S})$, then $\psi\circ \phi$ is a member of $\Phi(A,\mathcal{S})$.
\end{proof}

Here is an example of the sort of `bad property' the functors above can have.  We will look at some similar cases in Example \ref{ex-ideal-property} and Example \ref{sepex} below.

\begin{example}\label{ex:matrix}
For a crossed-product functor $\rtimes_\mu$, write $C^*_\mu (G)$ for the exotic group algebra $\C\rtimes_{\id,\mu}G$, and for a \cstar{}algebra $A$ let $\Mat_2(A)$ denote the \cstar{}algebra of $2\times 2$ matrices over $A$.  If $\rtimes_\mu$ is either a KLQ- or a BG-crossed product then it is not difficult to check that
\begin{equation}\label{matrices}
\Mat_2(C^*_\mu(G))\cong \Mat_2(\C)\rtimes_{\id,\mu}G.
\end{equation}
 However, for any non-amenable $G$ there is a crossed product where this isomorphism fails!  Indeed, in the notation above take $\mathcal{S}=\{(\C,\id)\}$ and $\rtimes_\mu=\rtimes_\mathcal{S}$.  Then $\Mat_2(C^*_\mu(G))=\Mat_2(C^*(G))$ but $\Mat_2(\C)\rtimes_{\id,\mu}G=\Mat_2(C^*_\red(G))$ (the latter follows as $\Mat_2(\C)$ has no \Star{}homomorphisms to $\C$).  As  $G$ is amenable if and only if $C^*_\red(G)$ admits a non-zero finite dimensional representation, and as $C^*(G)$ always admits a non-zero one-dimensional representation, the isomorphism in line \eqref{matrices} is impossible for this crossed product and any non-amenable $G$.
\end{example}

\section{The ideal property and generalised homomorphisms}\label{sec:id}

In this section, we study two fundamental properties that a crossed product may or may not have: the ideal property, and functoriality under generalised homomorphisms.  In the next section we will use these ideas to investigate correspondence functors.

  \begin{definition}\label{def-gen-hom}
A crossed-product functor $(A,\alpha)\mapsto A\rtimes_{\alpha,\mu}G$ is said to be
functorial for {\em generalised homomorphisms}  if for any (possibly degenerate) $G$\nb-equivariant \Star{}homomorphism
 $\phi:A\to \M(B)$  there exists a \Star{}homomorphism $\phi\rtimes_{\mu}G:A\rtimes_{\mu}G\to \M(B\rtimes_{\mu}G)$
 which is given on the level of functions $f\in C_c(G,A)$ by $f\mapsto \phi\circ f$ in the
 sense that $\phi\rtimes_\mu G (f)g=(\phi\circ f)*g$ for all $g\in C_c(G,B)$.
\end{definition}

An alternative (and equivalent) definition of functoriality for generalised homomorphisms would be to ask that
for every $G$\nb-equivariant \Star{}homomorphism
 $\phi:A\to \M(B)$ and every  \Star{}representation $\pi\rtimes u$ of $B\rtimes_{\beta,\un}G$ which factors through
$B\rtimes_{\beta,\mu}G$, the representation $(\pi\circ \phi)\rtimes u$
factors through $A\rtimes_{\alpha,\mu}G$.

It follows from the definition that
if a crossed-product functor is functorial for generalised homomorphisms, then it
will automatically send nondegenerate $G$\nb-equivariant \Star{}homomorphisms
$\phi:A\to \M(B)$ to nondegenerate \Star{}homomorphisms $\phi\rtimes_\mu G:A\rtimes_{\mu}G
\to \M(B\rtimes_{\mu}G)$, since in this case $\phi\rtimes_\mu G(C_c(G,A))C_c(G,B)$ will
be inductive limit dense in $C_c(G,B)$. Note also that whenever the composition
of two generalised $G$\nb-equivariant morphisms $\phi:A\to \M(B)$ and $\psi:B\to \M(D)$
makes sense -- \eg, if the image of $\phi$ lies in $B$ or if $\psi$ is nondegenerate (in which
case it uniquely extends to $\M(B)$), we get
$$(\psi\circ \phi)\rtimes_\mu G=(\psi\rtimes_{\mu}G)\circ (\phi\rtimes_\mu G),$$
since both morphisms agree on the dense subalgebra $C_c(G,A)\subseteq A\rtimes_{\alpha,\mu}G$.
We shall see below that functoriality for generalised homomorphisms  is equivalent
to the following ideal property:

\begin{definition}\label{def-ideal}
A crossed-product functor $(A,\alpha)\mapsto A\rtimes_{\alpha,\mu}G$ satisfies the {\em ideal property},
if for every $G$-invariant closed ideal in a $G$-algebra $A$, the inclusion map $\iota:I\into A$ descends
to an injective *-homomorphism $\iota\rtimes G: I\rtimes_{\mu}G\into A\rtimes_{\mu}G$.
\end{definition}

\begin{lemma}\label{lem-ideal}
Let $(A,\alpha)\mapsto A\rtimes_{\alpha,\mu}G$ be a crossed-product functor. Then the following are
equivalent:
\begin{enumerate}
\item $(A,\alpha)\mapsto A\rtimes_{\alpha,\mu}G$ is functorial for generalised homomorphisms;
\item $(A,\alpha)\mapsto A\rtimes_{\alpha,\mu}G$ is functorial for nondegenerate generalised homomorphisms;
\item $(A,\alpha)\mapsto A\rtimes_{\alpha,\mu}G$ has the ideal property.
\end{enumerate}
\end{lemma}
\begin{proof}
(1) $\Rightarrow$ (2) is trivial.
(2)$\Rightarrow$ (3):
Let $A$ be a $G$-algebra and $I\sbe A$ a $G$\nb-invariant closed ideal. Let $\iota:I\hookrightarrow A$ denote the inclusion map.
 By functoriality, we get a \Star{}homomorphism $\iota\rtimes_\mu G\colon I\rtimes_\mu G\to A\rtimes_\mu G$.
 Let $\phi\colon A\to \M(I)$ be the canonical map, which is a nondegenerate $G$-equivariant \Star{}homomorphism. By (2), it induces a (nondegenerate) \Star{}homomorphism $\phi\rtimes_\mu G\colon A\rtimes_\mu G\to \M(I\rtimes_\mu G)$.
Notice that $\phi\circ \iota\colon I\to \M(I)$ is the canonical inclusion map, which is   a nondegenerate generalised homomorphism, so it induces a \Star{}homomorphism $(\phi\circ\iota)\rtimes_\mu G\colon I\rtimes_\mu G\to\M(I\rtimes_\mu G)$.
It is then easily checked on $C_c(G,I)$ that $(\phi\circ\iota)\rtimes_\mu G$ coincides with the
  the canonical embedding $I\rtimes_\mu G\into \M(I\rtimes_\mu G)$. In particular,
  $(\phi\circ\iota)\rtimes_\mu G=(\phi\rtimes_\mu G)\circ(\iota\rtimes_\mu G)$ is injective, which implies the injectivity of $\iota\rtimes_\mu G$, as desired.
\medskip
\\
(3) $\Rightarrow$ (1): Assume now that the $\mu$-crossed-product functor has the ideal property and
let $A$ and $B$ be $G$\nb-algebras and $\phi:A\to \M(B)$ a $G$\nb-equivariant \Star{}homomorphism.
Let $\M_c(B)=\{m\in \M(B): s\mapsto \beta_s(m)\;\text{is norm continuous}\}$ be the subalgebra of $\M(B)$
of $G$\nb-continuous elements. Then $\phi$ takes values in $\M_c(B)$ and by functoriality we
see that we get a \Star{}homomorphism $\phi\rtimes_{\mu}G:A\rtimes_{\mu}G\to \M_c(B)\rtimes_{\mu}G$.
Now by the ideal property, we also have a faithful inclusion $\iota\rtimes_{\mu}G: B\rtimes_{\mu}G
\to \M_c(B)\rtimes_{\mu}G$, hence we may regard $B\rtimes_{\mu}G$ as an ideal in $\M_c(B)\rtimes_{\mu}G$.
But this implies that there is a canonical \Star{}homomorphism $\Phi:\M_c(B)\rtimes_{\mu}G\to \M(B\rtimes_{\mu}G)$.
It is then straightforward to check that the composition
$\Phi\circ (\phi\rtimes_{\mu}G): A\rtimes_{\mu}G\to \M(B\rtimes_{\mu}G)$ is given on functions in
$C_c(G,A)$ by sending $f\in C_c(G,A)$ to $\phi\circ f\in C_c(G, \M(B))$ as required.
\end{proof}

\begin{remark}\label{rem:klq bg id}
It is easy to see directly that all KLQ- and all BG-crossed products have the ideal property.

Recall from \cite{Baum-Guentner-Willett} that a crossed-product functor $A\mapsto A\rtimes_\mu G$ is \emph{exact} if it preserves short exact sequences, that is, if every short exact sequence $0\to I\to A\to B\to 0$ of $G$\nb-\cstar{}algebras induces a short
exact sequence $0\to I\rtimes_\mu G\to A\rtimes_\mu G\to B\rtimes_\mu G\to 0$ of \cstar{}algebras. In particular, all exact functors satisfy the ideal property, but the ideal property alone is far from being enough for exactness since the reduced crossed-product functor always satisfies the ideal property. By definition, the \emph{group $G$ is exact} if $A\mapsto A\rtimes_\red G$ is exact.
\end{remark}

\begin{example}\label{ex-ideal-property}
Every non-amenable locally compact group admits a crossed-product functor which does not satisfy the ideal property.  To see this, we use the construction of Section \ref{subsec:counter}.  Let $\mathcal{S}$ be the collection of $G$\nb-algebras consisting of only $C_0([0,1))$ equipped with the trivial action $\id$.  For each $G$\nb-algebra $(A,\alpha)$ we define $A\rtimes_{\alpha,\mathcal{S}}G$ as the completion of $C_c(G,A)$
with respect to the norm
$$\|f\|_\mathcal{S}:= \max\{\|f\|_{\red}, \sup_{\phi}\|\phi\circ f\|_{\un}\},$$
where the supremum is taken over all $G$\nb-equivariant \Star{}homomorphisms $\phi:A\to C_0([0,1))$.  As explained in Section \ref{subsec:counter}, this is a crossed-product functor.

To see that it does not have the ideal property, consider the
short exact sequence of (trivial) $G$\nb-algebras
$$0\to C_0([0,1))\to C([0,1])\to \C\to 0.$$
Observe that $A\rtimes_\mu G= A\rtimes_\red G$ for every unital $G$-algebra $A$ since there is no non-zero
homomorphism $A\to \contz([0,1))$. In particular, $\cont([0,1])\rtimes_\mu G=\cont([0,1])\rtimes_\red G\cong \cont([0,1])\otimes C^*_\red(G)$. On the other hand, we obviously have $C_0([0,1))\rtimes_{\mu}G= C_0([0,1))\rtimes_{\un}G\cong \cont([0,1))\otimes C^*(G)$.
Hence, if $G$ is not amenable, the canonical map  $\iota\rtimes_{\mu}G: C_0([0,1))\rtimes_{\mu}G\to C([0,1])\rtimes_{\mu}G$ will not be injective.
\end{example}

For crossed-product functors which enjoy the ideal property we have the  important result:

\begin{lemma}\label{tensor-with-spaces}
Let $(A,\alpha)\mapsto A\rtimes_{\alpha,\mu}G$ be a crossed-product functor with the ideal property and let
$X$ be a  locally compact Hausdorff space.
Let $(j_A, j_X)$ denote the canonical inclusions of $A$ and  $C_0(X)$ into $\M(A\otimes C_0(X))$,
respectively and let $i_{A\otimes C_0(X)}:A\otimes C_0(X)\to \M\big((A\otimes C_0(X))\rtimes_{\alpha\otimes \id_X, \mu}G\big)$ be the canonical map.
Then
$$\Psi_X:=(j_A\rtimes_\mu G)\times (i_{A\otimes C_0(X)}\circ j_X) : A\rtimes_{\alpha,\mu}G\otimes C_0(X)
\stackrel{\cong}{\longrightarrow} (A\otimes C_0(X))\rtimes_{\alpha\otimes\id_X, \mu}G$$
is an isomorphism.
\end{lemma}

\begin{remark} We should remark that if $X$ is compact, the result holds for any crossed-product functor
$\rtimes_\mu$. The proof of this fact is given for the case $X=[0,1]$ in  \cite[Lemma 4.3]{Baum-Guentner-Willett}, but
the same arguments will work for arbitrary compact Hausdorff spaces  as well. We will use this fact in the proof below.
\end{remark}

\begin{proof}
Let $X_\infty=X\cup\{\infty\}$ denote the one-point compactification of $X$. Then the above remark implies that
the lemma is true for $X_\infty$. Consider
the  diagram
$$
\begin{CD}
(A\otimes C_0(X))\rtimes_{\mu}G @>>> (A\otimes C(X_\infty))\rtimes_{\mu}G\\
@V\Psi_X VV   @V\Psi_{X_\infty}V\cong V  \\
A\rtimes_{\alpha,\mu}G\otimes C_0(X) @>>> A\rtimes_{\alpha,\mu}G\otimes C(X_\infty) \\
\end{CD}
$$
in which the horizontal maps are induced by the canonical inclusion $\contz(X)\into \cont(X_\infty)$. One  checks on elements of the form $f\otimes g$ with $f\in C_0(X)$ and
$g\in C_c(G,A)$ that the diagram commutes and that $\Psi_X$ has dense image in $A\rtimes_{\alpha,\mu}G\otimes C_0(X)$.
By the ideal property, the upper horizontal map is injective, which then implies that $\Psi_X$ is injective as well.
\end{proof}

 \section{Correspondence functors}\label{sec-prop}

 In this section we continue our study of properties that a crossed-product functor may or may not have.
 One particularly important property we would like to have is that the functor extends to a functor
 on the $G$\nb-equivariant correspondence category, in which the objects are $G$\nb-algebras and
 the morphisms are $G$\nb-equivariant correspondences. The importance of this comes from the fact that
 functoriality for correspondences allows the construction of a descent in equivariant $KK$-theory
$$J_G^\mu:KK^G(A,B)\to KK(A\rtimes_\mu G, B\rtimes_\mu G)$$
for our exotic functor $\rtimes_\mu$; we discuss this in \S\ref{sec:descent}.  The main result of this section is Theorem \ref{thm-corr}, which gives several equivalent characterisations of functors that extend to the correspondence category.

We start by recalling some background about correspondences.

\begin{definition}
If $A$ and $B$ are \cstar{}algebras, a \emph{correspondence from $A$ to $B$} (or simply an $A-B$ correspondence) is a pair $(\E, \phi)$ consisting of a Hilbert $B$-module $\E$ together with a (possibly degenerate)
\Star{}homomorphism $\phi:A\to \Lb(\E)$.

If $(A,\alpha)$ and $(B,\beta)$ are $G$\nb-algebras, then a
$G$\nb-equivariant $A-B$ correspondence $(\E, \phi,\gamma)$ consists of an $A-B$ correspondence
$(\E,\phi)$ together with a strongly continuous action $\gamma:G\to \Aut(\E)$ such that
$$\braket{\gamma_s(\xi)}{\gamma_s(\eta)}_B=\beta_s(\braket{\xi}{\eta}_B),\quad
\gamma_s(\xi\cdot b)=\gamma_s(\xi)\beta_s(b)$$
$$\quad\text{and}\quad\gamma_s(\phi(a)\xi)=\phi(\alpha_s(a))\gamma_s(\xi)$$
for all $a\in A, b\in B, \xi,\eta\in \E$ and $s\in G$.
\end{definition}

\begin{definition}\label{def-equivalence}
We say that two $G$\nb-equivariant $A-B$ correspondences $(\E, \phi, \gamma)$ and $(\E', \phi', \gamma')$ are {\em isomorphic}
 if there exists an isomorphism $U:\E\to \E'$
which preserves all structures. We  say that $(\E, \phi, \gamma)$ and $(\E', \phi', \gamma')$ are {\em equivalent}
if there exists an isomorphism between $(\phi(A)\E,\phi, \gamma)$ and $(\phi'(A)\E, \phi',\gamma')$.
In particular, every correspondence  $(\E, \phi, \gamma)$ is equivalent to the nondegenerate correspondence
$(\phi(A)\E, \phi,\gamma)$.
\end{definition}

We should note at this point that  $\phi(A)\E=\{\phi(a)\xi: a\in A, \xi\in \E\}$ will be a closed $G$\nb-invariant
Hilbert $B$-submodule of $\E$ -- just apply Cohen's factorisation theorem to the closed submodule
$\cspn\phi(A)\E$ to see that this must coincide with $\phi(A)\E$.

Note that allowing general (\ie, possibly degenerate) correspondences makes it more straightforward
to view an arbitrary \Star{}homomorphism $\phi:A\to B$ as a correspondence: it will be represented
by the correspondence $(B,\phi)$ (or $(B,\phi,\beta)$ in the equivariant setting)
where $B$ is regarded as a Hilbert $B$-module in the canonical way. Of course, by our definition of equivalence,
$\phi:A\to B$ will also be represented by the correspondence $(\phi(A)B, \phi)$.
Composition of an $(A,\alpha)-(B,\beta)$ correspondence $(\E,\phi, \gamma)$ with a $(B,\beta)-(D,\delta)$ correspondence
$(\F,\psi,\tau)$ is given by the  internal tensor product construction
$(\E\otimes_\psi\F, \phi\otimes 1, \gamma\otimes\tau)$ (or just $(\E\otimes_\psi\F, \phi\otimes 1)$ in the non-equivariant case).
The following lemma is folklore and we omit a proof.

\begin{lemma}\label{lem-cor-composition}
The construction of internal tensor products is associative up to isomorphism of correspondences.
Moreover, we have canonical isomorphisms
$$\E\otimes_\psi\F\cong \E\otimes_{\psi}(\psi(B)\F)\quad\text{and}\quad (\phi(A)\E)\otimes_\psi\F\cong (\phi\otimes 1)(A)(\E\otimes_{\psi}\F).$$
\end{lemma}

The above lemma together with \cite[Theorem 2.8]{Echterhoff-Kaliszewski-Quigg-Raeburn:Categorical} allows the
following definition:

\begin{definition}\label{def-cor}
There is a unique category $\Cor(G)$, which we call the {\em $G$-equivariant correspondence category},
in which the objects are $G$\nb-\cstar{}algebras and
the morphisms between two objects $(A,\alpha)$ and $(B,\beta)$ are equivalence classes of
$(A,\alpha)-(B,\beta)$ correspondences $(\E,\phi,\gamma)$ with composition of morphisms
given by internal tensor products. We  write $\Cor$ for the correspondence category
of \cstar{}algebras without group actions (\ie, where $G=\{e\}$ is the trivial group).
\end{definition}

By our definition of equivalence of correspondences together with the above lemma, our category $\Cor(G)$ is isomorphic to the category $\mathcal A(G)$ defined in \cite[Theorem 2.8]{Echterhoff-Kaliszewski-Quigg-Raeburn:Categorical} in which the objects are \cstar{}algebras and the morphisms are isomorphism classes of {\em nondegenerate} $G$\nb-equivariant correspondences.
More precisely, the isomorphism $\Cor(G)\congto \mathcal{A}(G)$ is given by the assignment $[(\E,\phi,\gamma)]\to [(\phi(A)\E, \phi,\gamma)]$.

For every object $(A,\alpha)$ of $\Cor(G)$, the identity morphism on $(A,\alpha)$ is given by the
equivalence class of the correspondence $(A,\id_A,\alpha)$. It has been shown in
\cite[Lemma 2.4 and Remark 2.9]{Echterhoff-Kaliszewski-Quigg-Raeburn:Categorical} that the invertible morphisms in $\Cor(G)\cong\mathcal{A}(G)$ are precisely the equivalence classes of equivariant Morita-equivalence bimodules (which we simply call {\em equivalence bimodules} below).

\medskip

We now want to discuss under what conditions one can extend a crossed-product functor
$(A,\alpha)\to A\rtimes_{\alpha,\mu}G$, which is functorial for $G$\nb-equivariant \Star{}homomorphisms
to a functor from  the correspondence category $\Cor(G)$ to the correspondence category
$\Cor=\Cor(\{e\})$.  For this it is necessary
to say what the functor should do on morphisms.

For this recall that if
$(\E,\phi,\gamma)$ is a $G$\nb-equivariant correspondence
from $(A,\alpha)$ to $(B,\beta)$, then there is a canonical construction of a ``correspondence'' $(C_c(G,\E), \phi\rtimes_cG)$
on the level of continuous functions with compact supports where the  actions and inner products  are given by
\begin{align*}
\braket{x}{y}_{C_c(G,B)}(t)\defeq &\int_G\beta_{s^{-1}}(\braket{x(s)}{y(st)}_B)\dd{s}\\
(x \cdot \varphi)(t)\defeq &\int_G x(s)\beta_s(\varphi(s^{-1}t))\dd{s}
\end{align*}
for all $x,y\in C_c(G,\E)$, $\varphi\in C_c(G,B)$, and
\begin{equation}\label{eq-left-action}
(\phi\rtimes_cG(f)x)(t) \defeq \int_G \phi(f(s)) \gamma_s(x(s^{-1}t))\dd{s},
\end{equation}
for all $x\in C_c(G,\E)$ and $f\in C_c(G,A)$. It is well known (\eg, see \cite{Combes:Crossed_Morita, Kasparov:Novikov})
that this construction completes to a correspondence $(\E\rtimes_{\gamma,\un}G, \phi\rtimes_\un G)$ between the
maximal crossed products $A\rtimes_{\alpha,\un}G$ and $B\rtimes_{\beta,\un}G$. We simply  regard
$\braket{x}{y}_{C_c(G,B)}$ as an element of $B\rtimes_{\beta,\un}G$ and take completion with respect
to the corresponding norm $\|x\|_\un:=\sqrt{\| \braket{x}{x}_{C_c(G,B)}\|_\un}$ on $C_c(G,\E)$.

On the right hand side we can do a similar procedure in complete generality by regarding $\braket{x}{y}_{C_c(G,B)}$ as an element in
$B\rtimes_{\beta,\mu}G$ for any exotic crossed-product functor $\rtimes_\mu$: the completion $\E\rtimes_{\gamma,\mu}G$ of $C_c(G,\E)$ with respect to $\|x\|:=\sqrt{\| \braket{x}{x}_{C_c(G,B)}\|_\mu}$
will always be a Hilbert $B\rtimes_{\beta,\mu}G$-module. Moreover, we have a canonical isomorphism of Hilbert $B\rtimes_{\beta,\mu}G$-modules:
\begin{equation}\label{eq:CanonicalIsom}
(\E\otimes_{\gamma,\un}G)\otimes_{B\rtimes_{\beta,\un}G}(B\rtimes_{\beta,\mu}G)\congto \E\rtimes_{\gamma,\mu}G,
\end{equation}
sending $x\otimes b$ to $x\cdot b$, where the universal crossed product acts on the $\mu$-crossed product via the quotient map.

The problem  we need to address is the question whether the left action of $C_c(G,A)$ on $C_c(G,\E)$ given by
(\ref{eq-left-action}) will always extend to $A\rtimes_{\alpha,\mu}G$. We should also want that
$\K(\E)\rtimes_{\Ad\gamma,\mu}G\cong \K(\E\rtimes_{\gamma,\mu}G)$ if the left action
 exists.
 Indeed, we shall see below (Remark \ref{bgnocor}) that this fails for all BG-functors different from the universal crossed product,
 which is the BG-functor corresponding to the full group algebra $C^*(G)$. On the other hand,
it has been already shown in \cite[{Section~2}]{Buss-Echterhoff:Imprimitivity}
that all KLQ-functors extend to the correspondence category.
We shall give a short {alternative} argument for this at the end of this section.

We shall now introduce a few conditions that a functor may or may not have, which are related
to these questions.

\begin{definition}\label{def-properties}
Let $(A,\alpha)\to A\rtimes_{\alpha,\mu}G$ be a crossed-product functor. Then we say
\begin{enumerate}
\item the functor is {\em strongly Morita compatible} if for every $G$\nb-equivariant $(A,\alpha)-(B,\beta)$
equivalence bimodule $(\E,\gamma)$, the left action of $C_c(G, A)$ on $C_c(G,\E)$ extends to an action of
$A\rtimes_{\alpha,\mu}G$ on $\E\rtimes_{\gamma,\mu}G$ and makes $\E\rtimes_{\gamma,\mu}G$
an $A\rtimes_{\alpha,\mu}G-B\rtimes_{\beta,\mu}G$ equivalence bimodule.
\item The functor has the {\em projection property} (resp. {\em full projection property}) if for every $G$\nb-algebra $A$ and every
$G$\nb-invariant (resp. full) projection $p\in \M(A)$, the inclusion $\iota: pAp\into A$ descends to
a  faithful homomorphism $\iota\rtimes_{\mu}G: pAp\rtimes_{\alpha,\mu}G\to A\rtimes_{\alpha,\mu}G$.
\item The functor has the {\em hereditary-subalgebra property} if for every hereditary
$G$\nb-invariant subalgebra $B$ of $A$, the inclusion $\iota:B\into A$ descends to a  faithful map
$\iota\rtimes_{\mu}G: B\rtimes_{\alpha,\mu}G\to A\rtimes_{\alpha,\mu}G$.
\item The functor has the \emph{cp map property} if for any completely positive and $G$-equivariant map $\phi:A\to B$ of $G$-algebras, the map
$$
C_c(G,A)\to C_c(G,B),~~~f\mapsto \phi\circ f
$$
extends to a completely positive map from $A\rtimes_{\alpha,\mu} G$ to $B\rtimes_{\beta,\mu} G$.
\end{enumerate}
\end{definition}

\begin{remark}\label{rem-projection}
The (full) projection property can be reformulated as follows:
For every $G$\nb-algebra $A$ and for every $G$\nb-invariant (full) projection $p\in \M(A)$, we have $pAp\rtimes_{\alpha,\mu}G=\tilde{p} (A\rtimes_{\alpha,\mu}G)\tilde{p}$, where $\tilde{p}$ denotes the canonical image of $p$ in $\M(A\rtimes_{\alpha,\mu}G)$.
This means that the closure of $C_c(G, pAp)=\tilde{p} C_c(G,A)\tilde{p}$ inside
$A\rtimes_{\alpha,\mu}G$ coincides with $pAp\rtimes_{\alpha,\mu}G$.
\end{remark}

Recall that the ideal property has been introduced in Definition~\ref{def-ideal}.
Since ideals are hereditary subalgebras it is weaker than the
hereditary-subalgebra property. Recall also from Lemma \ref{lem-ideal} that the
ideal property is equivalent to the property that the functor extends to
generalised homomorphisms $\phi:A\to \M(B)$.

\begin{theorem}\label{thm-corr}
Let $(A,\alpha)\to A\rtimes_{\alpha,\mu}G$ be a crossed-product functor. Then the
following conditions are equivalent:
\begin{enumerate}
\item The functor extends  to a correspondence functor from $\Cor(G)$ to $\Cor$ sending
an equivalence class of a correspondence $(\E,\phi,\gamma)$ to the equivalence class
of the correspondence $(\E\rtimes_{\gamma, \mu}G, \phi\rtimes_{\mu}G)$.
\item  For every $G$\nb-equivariant (right)  Hilbert $(B,\beta)$-module $(\E,\gamma)$, the
left action of $\K(\E)$ on $\E$ descends to an isomorphism
$\K(\E)\rtimes_{\Ad\gamma,\mu}G\cong \K(\E\rtimes_{\gamma,\mu}G)$.
\item The functor is strongly Morita compatible and has the ideal property.
\item The functor has the hereditary subalgebra property.
\item  The functor has the projection property.
\item The functor has the cp map property.
\end{enumerate}
If the above equivalent conditions hold, we say that $(A,\alpha)\mapsto A\rtimes_{\alpha,\mu}G$ is a \emph{correspondence crossed-product functor}.
\end{theorem}

We prepare the proof of Theorem \ref{thm-corr} with some lemmas:

\begin{lemma}\label{lem-corr-to-ideal}
Suppose that the functor $(A,\alpha)\to A\rtimes_{\alpha,\mu}G$ extends to a correspondence
functor from $\Cor(G)$ to $\Cor$. Then it has the ideal property.
\end{lemma}

\begin{proof}
Let $\phi:A\to \M(B)$ be any $G$\nb-equivariant generalised homomorphism
from $A$ to $B$. Consider the $(A,\alpha)-(B,\beta)$ correspondence $(B,\phi,\beta)$.
It is clear that the $\mu$-crossed product of $B$ by $G$, if $B$ is viewed as a Hilbert $B$-module,
is equal to the \cstar{}algebra-crossed product $B\rtimes_\mu G$. Hence the left action of
$A\rtimes_\mu G$ on the module exists if and only if our functor is functorial for generalised
homomorphisms. By Lemma \ref{lem-ideal} this is equivalent to the ideal property.
\end{proof}

\begin{lemma}\label{lem-projection}
Let $(A,\alpha)\to A\rtimes_{\alpha,\mu}G$  be a crossed-product functor. Then the following are
equivalent:
\begin{enumerate}
\item $\rtimes_\mu$ has the projection property.
\item $\rtimes_\mu$ has the full projection property and the ideal property.
\end{enumerate}
\end{lemma}
\begin{proof} (2) $\Rightarrow$ (1) Let $p\in \M(A)$ be a $G$\nb-invariant projection. Then $I:=\overline{ApA}$ is a
$G$\nb-invariant ideal and $p$ can be viewed as a full $G$\nb-invariant projection in $\M(I)$. Then $pAp=pIp$ and the map
$\iota\rtimes_{\mu} G: pAp\rtimes_{\mu} G\to A\rtimes_\mu G$ is the composition of the faithful inclusions
$$pAp\rtimes_\mu G\stackrel{\iota\rtimes_\mu G}{\hookrightarrow} I\rtimes_\mu G \stackrel{\iota\rtimes_\mu G}{\hookrightarrow}
A\rtimes_\mu G.$$

For the converse direction, assume that $\rtimes_\mu$ has the projection property. Let $I$ be a $G$\nb-invariant ideal in the $G$\nb-algebra
$A$. Consider
the algebra $L:=\left(\begin{matrix} I& I\\ I&A\end{matrix}\right)$ with the obvious $G$\nb-action.
Then  $p=\left(\begin{smallmatrix} 1 & 0\\  0 & 0\end{smallmatrix}\right)$ and $q=\left(\begin{smallmatrix} 0 & 0\\  0 & 1\end{smallmatrix}\right)$
are opposite $G$\nb-invariant projections in $\M(L)$ and by the projection property we see that $p$ and $q$ map to opposite
projections $\tilde{p}$ and $\tilde{q}$  in $\M(L\rtimes_{\mu}G)$ such that
$I\rtimes_{\mu} G=\tilde{p}(L\rtimes_{\mu}G)\tilde{p}$ and $A\rtimes_{\mu}G=\tilde{q}(L\rtimes_{\mu}G)\tilde{q}$.
One easily checks that the right Hilbert  $A\rtimes_\mu G$-module $\tilde{p}(L\rtimes_{\mu}G)\tilde{q}$
is the completion $I\rtimes_{\tilde\mu}G$ of $C_c(G,I)= C_c(G, pLq)$ with respect to the norm given by the inclusion
$C_c(G,I)\hookrightarrow A\rtimes_\mu G$. On the other hand,
regarding $\tilde{p}(L\rtimes_{\mu}G)\tilde{q}$ as a left Hilbert  $I\rtimes_{\mu} G$-module in the canonical way,
it will be the completion of  $C_c(G,I)=C_c(G, pLq)$  with respect to the norm on $I\rtimes_{\mu}G$. Hence
$I\rtimes_{\tilde\mu}G=\tilde{p}(L\rtimes_{\mu}G)\tilde{q}=I\rtimes_{\mu}G$ and the result follows.
\end{proof}

The next lemma deals with various versions of Morita compatibility:

\begin{lemma}\label{lem-Morita}
Let $\rtimes_\mu$ be a crossed-product functor for $G$. Then the following are equivalent:
\begin{enumerate}
\item For each full $G$\nb-equivariant Hilbert $(B,\beta)$-module $(\E,\gamma)$ the left action
of $\K(\E)$ on $\E$ descends to an isomorphism $\K(\E)\rtimes_{\Ad\gamma,\mu}G\cong \K(\E\rtimes_{\gamma,\mu}G)$.
\item $\rtimes_\mu$ is strongly Morita compatible.
\item $\rtimes_{\mu}$ has the hereditary subalgebra property for {\em full} $G$\nb-invariant hereditary subalgebras
$B\subseteq A$ (\ie, $\cspn ABA=A$).
\item $\rtimes_{\mu}$ has the full projection property.
\end{enumerate}
\end{lemma}
\begin{proof}
We   prove (1) $\Leftrightarrow$ (2) $\Rightarrow$ (3)
$\Rightarrow$ (4) $\Rightarrow$ (2).
\medskip
\\
(1) $\Leftrightarrow$ (2) follows from the fact that if $(\E,\gamma)$ is a $G$\nb-equivariant $(A,\alpha)- (B,\beta)$ equivalence bimodule,
then the left action of $A$ on $\E$ induces a $G$\nb-equivariant isomorphism $(A,\alpha)\cong (\K(\E),\Ad\gamma)$ and, conversely,
that every full $G$\nb-equivariant Hilbert $B$-module is a $G$\nb-equivariant $\K(\E)-B$ equivalence bimodule.
\medskip
\\
(2) $\Rightarrow$ (3): Recall that a closed subalgebra $B\subseteq A$ is a {\em full} hereditary subalgebra, if $BAB\subseteq B$
and $A=\cspn ABA $. Then $BA$ will be a $B-A$-equivalence bimodule.
Let $B\rtimes_{\alpha,\tilde\mu}G:=\iota\rtimes_\mu G(B\rtimes_{\beta,\mu}G)$; this coincides with
the closure of $C_c(G,B)$ inside $A\rtimes_{\alpha,\mu}G$. Then a simple computation
on the level of functions with compact supports shows that $B\rtimes_{\alpha,\tilde\mu}G$ is a full hereditary
subalgebra of  $A\rtimes_{\alpha,\mu}G$. For functions $f, g\in C_c(G,B)*C_c(G,A)\subseteq C_c(G, BA)$
it is clear that the $A\rtimes_{\alpha,\mu}G$-valued inner products given by viewing
$f,g$ as elements of $(B\rtimes_{\alpha,\tilde\mu}G)(A\rtimes_{\alpha,\mu}G)$ or as elements of
the module $(BA)\rtimes_{\mu}G$ coincide. Thus we see that the inclusion
$C_c(G,B)*C_c(G,A)\hookrightarrow C_c(G, BA)$ extends to an isomorphism
$(B\rtimes_{\alpha,\tilde\mu}G)(A\rtimes_{\alpha,\mu}G)\cong (BA)\rtimes_{\mu}G$
of Hilbert $A\rtimes_{\alpha,\mu}G$-modules.
But this induces an isomorphism of the compact operators
$$B\rtimes_{\alpha,\tilde\mu}G\cong \K((B\rtimes_{\beta,\tilde\mu}G)(A\rtimes_{\alpha,\mu}G))
\cong \K((BA)\rtimes_{\mu}G)\cong B\rtimes_{\alpha,\mu}G,$$
which extends the identity map on $C_c(G,B)$, where
the last isomorphism follows from the strong Morita invariance of our functor.
\medskip
\\
(3) $\Rightarrow$ (4): follows from the fact that if $p\in \M(A)$ is a full $G$\nb-invariant projection, then
$pAp$ is a full $G$\nb-invariant  hereditary subalgebra of $A$.
\medskip
\\
(4) $\Rightarrow$ (2): Assume that $(A,\alpha)\to A\rtimes_{\alpha,\mu}G$ satisfies the full projection
property. Let $(\E,\gamma)$  be a $(A,\alpha)-(B,\beta)$ equivalence bimodule.
Then  we can form the linking algebra $L=\left(\begin{matrix} A & \E\\ \E^* & B\end{matrix}\right)$
equipped with the action $\sigma=\left(\begin{matrix} \alpha & \gamma\\ \gamma^* & \beta\end{matrix}\right)$.
Then  $p=\left(\begin{matrix} 1 & 0\\  0 & 0\end{matrix}\right)$ and $q=\left(\begin{matrix} 0 & 0\\  0 & 1\end{matrix}\right)$
are opposite full projections in $\M(L)$ and by the projection property we see that $p$ and $q$ map to opposite full projections
$\tilde{p}$ and $\tilde{q}$  in $\M(L\rtimes_{\sigma,\mu}G)$ such that
$\tilde{p}(L\rtimes_{\sigma,\mu}G)\tilde{q}$ is an equivalence bimodule between
$A\rtimes_{\alpha, \mu}G= \tilde{p}(L\rtimes_{\sigma,\mu}G)\tilde{p}$ and
$B\rtimes_{\beta,\mu}G= \tilde{q}(L\rtimes_{\sigma,\mu}G)\tilde{q}$.
One easily checks that $C_c(G,\E)=C_c(G,pLq)$  equipped with the norm coming from the $B\rtimes_{\beta,\mu}G$-valued
inner product embeds isometrically and densely into
$\tilde{p}(L\rtimes_{\sigma,\mu}G)\tilde{q}$, hence we get
$\E\rtimes_{\gamma,\mu}G=\tilde{q}(L\rtimes_{\sigma,\mu}G)\tilde{q}$.
Therefore $\E\rtimes_{\gamma,\mu}G$ becomes an $A\rtimes_{\alpha,\mu}G-B\rtimes_{\beta,\mu}G$ equivalence bimodule.
\end{proof}

We are now ready for the

\begin{proof}[Proof of Theorem \ref{thm-corr}]
We  prove (1) $\Rightarrow$ (2) $\Rightarrow$ (3) $\Rightarrow$ (4)  $\Rightarrow$ (5) $\Rightarrow$ (6) $\Rightarrow$ (5) $\Rightarrow$ (1).
\medskip
 \\
(1) $\Rightarrow$ (2) Assume that $(A,\alpha)\mapsto A\rtimes_{\alpha,\mu}G$ extends to a correspondence
functor. Then by Lemma \ref{lem-corr-to-ideal} we know that the functor has the ideal property.
Let $(\E,\gamma)$ be a Hilbert $(B,\beta)$-module. Let $I=\cspn\braket{\E}{\E}_B$. Then
the ideal property implies that viewing $(\E,\gamma)$ as a Hilbert $(I,\beta)$-module or as a Hilbert $(B,\beta)$-module
does not change the induced norm $\|\cdot\|_\mu$ on $C_c(G,\E)$, and hence in both cases we get
the same completion $\E\rtimes_{\gamma,\mu}G$. Thus we may assume without loss of generality
that $\E$ is a full Hilbert $B$-module.  But then the equivalence class of $(\E,\gamma)$ will be an isomorphism
from $(\K(\E),\Ad\gamma)$ to $(B,\beta)$ in the
equivariant correspondence category $\Cor(G)$, which by the properties of a functor must be sent
to an isomorphism from $\K(\E)\rtimes_{\Ad\gamma,\mu}G$ to $B\rtimes_{\beta,\mu}G$ in $\Cor$,
which are equivalence classes of equivalence bimodules.
Since the left action of $\K(\E)\rtimes_{\Ad\gamma,\mu}G$ on $\E\rtimes_{\gamma, \mu} G$ is obviously
nondegenerate, it follows that $\E\rtimes_{\gamma,\mu}G$ is a $\K(\E)\rtimes_{\Ad\gamma,\mu}G-B\rtimes_{\beta,\mu}G$
equivalence bimodule.  This implies the desired isomorphism.
\medskip
\\
(2) $\Rightarrow$ (3) It follows clearly from (2) applied to full Hilbert modules
that the functor $(A,\alpha)\mapsto A\rtimes_{\alpha,\mu}G$
is strongly Morita compatible. To see that it  has the ideal property,  let $I$ be a $G$\nb-invariant closed ideal in
$A$ for some $G$\nb-algebra $(A,\alpha)$. Regard $(I,\alpha)$ as a $G$\nb-invariant Hilbert $(A,\alpha)$-module
in the canonical way. To avoid confusion, we write $I_A$ for the Hilbert $A$-module $I$.
We then have $(I,\alpha)=(\K(I_A),\Ad\gamma)$. Now if the descent
$\iota\rtimes_{\mu}G:I\rtimes_\mu G\to A\rtimes_{\mu}G$ is not injective, it induces a new norm $\|\cdot\|_{\tilde\mu}$
 on $C_c(G,I)$ which is (at least on some elements of $C_c(G,I)$) strictly smaller than the given
 norm $\|\cdot\|_\mu$. It is then easily checked
 that
 $I_A\rtimes_{\mu}G=I\rtimes_{\tilde\mu}G$. This also implies $\K(I_A\rtimes_{\mu}G)=I\rtimes_{\tilde\mu}G$
 which will be a proper quotient of $\K(I_A)\rtimes_{\mu}G=I\rtimes_{\mu}G$, which contradicts (2).
\medskip
\\
(3) $\Rightarrow$ (4) Let $B\subseteq A$ be a $G$\nb-invariant hereditary subalgebra of $A$.
Then $BA$ will be a $B-ABA$-equivalence bimodule, where $ABA$ denotes the closed ideal of $A$ generated by $B$.
Since, by assumption,  the $\mu$-crossed-product functor satisfies the ideal property, we may assume without loss
of generality that $B$ is {\em full} in the sense that $ABA=A$. But then the result follows from
part (2) $\Leftrightarrow$ (4) of Lemma \ref{lem-Morita}.
\medskip
\\
(4) $\Rightarrow$ (5) This follows from the fact that corners are hereditary subalgebras.
\medskip
\\
(5) $\Rightarrow$ (6) Let $\phi:A\to B$ be a $G$\nb-equivariant completely positive map between $G$\nb-algebras; we may assume $\phi$ is non-zero.  Multiplying $\phi$ by a suitable positive constant, we may assume that $\phi$ is contractive.  It then follows from \cite[Proposition 2.2.1]{Brown-Ozawa} that the (unital) map defined on unitisations by
$$
\widetilde{\phi}:\widetilde{A}\to \widetilde{B},~~~\widetilde{\phi}(a+z1_A)=\phi(a)+z1_B
$$
is also completely positive and contractive; it is moreover clearly equivariant for the extended actions of $G$.

Now, assume that $\widetilde{B}$ is represented faithfully, non-degenerately, and covariantly on a Hilbert space $\H$.  Identify $\widetilde{B}$ with its image in $\Lb(\H)$, and write $u$ for the given representation of $G$ on $\H$.   Define a bilinear form on the algebraic tensor product $\widetilde{A}\odot \H$ by
$$
\Big\langle \sum_{i=1}^n a_i\otimes \xi_i,\sum_{j=1}^m b_j\otimes \eta_j\Big\rangle \defeq\sum_{i,j}\langle \xi_i,\widetilde{\phi}(a_i^*b_j)\eta_j\rangle.
$$
As $\widetilde{\phi}$ is completely positive, this form is positive semi-definite, so separation and completion gives a Hilbert space $\H'$.  If $\alpha$ is the action of $G$ on $\widetilde{A}$, it follows from the fact that $\phi$ is equivariant that the formula
$$
v_g:\sum_{i=1}^n a_i\otimes \xi_i\mapsto \sum_{i=1}^n \alpha_g(a_i)\otimes u_g\xi_i
$$
defines a unitary action of $G$ on $\H'$.  Moreover, the action of $\widetilde{A}$ on $\widetilde{A}\odot \H$ defined by
$$
\pi(a)\cdot\Big(\sum_{i=1}^n a_i\otimes \xi_i\Big)\defeq\sum_{i=1}^n aa_i\otimes \xi_i
$$
gives rise to a bounded representation of $\widetilde{A}$ on $\H'$ by the Cauchy-Schwarz inequality for $\phi$.  We also write $\pi$ for the corresponding representation of $\widetilde{A}$ on $\H'$ and note that this is moreover covariant for $v$.

The formula
$$
V:\H\to \H',~~~\xi\mapsto 1_{\widetilde{A}}\otimes \xi
$$
is easily seen to define an equivariant isometry with adjoint given by
$$
V^*:\H'\to \H,~~~\sum_{i=1}^n a_i\otimes \xi_i\mapsto \sum_{i=1}^n \widetilde{\phi}(a_i)\xi_i.
$$
Using these formulas, one sees that for all $a\in \widetilde{A}$ we have
\begin{equation}\label{comp}
\widetilde{\phi}(a)=V^*\pi(a)V
\end{equation}
as elements of $\Lb(\H)$.

To complete the proof, let $C$ be the $C^*$-subalgebra of $\Lb(\H')$ generated by $\pi(\widetilde{A})$, $VV^*\pi(\widetilde{A})$, $\pi(\widetilde{A})VV^*$ and $V\widetilde{B}V^*$.  Note that the action $\Ad v$ of $G$ on $C$ induced by the unitary representation $v$ is norm continuous, so $C$ is a $G$-algebra with the induced action.  Moreover, the \Star{}homomorphism
$$
\pi:\widetilde{A}\to \Lb(\H')
$$
is equivariant and has image in $C$, so gives rise to a \Star{}homomorphism on crossed products
\begin{equation}\label{indbypi}
\pi\rtimes_\mu G:\widetilde{A}\rtimes_\mu G \to C\rtimes_\mu G
\end{equation}
by functoriality of $\rtimes_\mu$.  Note that by construction of $C$, $p\defeq VV^*$ is in the multiplier algebra of $C$ and $pCp=V\widetilde{B}V^*$, so we have an equivariant \Star{}isomorphism
\begin{equation}\label{waskappa}
\kappa:pCp\to B,~~~x\mapsto V^*xV.
\end{equation}
Now, if we denote by $\tilde{p}$ the element of $\M(C\rtimes_\mu G)$ induced by $p$, then by the projection property we have a \Star{}isomorphism
\begin{equation}\label{projuse}
\tilde{p}(C\rtimes_\mu G)\tilde{p}\cong pCp\rtimes_\mu G
\end{equation}
and so a \Star{}isomorphism
\begin{equation}\label{indbyp}
\psi:\tilde{p}(C\rtimes_\mu G)\tilde{p}\to \widetilde{B}\rtimes_\mu G
\end{equation}
defined as the map on crossed products induced by the composition of the isomorphism in line \eqref{projuse} and the {\Star}isomorphism on crossed products induced by the isomorphism in line \eqref{waskappa}.

Now, consider the composition of maps
$$
A\rtimes_\mu G \to \widetilde{A}\rtimes_\mu G \stackrel{\pi\rtimes_\mu G}{\to} C\rtimes_\mu G \stackrel{x\mapsto \tilde{p}x\tilde{p}}{\longrightarrow} \tilde{p}(C\rtimes_\mu G)\tilde{p}\stackrel{\psi}{\to} \widetilde{B}\rtimes_\mu G
$$
where the first map is the \Star{}homomorphism on crossed products induced from the equivariant inclusion $A\mapsto \widetilde{A}$, the second is as in line \eqref{indbypi}, the third is compression by $\tilde{p}$, and the fourth is $\psi$ as in line \eqref{indbyp}.  Each map appearing in the sequence is completely positive, whence the composition is completely positive.  Checking the image of $C_c(G,A)$ using the formula in line \eqref{comp} shows that the map agrees with the map $f\mapsto \phi\circ f$ from $A\rtimes_\mu G$ to $D$, where $D$ is the image of $B\rtimes_\mu G$ under the canonical map of this $C^*$-algebra into $\widetilde{B}\rtimes_\mu G$.  However, as $\rtimes_\mu$ has the projection property, it also has the ideal property, and so $D$ is just a copy of $B\rtimes_\mu G$.  This completes the proof.
\medskip
\\
(6) $\Rightarrow$ (5) Let $\rtimes_\mu$ be a cp-functorial crossed product for $G$.  Let $p$ be a $G$-invariant projection in the multiplier algebra of a $G$-algebra $A$, and consider the $G$-equivariant maps
$$
pAp \to A \to pAp,
$$
where the first map is the canonical inclusion, and the second map is compression by $p$; note that the first map is a \Star{}homomorphism, the second is completely positive, and the composition of the two is the identity on $pAp$.  Functoriality then gives \Star{}homomorphisms on crossed products
$$
pAp\rtimes_\mu G \to A\rtimes_\mu G \to pAp\rtimes_\mu G
$$
whose composition is the identity; in particular, the first \Star{}homomorphism is injective, which is the projection property.
\medskip
\\
(5) $\Rightarrow$ (1) Assume that $(A,\alpha)\to A\rtimes_{\alpha,\mu}G$ satisfies the projection
property. We first show that this implies (2), \ie, for any Hilbert $(B,\beta)$-module $(\E,\gamma)$  we get
$\K(\E)\rtimes_{\Ad\gamma,\mu}G\cong \K(\E\rtimes_{\gamma,\mu}G)$ in the canonical way.
By Lemma \ref{lem-projection} we know that $\rtimes_\mu$  satisfies the ideal property.
Hence we may assume without loss of generality that $\E$ is a full Hilbert $B$-module.
The result then follows from (4) $\Leftrightarrow$ (1) in Lemma~\ref{lem-Morita}.

To see that $(A,\alpha)\to A\rtimes_{\alpha,\mu}G$ extends to a correspondence functor we can now use
the fact that by the ideal property  we get functoriality for generalised homomorphisms. Hence
for any $G$\nb-equivariant correspondence $(\E,\phi,\gamma)$ from $(A,\alpha)$ to $(B,\beta)$,
the homomorphism $\phi:A\to\Lb(\E)\cong \M(\K(\E))$ descends to a \Star{}homomorphism
$$A\rtimes_\mu G\to \M(\K(\E)\rtimes_{\Ad\gamma,\mu}G)\cong \Lb(\E\rtimes_{\gamma,\mu}G),$$
and therefore provides a correspondence $(\E\rtimes_{\gamma,\mu}G, \phi\rtimes_\mu G)$
from $A\rtimes_{\alpha,\mu}G$ to $B\rtimes_{\beta,\mu}G$. We already saw above that it
preserves isomorphisms (\ie, equivalence bimodules) and compatibility with compositions
is proved on the level of functions with compact supports as in \cite[Chapter 3]{Echterhoff-Kaliszewski-Quigg-Raeburn:Categorical}.
\end{proof}

We now give a few applications of Theorem \ref{thm-corr}. It is already shown in \cite[Corollary 2.9]{Buss-Echterhoff:Imprimitivity}
that all KLQ-functors extend to correspondences.
We now give a very short argument for this:

\begin{corollary}\label{cor-KLQproj}
Every KLQ-functor has the projection property and therefore extends to a functor from $\Cor(G)$ to $\Cor$.
\end{corollary}
\begin{proof} Let $v:C^*(G)\to C_v^*(G)$ be a quotient map which absorbs the universal representation $u_G$.
Then the corollary follows from commutativity of the diagram
$$
\begin{CD}
pAp\rtimes_{v_{KLQ}} G  @>(\id_{pAp\rtimes G}\otimes v)\circ \widehat{\alpha} >> \M(pAp\rtimes_{\un} G\otimes C_v^*(G))\\
@V\iota\rtimes_\mu G VV @VV\iota\rtimes_\un G \otimes \id_{C_v^*(G)} V\\
A\rtimes_{v_{KLQ}} G  @>(\id_{A\rtimes G}\otimes v)\circ \widehat{\alpha} >> \M(A\rtimes_{\un} G\otimes C_v^*(G))\\
\end{CD}
$$
together with the fact that the projection property holds for the universal crossed-product functor.
\end{proof}

\begin{remark}\label{bgnocor}
On the other hand, the Brown-Guentner functors of Section \ref{sec:BG} are almost never correspondence
functors. In case of the BG-functor corresponding to the regular representation $\lambda_G$ this is
well illustrated in Example \ref{ex-BG} above.
In fact, it has been already observed in \cite[Lemma A.6]{Baum-Guentner-Willett} that a BG-crossed product is
Morita compatible in their sense (see Remark \ref{rem-ext2} and \S\ref{sec:minimal} for more discussion of this property) if and only if it coincides with the universal
crossed product. Since every correspondence functor is also Morita compatible (see Corollary \ref{cor-ext} and Remark \ref{rem-ext2} below), we
also see that the universal crossed-product functors are the only BG-functors which are also correspondence functors.
\end{remark}

The next corollary shows that if we derive a functor from a correspondence functor via tensoring with a  $G$\nb-algebra
$(D,\delta)$, then the new functor is also a correspondence functor. At the same  time we
 extend the tensor-product construction
of \S \ref{subsec-tensor} to the case of non-unital $G$-algebras:

\begin{corollary}\label{cor-tensor} Suppose that $(A,\alpha)\mapsto A\rtimes_{\alpha,\mu}G$ is a crossed-product
functor which satisfies the ideal property and let $(D,\delta)$ be any $G$\nb-algebra. Then there
is a  crossed-product functor $\rtimes_{\mu_D^\nu}$, also satisfying the ideal property, defined as
$$A\rtimes_{\alpha,\mu_D^\nu}G:=j_A\rtimes_\mu G(A\rtimes_{\alpha,\mu}G)\subseteq
\M((A\otimes_{\nu}D)\rtimes_{\alpha\otimes \delta, \mu}G),$$
where $\otimes_\nu$ is either the minimal or the maximal tensor product of $A$ with $D$
and $j_A:A\to \M(A\otimes_\nu D)$ denotes the canonical inclusion. Moreover,  if $\rtimes_\mu$ is a correspondence
functor, then so is $\rtimes_{\mu_D^\nu}$.
\end{corollary}
\begin{proof} Let $\M_D(A\otimes_\nu D)$ denote the closed subalgebra of $\M(A\otimes_\mu D)$
which consists of all elements $m$ such that  $m(1\otimes_\nu D)\subseteq A\otimes_\nu D$.
Then for any (possibly degenerate) \Star{}homomorphism $\phi:A\to \M(B)$
the  \Star{}homomorphism $\phi\otimes \id_D: A\otimes_{\nu} D\to \M(B\otimes_\nu D)$
extends uniquely to $\M_D(A\otimes_\nu D)$ (\eg, see \cite[Proposition~A.6]{Echterhoff-Kaliszewski-Quigg-Raeburn:Categorical}).
Now let $\M_{D,c}(A\otimes_\nu D)$ denote the set
of $G$\nb-continuous elements in $\M_D(A\otimes_\nu D)$.
Then it follows from the  ideal property that $(A\otimes_\nu D)\rtimes_\mu G$ is an essential
ideal in $\M_{D,c}(A\otimes_\nu D)\rtimes_{\mu}G$, which implies that there is a canonical faithful
inclusion of $\M_{D,c}(A\otimes_\nu D)\rtimes_{\mu}G$ into $\M((A\otimes_\nu D)\rtimes_{\mu}G)$.
Hence, we can alternatively define
$A\rtimes_{\alpha,\mu_D^\nu}G$ as the image of $j_A\rtimes_\mu G$ inside $\M_{D,c}(A\otimes_\nu D)\rtimes_{\mu}G$.
A similar argument applies if we replace $\M_{D,c}(A\otimes_\nu D)$ by the \cstar{}algebra $\M_c(A\otimes_\nu D)$
of $G$\nb-continuous elements in $\M(A\otimes_\nu D)$.

Now let  $\phi:A\to \M(B)$ be a $G$\nb-equivariant \Star{}homomorphism. Then the diagram
$$
\begin{CD}
A\rtimes_{\mu}G @>\quad\quad\quad\phi\rtimes_{\mu}G\quad\quad\quad >> \M(B\rtimes_\mu G)\\
@V j_A\rtimes_{\mu} G VV       @VV j_B\rtimes_{\mu} G V\\
\M_{D,c}(A\otimes_\nu D)\rtimes_\mu G @>> (\phi\otimes_\nu \id_D)\rtimes_\mu G > \M_c(B\otimes_\nu D)\rtimes_\mu G
\subset \M((B\otimes_\nu D)\rtimes_\mu G)
\end{CD}
$$
commutes,  since $(j_B\otimes_\nu\id_D)\circ \phi=(\phi\otimes _\nu \id_D)\circ j_A$. It follows that
the bottom arrow restricts to a well-defined \Star{}homomorphism from
$A\rtimes_{\mu_D^\nu}G$  into $\M(B\rtimes_{\mu_D^\nu}G)$. It follows that $\rtimes_{\mu_D^\nu}$
is a functor for generalised homomorphisms.
By Lemma \ref{lem-ideal} this also proves that the functor $\rtimes_{\mu_D^\nu}$ has the ideal property.

Assume now that $\rtimes_\mu$ is a correspondence functor. To see that this is then also true
for $\rtimes_{\mu_D^\nu}$, let $p\in \M(A)$ be any $G$\nb-invariant
projection and let $\iota:pAp\to A$ be the inclusion map. Consider the commutative diagram
$$
\begin{CD}
pAp\rtimes_{\mu}G @>\quad\quad\quad\iota \rtimes_{\mu}G\quad\quad\quad >> A\rtimes_\mu G\\
@V j_{pAp}\rtimes_{\mu} G VV       @VV j_A\rtimes_{\mu} G V\\
\M_{D,c}(pAp\otimes_\nu D)\rtimes_\mu G @>> (\iota \otimes_\nu \id_D)\rtimes_\mu G > \M_c(A\otimes_\nu D)\rtimes_\mu G
\subset \M((A\otimes_\nu D)\rtimes_\mu G)
\end{CD}
$$
Since $\rtimes_\mu$ satisfies the projection property, we have injectivity of
$$(\iota\otimes\id_D)\rtimes_\mu G: (pAp\otimes_\nu D)\rtimes_{\mu}G\hookrightarrow (A\otimes_\nu D)\rtimes_{\mu}G$$
which then extends to a unique  injective \Star{}homomorphism
$\M_{D,c}(pAp\otimes_\nu D)\rtimes_\mu G\to \M((A\otimes_\nu D)\rtimes_\mu G)$.
Thus the lower horizontal map in the diagram is injective. But then the commutativity of the diagram
implies injectivity of
$\iota\rtimes_{\mu_D^\nu}G: pAp\rtimes_{\mu_G^\nu}G\to A\rtimes_{\mu_G^\nu}G$, as desired.
\end{proof}

\begin{remark}\label{rem-exact-D} In \cite[Lemma 5.4]{Baum-Guentner-Willett} it is shown that if $\rtimes_\mu$ is exact and $D$ is unital, then $\rtimes_{\mu_D^{\max}}$ is also exact.  We should note that for non-unital $D$ this need not be true. To see an example, let $G$ be any non-exact group,
 let $D=C_0(G)$ equipped with the translation action and let $\rtimes_\mu=\rtimes_\un$. Then
  $$(A\otimes C_0(G))\rtimes_{\un}G\cong (A\otimes C_0(G))\rtimes_{\red} G\cong A\otimes \K(L^2(G))$$
  and the map $j_A\rtimes_\un G: A\rtimes_\un G\to \M(A\otimes C_0(G))\rtimes_{\un}G$ factors
  through the (faithful) map $j_A\rtimes_\red G: A\rtimes_\red G\to \M(A\otimes C_0(G))\rtimes_{\un}G)$.
Therefore we get $\rtimes_{\un_{C_0(G)}^{\max}}=\rtimes_\red$, which by the choice of $G$ is not
exact.
\end{remark}

We conclude this section by showing that a correspondence crossed-product functor allows a
nice description of $\K(\E\rtimes_{\mu}G, \F\rtimes_{\mu}G)$ in which $(\E,\gamma)$ and $(\F,\tau)$
are two $G$-equivariant Hilbert $(B,\beta)$-modules.  This will be useful for our discussion of $KK$-theory in the next section.  For this we first observe that
$\K(\E,\F)$ can be regarded as a  $\K(\F)-\K(\E)$ correspondence with respect
to the canonical left action of $\K(\F)$ on $\K(\E,\F)$ given by composition of
operators and with the $\K(\E)$-valued inner product given by
$$\braket{T}{S}_{\K(\E)}=T^*\circ S.$$
The $G$-actions $\gamma$ and $\nu$ induce an action $\Ad(\gamma,\tau)$ of $G$ on
$\K(\E,\F)$ by
$$\Ad(\gamma,\tau)_s(T)=\tau_{s^{-1}}T\gamma_s.$$
Then $(\K(\E,\F), \Ad(\gamma,\tau))$ becomes a $G$-equivariant  $(\K(\E),\Ad\gamma)-(\K(\F),\Ad\tau)$ correspondence.
If $\rtimes_{\mu}$ is any crossed-product functor for $G$,   we may consider the crossed product
$\K(\E,\F)\rtimes_{\Ad(\gamma,\tau),\mu}G$ as a completion of $C_c(G,\K(\E,\F))$ as
described at the beginning of this section.

\begin{lemma}\label{lem-iso}
Suppose that $\rtimes_\mu$ is a correspondence crossed-product functor. Then there is a canonical
isomorphism
$$\K(\E,\F)\rtimes_{\Ad(\gamma,\tau),\mu}G\cong \K(\E\rtimes_{\gamma,\mu}G, \F\rtimes_{\tau,\mu}G)$$
which sends a function $f\in C_c(G, \K(\E,\F))\subseteq \K(\E,\F)\rtimes_{\mu}G$ to the operator
$T_f\in \K(\E\rtimes_{\mu}G, \F\rtimes_{\mu}G)$ given on the dense submodule $C_c(G,\E)\subseteq \E\rtimes_{\gamma,\mu}G$ by the convolution formula
$$(T_f\xi)(s)=\int_G f(t)\gamma_s(\xi(t^{-1}s))\,dt.$$
\end{lemma}
\begin{proof} Let $\E\oplus \F$ denote the direct sum of the Hilbert B-modules $\E$ and $\F$.  Let $p,q\in \Lb(\E\oplus \F)$
denote the orthogonal projections to $\E$ and $\F$.
This gives a canonical  decomposition
$$\K(\E\oplus \F)\cong \left(\begin{matrix} \K(\E) & \K(\F,\E)\\ \K(\E,\F)& \K(\F)\end{matrix}\right)$$
by identifying $\K(\E)\cong p\K(\E\oplus \F)p$, $\K(\F,\E)=p\K(\E\oplus \F)q$ and so on.
The projections $p$ and $q$ are  $G$\nb-invariant and therefore map to
opposite $G$\nb-invariant projections $\tilde p$ and $\tilde q$ in
$\M(\K(\E\oplus \F)\rtimes_{\Ad({\gamma\oplus\tau}),\mu}G)$ under the canonical map.
Taking crossed products it follows from the properties of a correspondence functor shown in
Theorem \ref{thm-corr} that we get a decomposition
\begin{align*}
\K\big((\E\oplus \F)\rtimes_{\gamma\oplus\tau,\mu}G\big)&\cong
\K(\E\oplus \F)\rtimes_{\Ad({\gamma\oplus\tau}),{\mu}}G\\
&=\left(\begin{matrix} \K(\E)\rtimes_{\Ad\gamma,\mu}G & \tilde{p}\big(\K(\F\oplus\E)\rtimes_\mu G\big)\tilde{q}\\
\tilde{q}\big(\K(\F\oplus\E)\rtimes_\mu G\big)\tilde{p}& \K(\F)\rtimes_{\Ad\tau,\mu}G\end{matrix}\right)
\end{align*}
Notice that we have the identity
$$(\E\oplus\F)\rtimes_{\gamma\oplus\tau,\mu}G
\cong \E\rtimes_{\gamma,\mu}G\oplus  \F\rtimes_{\tau,\mu }G,$$
which follows directly from the definition of the inner products on $C_c(G,\E\oplus \F)$.
It implies the  decomposition
\begin{align*}
 \K((\E\oplus F)\rtimes_{\gamma\oplus \tau,\mu}G)&\cong \K(\E\rtimes_{\gamma, \mu} G \oplus \F\rtimes_{\tau,\mu} G)\\
&\cong
\left(\begin{matrix} \K(\E\rtimes_{\gamma, \mu}G) & \K(\E\rtimes_{\gamma,\mu} G, \F\rtimes_{\tau,\mu}G)\\
\K(\F\rtimes_{\tau,\mu} G, \E\rtimes_{\gamma, \mu}G)& \K(\F\rtimes_{\tau, \mu}G)\end{matrix}\right)\\
&\cong
\left(\begin{matrix} \K(\E)\rtimes_{\Ad\gamma, \mu}G & \K(\E\rtimes_{\gamma,\mu} G, \F\rtimes_{\tau,\mu}G)\\
\K(\F\rtimes_{\tau,\mu} G, \E\rtimes_{\gamma, \mu}G)& \K(\F)\rtimes_{\Ad\tau, \mu}G\end{matrix}\right).
\end{align*}
Comparing both isomorphisms on functions in $C_c(G, \K(\E\oplus \F))$, we see that they agree
and that the isomorphism of the upper  right corner is given on the level of functions in
$C_c(G, \K(\E,\F))=\tilde{p}C_c(G, \K(\E\oplus\F))\tilde{q}$ as in the statement.
\end{proof}

\section{Duality for correspondence functors} \label{sec:duality}
We saw in Corollary \ref{cor-KLQproj} that all KLQ-functors are correspondence functors.
Recall from Section \ref{sec:KLQ} that the group algebras $C_v^*(G)\cong \C\rtimes_{v_{KLQ}}G$
corresponding to KLQ-functors are precisely those group algebras which correspond to
weak equivalence classes of
unitary representations $v:G\to \U(\H_v)$ which absorb the universal representation
$u_G: G\to \U(C^*(G))$ of $G$ in the sense that $v$ is weakly equivalent to $v\otimes u_G$.
This just means that the comultiplication
$$\delta_G: C^*(G)\to \M(C^*(G)\otimes C^*(G))$$
on $C^*(G)$ which is given as the integrated form of $u_G\otimes u_G$ factors through a
coaction
$$\delta_v: C_v^*(G)\to \M(C_v^*(G)\otimes C^*(G))$$
of $C^*(G)$ on $C_v^*(G)$. As discussed in Section \ref{sec:KLQ}, it is shown in \cite{Kaliszewski-Landstad-Quigg:Exotic}
that these group algebras are in one-to-one correspondence to the translation invariant weak-* closed
ideals $E\subseteq B(G)$ of the Fourier-Stieltjes algebra $B(G)$.  Moreover, it is also shown in \cite{Kaliszewski-Landstad-Quigg:Exotic}
that for any system $(A,G,\alpha)$ the dual coaction
$$\widehat{\alpha}=(i_A^u\otimes \id_G)\rtimes (i_G^u\otimes u_G):A\rtimes_{\alpha,u}G
\to \M(A\rtimes_{\alpha,u}G\otimes C^*(G))$$
of $C^*(G)$ on the full crossed product $A\rtimes_{\alpha,u}G$ factors through a dual coaction
$$\widehat{\alpha}_{v_{KLQ}}:A\rtimes_{\alpha,v_{KLQ}}G
\to \M(A\rtimes_{\alpha,v_{KLQ}}G\otimes C^*(G)).$$
We refer to \cite[Appendix A]{Echterhoff-Kaliszewski-Quigg-Raeburn:Categorical} for a survey on duality theory for actions and coactions of groups.

We shall now show that the existence of a dual coaction holds true for every
correspondence functor $\rtimes_\mu$, so that versions of Imai-Takai and Katayama
duality will hold for the corresponding double (or triple) crossed products.
In particular, this implies that the group algebra $C_\mu^*(G)=\C\rtimes_\mu G$ of a
correspondence functor must correspond to a translation invariant
weak-* closed ideal $E$ in $B(G)$.

We start with a result about exterior equivalent actions, which is interesting in its own right.
For this recall that two actions $\alpha,\beta:G\to \Aut(A)$ on a given C*-algebra $A$ are called
{\em exterior equivalent} if there exists a strictly continuous map $u: G\to\U\M(A)$ such that
for all $s,t\in G$ we have
\begin{equation}\label{eq-ext}
\alpha_s=\Ad u_s \circ \beta_s\quad\text{and}\quad u_{st}=u_s\beta_s(u_t).
\end{equation}
It is a well-known fact that exterior equivalent actions give isomorphic full (and reduced) crossed products.
Indeed, it is fairly straightforward to check that the map
$$\Phi_c: C_c(G,A)\to C_c(G,A); \Phi_c(f)(s)=f(s)u_s$$
extends to an isomorphism $\Phi_u: A\rtimes_{\alpha, u}G\congto A\rtimes_{\beta,u}G$.
A short computation shows that it is a $*$-isomorphism on the level of $C_c(G,A)$.
Then one observes that there is a one-to-one correspondence between nondegenerate
covariant homomorphisms $(A,G,\beta)$ and $(A,G,\alpha)$
such that for any covariant homomorphism $(\pi, V)$ for $(A,G,\beta)$
the pair $(\pi, (\pi\circ u)\cdot V)$ is the corresponding covariant representation of $(A, G,\alpha)$.
The result then follows from the fact that
$\pi\rtimes (\pi\circ u\cdot V)=(\pi\rtimes V)\circ \Phi_c$ for all
$f\in C_c(G,A)$. Alternatively, the isomorphism is given as the integrated form of the
covariant homomorphism $(i_A^\beta, (i_A^\beta\circ u)\cdot i_G^\beta)$ of $(A,G,\alpha)$
into $\M(A\rtimes_{\beta, u}G)$. We refer to \cite[Lemma 3.3]{PR} for a detailed proof in the more general
setting of twisted actions (with roles of $\alpha$ and $\beta$ changed).  We then have

\begin{lemma}\label{lem-ext}
Suppose that $\rtimes_\mu$ is a strongly Morita compatible crossed-product functor for $G$.
Suppose further that $\alpha,\beta:G\to \Aut(A)$ are exterior equivalent with respect to the one-cocycle
$u:G\to\U\M(A)$. Then the isomorphism $\Phi_u: A\rtimes_{\alpha, u}G\congto A\rtimes_{\beta,u}G$ factors
through an isomorphism $\Phi_\mu: A\rtimes_{\alpha, \mu}G\congto A\rtimes_{\beta,\mu}G$.
\end{lemma}
\begin{proof}
Using the above described correspondence of covariant representations of $A\rtimes_{\alpha,u}G$ and $A\rtimes_{\beta,u}G$
it suffices to show that the integrated form of a covariant representation $(\pi, V)$ of $(A,G,\beta)$ factors through
$A\rtimes_{\beta,\mu}G$ if and only if the corresponding covariant representation $(\pi, (\pi\circ u)\cdot V)$
of $(A,G,\alpha)$ factors through $A\rtimes_{\alpha,\mu}G$. To see this consider $A$
as the standard  $A-A$ equivalence bimodule. Then the
 action $\gamma:G\to \Aut(A)$ given by $\gamma_s(a)=u_s\beta_s(a)$  makes $A$ into a
$G$-equivariant $(A,\alpha)-(A,\beta)$ equivalence bimodule. Since $\rtimes_\mu$ is strongly Morita compatible
we obtain an $A\rtimes_{\alpha,\mu}G-A\rtimes_{\beta,\mu}G$ equivalence bimodule $A\rtimes_{\gamma,\mu}G$,
which is a quotient of the corresponding $A\rtimes_{\alpha,u}G-A\rtimes_{\beta,u}G$ equivalence bimodule
$A\rtimes_{\gamma,u}G$. It follows then from the Rieffel-correspondence
that induction of representations from $A\rtimes_{\beta,\mu}G$ to $A\rtimes_{\beta,\mu}G$
via $A\rtimes_{\gamma,\mu}G$ is the same as induction from
$A\rtimes_{\beta,u}G$ to $A\rtimes_{\beta,u}G$
via $A\rtimes_{\gamma,u}G$, if we regard representations
of the $\mu$-crossed products as representations of the universal crossed product via pull-back with the
quotient maps. This induction process can  be described on the
level of covariant representations by inducing a covariant representation $(\pi, V)$ of $(A,G,\beta)$ via the
$G$\nb-equivariant module $(A,\gamma)$ (see {\cite{Combes:Crossed_Morita}}):
 if $(\pi,V)$ acts on the Hilbert space $\H$, then the induced covariant representation $(\ind^A\pi, \ind^AV)$
acts on the Hilbert space $A\otimes_A\H$ via
$$\ind^A\pi(a)(b\otimes \xi)=ab\otimes \xi\quad\text{and}\quad \ind^AV_s(b\otimes \xi)=\gamma_s(b)\otimes V_s\xi$$
for all $a\in A$, $s\in G$. But then a short calculation shows that the
 canonical isomorphism $A\otimes_A\H\cong \H$ given by  $b\otimes \xi\mapsto \pi(b)\xi$
tranfers this representation to $(\pi, (\pi\circ u)\cdot V)$. Hence induction via $A\rtimes_{\gamma,\mu}G$
gives the desired correspondence of representations.
\end{proof}

\begin{remark}\label{rem-ext}
(1) The conclusion of the above lemma can definitely fail for crossed-product functors $\rtimes_\mu$ which are
not strongly Morita compatible. To see an example let $G$ be any non-amenable group.
Let $\lambda=\lambda_G:G\to \U(L^2(G))=\U\M(\K(L^2(G)))$ denote the regular representation of $G$.
Then $\lambda$ induces an exterior equivalence between the actions $\Ad\lambda$ and $\id_\K$
of $G$ on $\K=\K(L^2(G))$.
 Note that all irreducible covariant representations of
$(\K, G,\id_\K)$ are equivalent to one of the form $(\id_\K\otimes 1_\H, 1_{L^2(G)}\otimes V)$ on
$L^2(G)\otimes \H$ for some  Hilbert space $\H$, where $V:G\to \U(\H)$ is an irreducible unitary representation of $G$.
This  follows easily from the isomorphism $\K\rtimes_{\id,u}G\cong \K\otimes C^*(G)$.

Now let $\rtimes_{\lambda_{BG}}$ be the Brown-Guentner functor
corresponding to the regular representation (compare Section \ref{sec:BG}).
Since the unitary part $1_{L^2(G)}\otimes V\sim V$ is weakly contained in $\lambda$ if and only if
$V$ is weakly contained in $\lambda$ it follows that
$\K\rtimes_{\id, \lambda_{BG}}G\cong \K\otimes C_r^*(G)\cong \K\rtimes_{\id,r}G$.
On the other hand, the covariant representation of $(\K,G,\Ad\lambda)$
corresponding to $(\id_\K\otimes 1_\H, 1_{L^2(G)}\otimes V)$ is given by
$(\id_\K\otimes 1_\H, \lambda \otimes V)$. Since by Fell's trick
$\lambda\otimes V\sim \lambda$, we see that
the integrated form of every irreducible covariant representation of $(\K, G,\Ad\lambda)$
factors through $\K\rtimes_{\Ad\lambda, \lambda_{BG}}G$, which implies that this crossed
product coincides with the universal crossed product $\K\rtimes_{\Ad\lambda, u}G\cong \K\otimes C^*(G)$.
Since $G$ is not amenable, this is not (canonically)
isomorphic to $\K\rtimes_{\Ad\lambda, \lambda_{BG}}G\cong \K\otimes C_r^*(G)$.

(2)  Recall that two systems $(A,G,\alpha)$ and $(B,G,\beta)$ are {\em conjugate}
if there exists an isomorphism $\Psi: A\to B$ such that the action $\alpha^\Psi:=\Psi\circ \alpha\circ \Psi^{-1}$ on
$B$ coincides with $\beta$
and they are called {\em outer conjugate} if $\alpha^\Psi$
is exterior equivalent to $\beta$. Since by functoriality, conjugate actions give isomorphic crossed products
for any crossed-product functor $\rtimes_\mu$, it follows from Lemma \ref{lem-ext} that
outer conjugate actions give isomorphic $\mu$-crossed product if $\rtimes_\mu$
is strongly Morita compatible.
\end{remark}

As a corollary of Lemma \ref{lem-ext} we get

\begin{corollary}\label{cor-ext} Suppose that $\alpha:G\to \Aut(A)$ is an action and that $U:G\to \U(\H)$
is a unitary representation of $G$. Assume further
 that $\rtimes_\mu$ is a strongly Morita compatible crossed-product functor
for $G$ and let $(i_A^\mu, i_G^\mu)$ be the canonical covariant homomorphism of $(A,G,\alpha)$ into
$\M(A\rtimes_{\alpha,\mu}G)$. Then $(i_A^\mu\otimes \id_{\K(\H)}, i_G^\mu\otimes U)$ is a covariant homomorphism
of the system $(A\otimes \K(\H), G,\alpha\otimes\Ad U)$ into $\M(A\rtimes_{\alpha,\mu}G\otimes \K(\H)$
whose integrated form factors through an isomorphism
$$\Phi_\mu: (A\otimes\K(\H))\rtimes_{\alpha\otimes\Ad U, \mu}G\congto A\rtimes_{\alpha,\mu}G\otimes \K(\H).$$
\end{corollary}
\begin{proof} We first observe that the homomorphism $1_A\otimes U:G\to \U\M(A\otimes \K(\H))$
implements an exterior equivalence between $(A\otimes \K(\H), \alpha\otimes \Ad U)$ and
$(A\otimes \K(\H),\alpha\otimes\id_{\K(\H)})$, hence by Lemma \ref{lem-ext} we get an isomorphism
$$(A\otimes\K(\H))\rtimes_{\alpha\otimes\Ad U, \mu}G\congto (A\otimes\K(\H))\rtimes_{\alpha\otimes\id_\K, \mu}G$$
given by the covariant homomorphism
$(i_{A\otimes\K(\H)}^\mu, i_{A\otimes \K(\H)}^\mu\circ (1_A\otimes U)\cdot i_G^\mu)$.
We also have an isomorphism $(A\otimes \K(\H))\rtimes_{\alpha\otimes\id_\K,\mu}G\cong A\rtimes_{\alpha,\mu}G\otimes \K(\H)$
given by the integrated form of the covariant homomorphism $(i_A^\mu\otimes\id_\K, i_G^\mu\otimes 1_\H)$.
To see this we consider the $(A\otimes\K(\H),\alpha\otimes \id_\K)-(A,\alpha)$ equivalence bimodule
$(A\otimes \H, \alpha\otimes 1_\H)$. It is easy to check that the
 map
$C_c(G,A)\otimes\H\to C_c(G,A\otimes \H); f\otimes\xi\mapsto[s\mapsto f(s)\otimes \xi]$ is isometric with respect to the
$A\rtimes_{\alpha,\mu}G$-valued inner products and therefore  extends to an isomorphism
 $A\rtimes_{\alpha,\mu}G\otimes\H\cong (A\otimes \H)\rtimes_{\alpha\otimes 1_\H,\mu}G$.
 In this picture, the left action of
$(A\otimes\K(\H))\rtimes_{\alpha\otimes\id_\K, \mu}G$ on $A\rtimes_{\alpha,\mu}G\otimes\H$
is given by the integrated form of $(i_A^\mu\otimes\id_\K, i_G^\mu\otimes 1_\H)$.
By strong Morita compatibility this integrates to an isomorphism
$$(A\otimes\K(\H))\rtimes_{\alpha\otimes\id_\K, \mu}G\cong \K(A\rtimes_{\alpha,\mu}G\otimes\H)=A\rtimes_{\alpha,\mu}G\otimes \K(\H).$$
The result now follows from composing  this isomorphism with the integrated form of
$(i_{A\otimes\K(\H)}^\mu,  (i_{A\otimes \K(\H)}^\mu\circ (1_A\otimes U))\cdot i_G^\mu)$.
\end{proof}

\begin{remark}\label{rem-ext2}  Recall that in \cite{Baum-Guentner-Willett} Baum, Guentner and Willett define a crossed product functor
$\rtimes_\mu$ to be {\em Morita compatible} if the isomorphism
$$\Phi_\mu: (A\otimes\K(\H))\rtimes_{\alpha\otimes\Ad U, \mu}G\congto A\rtimes_{\alpha,\mu}G\otimes \K(\H)$$
holds in the special case where $\H=\ell^2(\N)\otimes L^2(G)$ and $U=1_{\ell^2(\N)}\otimes  \lambda_G$.
Thus the above corollary shows that strong Morita compatibility in our sense implies
Morita compatibility in the sense of \cite{Baum-Guentner-Willett}. We refer to Section \ref{sec:minimal} below
for a more detailed comparison of the two notions of Morita compatibility. In particular
we shall see in Proposition \ref{prop-Morita} below that both conditions are equivalent if we restrict our attention
to systems $(A,G,\alpha)$ with  $A$
$\sigma$-unital.
\end{remark}

We are now ready for the main result of this section which shows that all crossed products
coming from correspondence functors admit dual coactions and enjoy Imai-Takai
duality. Recall for this that the dual coaction
$$\widehat{\alpha}: A\rtimes_{\alpha,u}G\to \M(A\rtimes_{\alpha,u}G\otimes C^*(G))$$
is given as the integrated form of the covariant homomorphism $(i_A^u\otimes 1_G, i_G^u\otimes u_G)$,
where $u_G$ denotes the universal representation of $G$.

\begin{theorem}\label{thm-dual}
Suppose that $\rtimes_\mu$ is a correspondence crossed-product functor for $G$. Then the dual coaction
$\widehat{\alpha}$ of $G$ on $A\rtimes_{\alpha,u}G$
factors to give a dual coaction
$$\widehat{\alpha}_\mu: A\rtimes_{\alpha,\mu}G\to \M(A\rtimes_{\alpha,\mu}G\otimes C^*(G))$$
given by the integrated form of $(i_A^\mu\otimes 1_{C^*(G)}, i_G^\mu\otimes u_G)$.
Moreover, the double dual crossed product $A\rtimes_{\alpha,\mu}G\rtimes_{\widehat{\alpha}_\mu}\widehat{G}$
is isomorphic to $A\otimes \K(L^2(G))$ such that this isomorphism transfers the double dual action
$\widehat{\widehat{\alpha}_\mu}:G\to \Aut( A\rtimes_{\alpha,\mu}G\rtimes_{\widehat{\alpha}_\mu}\widehat{G})$
to the action $\alpha\otimes\Ad\rho:G\to \Aut(A\otimes \K(L^2(G)))$. In particular, it follows that
$$A\rtimes_{\alpha,\mu}G\rtimes_{\widehat{\alpha}_\mu}\widehat{G}
\rtimes_{\widehat{\widehat{\alpha}}_\mu,\mu}G\cong A\rtimes_{\alpha,\mu}G\otimes \K(L^2(G)).$$
\end{theorem}

Note that the map $\widehat{\alpha}_\mu$, if it exists, is automatically faithful, since the
composition $\big(\id_{A\rtimes_{\alpha,\mu}G}\otimes 1_G)\circ \widehat{\alpha}_G$,
with $1_G$ the trivial representation of $G$, is the integrated form of $(i_A^\mu,i_G^\mu)$, hence
the identity on $A\rtimes_{\alpha,\mu}G$. In the language of \cite{Buss-Echterhoff:Exotic_GFPA}*{Definition~4.5}, the above result says that $\widehat{\alpha}_\mu$ satisfies $\mu$-duality or is a $\mu$-coaction.

\begin{proof} For the first assertion it suffices to show that the integrated form of
$(i_A^\mu\otimes 1_{C^*(G)}, i_G^\mu\otimes u_G)$ factors through $A\rtimes_{\alpha,\mu}G$. Then  all requirements for a coaction carry over to $\widehat{\alpha}_\mu$.
For this let us represent $C^*(G)$ faithfully on Hilbert space via the integrated form
of a unitary representation $U: G\to \U(\H)$. Then it suffices to show that the covariant representation
$(i_A^\mu\otimes 1_\H, i_G^\mu\otimes U)$ of $(A,G,\alpha)$ into $\M(A\rtimes_{\alpha,\mu}G\otimes \K(\H))$
factors through $A\rtimes_{\alpha,\mu}G$. Since $\rtimes_\mu$ is Morita compatible, we
know from Corollary \ref{cor-ext} that the covariant representation
$(i_A^\mu\otimes \id_\K, i_G^\mu\otimes U)$ of $(A\otimes \K(\H), G, \alpha\otimes \Ad U)$ into
$\M(A\rtimes_{\alpha,\mu}G\otimes \K(\H))$ factors through an isomorphism
$$\Phi_\mu: (A\otimes\K(\H))\rtimes_{\alpha\otimes\Ad U, \mu}G\congto A\rtimes_{\alpha,\mu}G\otimes \K(\H).$$
Since $\rtimes_\mu$ is functorial for generalised homomorphisms,
the $G$-equivariant $*$-homo\-morphism $\id_A\otimes 1_\H: A\to \M(A\otimes \K(\H))$ induces a
$*$-homomorphism $A\rtimes_{\alpha,\mu}G\to \M( (A\otimes\K(\H))\rtimes_{\alpha\otimes\Ad U, \mu}G)$
given by sending a function $f\in C_c(G,A)$ to the function $[s\mapsto (f(s)\otimes 1_\H)]\in C_c(G, \M(A\otimes \K(\H)))$.
The composition of both maps, which is well defined on $A\rtimes_{\alpha,\mu}G$, is given
by the integrated form of $(i_A^\mu\otimes 1_\H, i_G^\mu\otimes U)$.

Imai-Takai duality now follows from the well-known Imai-Takai duality theorems for the full and
reduced crossed products (see \cite{IT} for the reduced version or \cite[Theorem A.67]{Echterhoff-Kaliszewski-Quigg-Raeburn:Categorical} for
a discussion of the full and reduced cases): By construction, both quotient maps
$$A\rtimes_{\alpha,u}G\stackrel{q_\mu}{\onto} A\rtimes_{\alpha,\mu}G\stackrel{q_r}{\onto}A\rtimes_{\alpha,r}G$$
are equivariant for the coactions $\widehat{\alpha}$, $\widehat{\alpha}_\mu$ and $\widehat{\alpha}_r$, respectively.
They therefore induce quotient maps
$$A\rtimes_{\alpha,u}G\rtimes_{\widehat{\alpha}}\widehat{G}\stackrel{q_\mu\rtimes\widehat{G}}{\onto}
A\rtimes_{\alpha,\mu}G\rtimes_{\widehat{\alpha}_\mu}\widehat{G}\stackrel{q_r\rtimes\widehat{G}}{\onto}
A\rtimes_{\alpha,r}G\rtimes_{\widehat{\alpha}_r}\widehat{G}.$$
However, the composition is an isomorphism by  \cite[Proposition A.6.1]{Echterhoff-Kaliszewski-Quigg-Raeburn:Categorical}, hence
Imai-Takai duality for $\mu$-crossed products follows from Imai-Takai duality for
full and/or reduced crossed products.
The last assertion of the theorem is then
a direct consequence of Corollary \ref{cor-ext}.
\end{proof}

\begin{corollary}\label{cor-ideal}
Suppose that $\rtimes_\mu$ is a correspondence crossed-product functor for $G$.
Then the group algebra $C_\mu^*(G):=\C\rtimes_{\mu}G$ corresponds to
a translation invariant weak-* closed ideal $E\subseteq B(G)$ as in Proposition \ref{prop-BG-KLQ}.
\end{corollary}
\begin{proof}  Let $v: G\to \U\M(C_\mu^*(G))$ be the homomorphism which integrates
to the quotient map $q_\mu: C^*(G)\to C_\mu^*(G)$. Then it follows from
Theorem \ref{thm-dual} that the integrated form of $v\otimes u_G$ factors through $C_\mu^*(G)$.
Since $(\id_{C_\mu^*(G)}\otimes 1_G)\circ (v\otimes u_G)=v$ we see that
$v\sim v\otimes u_G$ and result  follows from the discussions in Section \ref{sec:KLQ}.
\end{proof}

\section{KK-descent and the Baum-Connes conjecture}\label{sec:descent}
In this section we want to show that every crossed-product functor
$(A,\alpha)\mapsto A\rtimes_{\alpha,\mu}G$ which extends to
a functor $\Cor(G)\to \Cor$ as discussed in the previous section
will always allow a descent in Kasparov's bivariant $KK$-theory.
Let us briefly recall the cycles  in Kasparov's equivariant KK-theory group
$KK^G(A,B)$ in which $(A,\alpha)$ and $(B,\beta)$ are $G$\nb-algebras.
As in \cite{Kasparov:Novikov} we allow $\Z/2\Z$\nb-graded \cstar{}algebras and Hilbert modules.
In this case all $G$\nb-actions have to commute with  the given  gradings, so the grading
on the objects $(A,\alpha)$ and morphisms $(\E,\gamma)$ will always
induce canonical gradings on their crossed products $A\rtimes_{\alpha,\mu}G$
and $\E\rtimes_{\gamma,\mu}G$, respectively.
For convenience and to make sure that Kasparov products always exist we shall
assume that $A$ and $B$  are separable and that $G$ is second countable.

If  $(A,\alpha)$ and $(B,\beta)$ are two $G$\nb-algebras, we define $\EE^G(A,B)$ to be the set of
all quadruples  $(\E, \gamma, \phi, T)$ such that  $(\E,\phi, \gamma)$ is a
correspondence from $(A,\alpha)$ to $(B,\beta)$ such that $\E$ is countably generated,
and $T\in \Lb(\E)$ is an operator of degree one satisfying the following conditions:
\begin{enumerate}
\item $s\mapsto \phi(\alpha_s(a))T$ is continuous  for all $a\in A$, and
\item $[T,\phi(a)], (T-T^*)\phi(a), (T^2-1)\phi(a), (\Ad\gamma_s(T)-T)\phi(a)\in \K(\E)$
\end{enumerate}
for all $a\in A$, where we use the graded commutator $[\cdot,\cdot]$ in (2).
Two Kasparov cycles $(\E_1, \gamma_1, \phi_1, T_1)$  and
$(\E_2, \gamma_2, \phi_2, T_2)$ are  {\em isomorphic} if there exists
a bijection $u:\E_1\to \E_2$ which preserves all structures.
  If $G$ is trivial, there is no action $\gamma$ on $\E$
and the elements in $\EE(A,B)$ are just given by triples $(\E,\phi,T)$ with the obvious properties.

Write $B[0,1]$ for the algebra of continuous functions from
the unit interval $[0,1]$ to $B$ and $\epsilon_t: B[0,1]\to B$ for the evaluation at $t\in[0,1]$.
The action of $G$ on $B[0,1]=B\otimes C([0,1])$ is given by $\beta\otimes \id$.
A {\em homotopy} between two elements  $(\E_0, \gamma_0, \phi_0, T_0)$  and
$(\E_1, \gamma_1, \phi_1, T_1)$ in $\EE^G(A,B)$ is an element
$(\E,\gamma, \phi, T)\in \EE^G(A, B[0,1])$ such that
$$(\E_i, \gamma_i, \phi_i, T_i)\cong (\E\hotimes_{\eps_i}B, \gamma\hotimes 1, \phi\hotimes 1, T\hotimes 1), \quad i=0,1,$$
where, for each $t\in [0,1]$, $\E\hotimes_{\eps_t}B$ denotes the balanced tensor product $\E\otimes_{B[0,1]}B$
in which $B[0,1]$ acts on $B$ via $\epsilon_t$.
Then Kasparov defines $KK^G(A,B)$ as the set of homotopy classes in $\EE^G(A,B)$. It is an abelian group with respect to addition given by
taking direct sums of elements in $\EE^G(A,B)$.

Recall also that for $i=0,1$ one may define $KK_0^G(A,B):=KK^G(A,B)$ and $KK_1^G(A,B):=
KK^G(C_0(\mathbb R)\otimes A,B)\cong KK^G(A, C_0(\mathbb R)\otimes B)$ where $G$ acts trivially on $C_0(\mathbb R)$ (the last isomorphism follows from
Kasparov's Bott-periodicity theorem for $KK$-theory).

Note that if $\E$ is a graded module, then there  are natural gradings on the
crossed products determined by applying the given gradings pointwise to
functions with compact supports. The balanced tensor product of two
graded modules carries the diagonal grading. The following proposition extends
the descent given by Kasparov in \cite{Kasparov:Novikov} for full and reduced
crossed products to arbitrary correspondence crossed-product functors.

\begin{proposition}\label{prop-descent}
Let $(A,\alpha)\mapsto A\rtimes_{\alpha,\mu}G$ be a correspondence crossed-product functor. Then there is a
well-defined descent homomorphism
$$J_G^\mu: KK_{i}^G(A,B)\to KK_{i}(A\rtimes_{\alpha,\mu}G, B\rtimes_{\beta,\mu}G), \quad i=0,1,$$
given by sending a class $[\E,\gamma, \phi,T]\in KK^G_{i}(A,B)$ to the class
$$[\E\rtimes_{\gamma,\mu}G, \phi\rtimes_\mu G, T_\mu]\in KK_{i}(A\rtimes_{\alpha,\mu}G, B\rtimes_{\beta,\mu}G),$$
where $T_\mu\in \Lb(\E\rtimes_{\gamma,\mu}G)$ is the unique continuous extension
of the operator $\tilde{T}:C_c(G,\E)\to C_c(G,\E)$ given by $(\tilde{T}\xi)(s):=T(\xi(s))$ for all $s\in G$.
The descent is compatible with Kasparov products, \ie, we have
$$J_G^\mu(x\otimes_B y)=J_G^\mu(x)\otimes_{B\rtimes_{\beta,{\mu}}G}J_G^\mu(y)$$
for all $x\in KK^G_{i}(A,B)$ and $y\in KK^G_{j}(B,D)$ and $J_G^\mu(1_B)=1_{B\rtimes_{\beta, \mu}G}$.
\end{proposition}

\begin{proof}
We basically follow the proof of \cite[Theorem on p. 172]{Kasparov:Novikov}.
First observe that since $\Lb(\E\rtimes_{\gamma,\mu}G)=\M(\K(\E\rtimes_{\gamma,\mu}G))=\M(\K(\E)\rtimes_{\Ad\gamma,\mu}G)$,
the operator $T_\mu$ is just the image of $T$ under the extension of the canonical  inclusion
from $\K(\E)$ into $\M(\K(\E)\rtimes_{\Ad\gamma,\mu}G)$ to $\M(\K(\E))\cong \Lb(\E)$. It therefore exists.

Next, the conditions for $(\E\rtimes_{\gamma,\mu}G, \phi\rtimes_\mu G, T_\mu)$ follow from the fact
that for any $a\in C_c(G,A)$ the elements $[\tilde{T},\phi(a)], (\tilde{T}^2-1)\phi(a), (\tilde{T^*}-\tilde{T})\phi(a)$ (as operators on $C_c(G,\E)$) lie in $C_c(G,\K(\E))$, and hence in  $\K(\E)\rtimes_{\Ad\gamma,\mu}G\cong \K(\E\rtimes_{\gamma, \mu}G)$.
This shows that $(\E\rtimes_{\beta,\mu}G, \phi\rtimes_\mu G, T_\mu)$ determines a class in
$KK(A\rtimes_{\alpha,\mu}G,B\rtimes_{\beta,\mu}G)$.

To see that this class does not depend on
the choice of the representative $(\E,\gamma, \phi,T)\in \EE^G(A,B)$ we need
to observe that the procedure preserves homotopy. So assume now that
$(\E,\gamma, \phi,T)\in \EE^G(A,B[0,1])$. By Lemma \ref{tensor-with-spaces}
we know that $B[0,1]\rtimes_{\beta\otimes\id,\mu}G\cong (B\rtimes_{\beta,\mu}G)[0,1]$,
hence $(\E\rtimes_{\gamma,\mu}G, \phi\rtimes_\mu G, T_\mu)\in \EE(A\rtimes_{\alpha,\mu}G, B\rtimes_{\beta,\mu}G[0,1])$. It follows then from correspondence functoriality of our crossed-product functor that
evaluation of $(\E\rtimes_{\gamma,\mu}G, \phi\rtimes_\mu G, T_\mu)$ at any $t\in [0,1]$
coincides with the descent of the evaluation of
$(\E,\gamma, \phi, T)$ at $t$ (this implies that the modules coincide, but a short
look at the operator on functions with compact supports also shows that the operators
coincide). This shows that the descent gives a well-defined homomorphism of  $KK$-groups.
Using the isomorphism $(C_0(\mathbb R)\otimes A)\rtimes_{\mu}G\cong C_0(\mathbb R)\otimes (A\rtimes_{\mu}G)$ if $G$ acts
trivially on $\mathbb R$,
which follows from Lemma \ref{tensor-with-spaces}, it follows that it preserves the  dimension of
the $KK$-groups.

Finally the fact that the descent is compatible with Kasparov products follows from
the same arguments as used in the original proof of Kasparov as given
on \cite[page 173]{Kasparov:Novikov} for the full and reduced crossed products
together with an application of Lemma \ref{lem-iso} to $\K(\E_2, \E_1\hotimes_B\E_2)$.
\end{proof}

Suppose now that $(A,G,\alpha)$ is a \cstar{}dynamical system with $A$ separable and $G$ second countable.  Suppose that $A\rtimes_{\alpha,\nu}G$ is any quotient of the universal crossed product
$A\rtimes_{\alpha,\un}G$ with quotient map $q_{\nu}:A\rtimes_{\alpha,\un}G\to A\rtimes_{\alpha,\nu}G$. Using the Baum-Connes assembly map
$$\as_{(A,G)}^\un: K_*^\top(G;A)\to K_*(A\rtimes_{\alpha,\un}G)$$
for the full crossed product, we obtain an assembly map
$$\as_{(A,G)}^\nu:K_*^\top(G;A)\to K_*(A\rtimes_{\alpha,\nu}G)$$
by putting
\begin{equation}\label{as0}
\as_{(A,G)}^\nu=q_{\nu,*}\circ \as_{(A,G)}^\un.
\end{equation}
Of course, in general we cannot assume that
this map has any good properties, since we should assume that there is a huge variety
of  quotients of $A\rtimes_{\alpha,\un}G$ with different K-theory groups.  If the \cstar{}algebra $A\rtimes_{\alpha,\nu}G$ lies
between the maximal and reduced crossed products, then the Baum-Connes conjecture predicts that $\as_{(A,G)}^\nu$ is split injective, but we cannot say more than that.

On the other hand, if $G$ is $K$-amenable, then the quotient map
$q_\red:A\rtimes_{\alpha,\un}G\to A\rtimes_{\alpha,\red}G$ induces an isomorphism
of K-theory groups, and there might be a chance that a similar result will hold for
intermediate crossed products. For arbitrary intermediate crossed products this is not true,
as follows from \cite[Example 6.4]{Baum-Guentner-Willett}, where an easy
counter example is given that applies to any $K$-amenable\footnote{The discussion in \cite[Example 6.4]{Baum-Guentner-Willett} is for a-T-menable, non-amenable groups, but the same argument applies if one assumes $K$-amenability in place of a-T-menability.}, non-amenable group. However, we shall
see below that it holds for all correspondence crossed-product functors
for any $K$-amenable group!

We need to recall the construction of the assembly map:
For this assume that $G$ is a second countable locally compact group and that $\underline{EG}$ is
a universal proper $G$\nb-space, \ie, $\underline{EG}$ is a locally compact
proper $G$\nb-space such that for every locally compact $G$\nb-space $Y$ there is a
(unique up to $G$\nb-homotopy) continuous $G$\nb-map $\varphi:Y\to\underline{EG}$.
The topological $K$-theory of $G$ with coefficient $A$ is defined as
$$K_*^{\top}(G;A):=\lim_{X\subseteq \underline{EG}}KK_*^G(C_0(X),A),$$
where $X\subseteq \underline{EG}$ runs through all $G$\nb-compact  (which means that $X$ is
a closed $G$\nb-invariant subset with $G\backslash X$ compact) subsets of $\underline{EG}$.

Now, using properness, for any such $G$\nb-compact subset $X\subseteq\underline{EG}$ we may choose a
cut-off function $c\in C_c(G)^+$ such that $\int_{G} c^2(s^{-1}x)\, ds=1$ for all $x\in X$.
Define $p_X(s,x)=\Delta_G(s)^{-1/2} c(x)x(s^{-1}x)$. It follows from the properties of $c$ that
the function
$p_X\in C_c(G, C_0(X))\subseteq C_0(X)\rtimes G$ is an orthogonal projection.  As any two cut-off functions with the properties above are homotopic,
the $K$-theory class $[p_X]\in K_*(C_0(X)\rtimes G)$ does not depend on the choice of
the cut-off function $c$. Note also that proper actions on spaces are always amenable,
hence the universal and reduced crossed products coincide. This also implies that
all intermediate crossed products coincide.
Using this and the above defined descent in KK-theory, we may now construct a
direct assembly map for any correspondence crossed-product functor
$(A,\alpha)\mapsto A\rtimes_{\alpha,\mu}G$ and separable $G$-algebra $A$ given by the composition
$${\jas}_{X}^\mu: KK_*^G(C_0(X), A)\stackrel{J^\mu}{\longrightarrow} KK_*(C_0(X)\rtimes G, A\rtimes_{\alpha,\mu}G)
\stackrel{[p_X]\otimes \cdot}{\longrightarrow} K_*(A\rtimes_{\alpha,\mu}G).$$
Note that this is precisely the same construction as given by Baum, Connes and Higson
in \cite{Baum-Connes-Higson:BC} for the universal or reduced crossed products, and as in those cases it is
easy to check that the maps are compatible with the limit structure and therefore define a
direct assembly map
\begin{equation}\label{as1}
\jas_{(A,G)}^\mu: K_*^\top(G;A)\to K_*(A\rtimes_{\alpha,\mu}G);
\end{equation}
the notation $\jas$ is meant to recall that we used the descent functor in the definition.  However, the next result shows that this assembly map is the same as that in line \eqref{as0} above: it is a $KK$-theoretic version of \cite[Proposition 4.5]{Baum-Guentner-Willett}.

\begin{lemma}\label{twoas}
Assume $(A,G,\alpha)$ is a \cstar{}dynamical system with $A$ separable and $G$ second countable.  Assume $\rtimes_\mu$ is a correspondence crossed-product functor for $G$. Then the diagram
$$
\xymatrix{ & & K_*(A\rtimes_{\un}G)\ar[d]^{(q_\mu)_*}\\
K_*^\top(G;A) \ar[urr]^{\as^{\un}_{(A,G)}}   \ar[rr]^{\jas^{\mu}_{(A,G)}}  \ar[drr]_{\as^{\red}_{(A,G)}} & &  K_*(A\rtimes_{\mu}G)\ar[d]^{(q^\mu_\red)_*}\\
& & K_*(A\rtimes_{\red}G)}
$$
commutes. In particular, the assembly maps $\as_{(A,G)}^\mu$ and $\jas_{(A,G)}^\mu$ of lines \eqref{as0} and \eqref{as1} above agree.
\end{lemma}

\begin{proof}
Let $[q_\mu]\in KK(A\rtimes_{\un}G,A\rtimes_\mu G)$ be the $KK$-class induced by the quotient \Star{}homomorphism
$q_\mu:A\rtimes_{\un} G \onto A\rtimes_\mu G$.
For commutativity of the top triangle, it suffices to show that for any $G$-compact subset $X$ of $\underline{EG}$ the diagram
$$
\xymatrix{ KK_*^G(C_0(X),A) \ar[dd]^{J^\mu} \ar[r]^-{J^\un} & KK_*(C_0(X)\rtimes G , A\rtimes_{\alpha,\un} G) \ar[d]^{[p_X]\otimes\cdot} \\
& KK_*(\C,A\rtimes_{\alpha,\un} G) \ar[d]^{\cdot \otimes [q_\mu]} \\
KK_*(C_0(X)\rtimes G , A\rtimes_{\alpha,\mu} G )\ar[r]^-{[p_X]\otimes \cdot} & KK_*(\C,A\rtimes_{\alpha,\mu} G) }
$$
commutes (we have used that $C_0(X)\rtimes G=C_0(X)\rtimes_\mu G$ to identify the group in the bottom left corner).  In symbols, commutativity of this diagram means that for any $x\in KK_*^G(C_0(X),A)$ we have
$$
([p_X]\otimes J^\un(x))\otimes [q_\mu]=[p_X]\otimes J^\mu(x).
$$
Using associativity of the Kasparov product, it thus suffices to show that for any $x\in KK_*^G(C_0(X),A)$,  $J^\un(x)\otimes[q_{\mu}]=J^\mu(x)$.

Indeed, say $x$ is represented by the quadruple $(\mathcal{E},\gamma,\phi,T)$.  Then $J^\un(x)\otimes[q_{\mu}]$ and $J^\mu(x)$ are represented by the triples
$$
(\mathcal{E}\rtimes_{\gamma,\un}G\otimes_{B\rtimes_{\beta} G}B\rtimes_{\beta,\mu} G,\phi\rtimes_\mu G \otimes 1,T_\un\otimes 1) \text{ and }(\mathcal{E}\rtimes_{\gamma,\mu}G,\phi\rtimes_\mu G ,T_\mu)
$$
respectively.  It is easy to see that the canonical isomorphism
$\mathcal{E}\rtimes_{\gamma,\un}G\otimes_{B\rtimes_{\beta,\un} G}B\rtimes_{\beta,\mu} G\congto \mathcal{E}\rtimes_{\gamma,\mu}G$
of Equation~\eqref{eq:CanonicalIsom} sends $\phi\rtimes_\mu G \otimes 1$ to $\phi\rtimes_\mu G$ and $T_\un\otimes 1$ to $T_\mu$, and hence
induces an isomorphism of these triples.

This shows commutativity of the upper diagram. Commutativity of the lower one follows by similar arguments.
\end{proof}

Recall now that a group $G$ is said to satisfy the {\em strong Baum-Connes conjecture}
(or that $G$ has a $\gamma$-element equal to one), if
there exists an $\underline{EG}\rtimes G$-algebra $\mathcal B$ (which means that
there is a $G$\nb-equivariant nondegenerate \Star{}homomorphism $C_0(\underline{EG})$ into
the center $Z\M(\mathcal B)$ of the multiplier algebra $\M(\mathcal B)$),  together with classes
$D\in KK_*^G(\mathcal B, \C)$ and $\eta\in KK_*^G(\C,\mathcal B)$ such that
$\eta\otimes_{\mathcal B} D =1_G\in KK^G(\C,\C)$. It was shown by
Higson and Kasparov in \cite{Higson-Kasparov:Etheory} that all a-T-menable groups
satisfy the strong Baum-Connes conjecture. We then get

\begin{theorem}\label{thm-gamma}
Suppose that $G$ is a locally compact group which satisfies the strong Baum-Connes conjecture
and let $(A,\alpha)\to A\rtimes_{\alpha,\mu}G$ be a correspondence crossed-product functor.
Then the $\mu$-assembly map
$$\as^\mu_{(A,G)}: K_*^\top(G;A)\to K_*(A\rtimes_{\alpha,\mu}G)$$
is an isomorphism for all $G$\nb-algebras $(A,\alpha)$.
\end{theorem}

This can be proved directly in the same way as the special case dealing with the reduced crossed product.  However, we shall obtain the result as
a consequence of the known  isomorphism of the assembly map for the reduced crossed products
(\eg, see \cite{Tu:Kamenable}*{Theorem 2.2} or \cite{Higson-Kasparov:Etheory}*{Theorem 1.1}) together
with Lemma \ref{twoas} above and Theorem \ref{thm-Kamenable} below, which in particular implies that for every $K$-amenable group
$G$ the quotient map $A\rtimes_{\un}G\onto A\rtimes_{\mu}G$ induces an isomorphism in $K$-theory.

Recall that a locally compact group $G$ is called {\em $K$-amenable} in the sense of Cuntz and
Julg-Valette
(\cite{Cuntz:Kamenable, Julg-Valette:Kamenable})
if the unit element
$1_G\in KK^G(\C,\C)$ can be represented by a cycle $(\H, \gamma, T)$ in which
the unitary representation $\gamma:G\to \U(\H)$  on the Hilbert space $\H$ is weakly
contained in the regular representation. In other words, its integrated form factors through
a $*$-representation $\gamma: C_r^*(G)\to \Lb(\H)$. (The left action of $\C$ on $\H$ is assumed to be
by scalar multiplication, so we omit it in our notation.)
It is shown by Tu in \cite{Tu:Kamenable}*{Theorem 2.2} that every
locally compact group which satisfies the strong Baum-Connes conjecture is also $K$-amenable
and it  is shown by Julg and Valette in \cite{Julg-Valette:Kamenable}*{Proposition 3.4}
 that for any $K$-amenable group
the regular representation
$$A\rtimes_{\alpha,\un}G\onto A\rtimes_{\alpha,\red}G$$
induces a $KK$-equivalence for any $G$\nb-\cstar{}algebra $(A,\alpha)$.
We shall now extend this result to  correspondence functors:

\begin{theorem}[{\cf~\cite{Julg-Valette:Kamenable}*{Proposition 3.4}}]\label{thm-Kamenable}
Suppose that $G$ is $K$-amenable and let $\rtimes_\mu$ be any correspondence functor for $G$.
Then both  maps in the  sequence
$$A\rtimes_{\alpha,\un}G\;\stackrel{q_\mu}{\onto}\;A\rtimes_{\alpha,\mu}G\;\stackrel{q_\red^\mu}{\onto}\;A\rtimes_{\alpha,\red}G$$
are $KK$-equivalences.
\end{theorem}

\begin{proof}[Proof of Theorem \ref{thm-Kamenable}]
It is enough to show that the  quotient map $q_\red^\mu:A\rtimes_{\alpha,\mu}G\to A\rtimes_{\alpha,\red}G$ is a $KK$-equivalence.
Then the result for the first map follows from composing the $KK$-equivalence in case $\rtimes_\mu=\rtimes_{\un}$ with the
inverse of $[q_\red^\mu]$ in $KK(A\rtimes_{\alpha,\red}G, A\rtimes_{\alpha,\mu}G)$.
We closely follow the arguments given in the proof of  \cite{Julg-Valette:Kamenable}*{Proposition 3.4}.

So assume that $(\H, \gamma, T)$ is a cycle for $1_G\in KK^G(\C,\C)$ such that $\gamma$ is weakly
contained in $\lambda_G$ and let $(A,\alpha)$ be any $G$\nb-\cstar{}algebra. Since $1_G\otimes_\C 1_A=1_A$, it follows that
 $1_A\in KK^G(A,A)$ is represented by the cycle
$(\H\otimes A, 1_\H\otimes \id_A, \gamma\otimes \alpha, T\otimes 1)$.
Using the isomorphism
$$(\H\otimes A)\rtimes_{\mu}G\cong \H\otimes (A\rtimes_{\mu}G)$$
of Corollary \ref{cor-ext}
we see that $J_\mu(1_A)=1_{A\rtimes_{\mu}G}\in KK(A\rtimes_{\mu}G, A\rtimes_{\mu}G)$ is represented by
the cycle
$$\big(\H\otimes A\rtimes_{\mu} G, \Psi, T\otimes 1\big)$$
The action  $\Psi$ of $A\rtimes_{\alpha,\mu}G$ on  $\H\otimes (A\rtimes_{\mu}G)$ is given by the integrated form of the
covariant homomorphism  $(1_\H\otimes i_A^\mu, \gamma\otimes i_G^\mu)$ which is
just the composition
of the homomorphism $(1_\H\otimes\id_A)\rtimes_\mu G$
of $A\rtimes_{\alpha,\mu}G$ into $\M\big((\K(\H)\otimes A)\rtimes_{\Ad\gamma\otimes \alpha, \mu}G\big)$
with the  left action of  $(\K(\H)\otimes A)\rtimes_{\Ad\gamma\otimes \alpha, \mu}G$ onto $\H\otimes A\rtimes_{\alpha,\mu}G$,
which,
 by Corollary \ref{cor-ext}, is given by the integrated form  of the
covariant homomorphism $(\id_{\K(\H)}\otimes i_A^\mu, \gamma\otimes i_G^\mu)$.

To see that this action factors through $A\rtimes_{\alpha, \red}G$ we first observe that
$C_\red^*(G)\otimes A\rtimes_{\alpha,\mu}G$ acts on
$\H\otimes A\rtimes_{\alpha,\mu}G$ via $\gamma\otimes \id_{A\rtimes_{\mu}G}$, since $\gamma$ is weakly contained in $\lambda_G$
by assumption. Moreover, the covariant homomorphism
$$(1_{C_\red^*(G)}\otimes i_A^\mu)\rtimes (\lambda_G\otimes i_G^\mu):A\rtimes_{\alpha,\un}G \to \M(C_\red^*(G)\otimes A\rtimes_{\alpha,\mu}G)$$
factors through $A\rtimes_{\alpha,\red}G$, since this is already true if we replace $A\rtimes_{\alpha,\mu}G$ by $A\rtimes_{\alpha,\un}G$ (compare the
discussion on KLQ-functors in \S\ref{sec:KLQ}). Now one  checks on generators that
$$\Psi= (\gamma\otimes \id_{A\rtimes_{\mu}G})\circ \big((1_{C_\red^*(G)}\otimes i_A^\mu)\rtimes (\lambda_G\otimes i_G^\mu)\big),$$
hence the result. Thus, replacing the left action of $A\rtimes_{\alpha,\mu}G$ by the left action of $A\rtimes_{\alpha,\red}G$, we obtain an element $x_A\in KK(A\rtimes_{\alpha,\red}G,A\rtimes_{\alpha,\mu}G)$ such that $1_{A\rtimes_{\mu}G}=(q_\red^\mu)^*(x_A)= [q_\red^\mu]\otimes_{A\rtimes_\red G}x_A$.
The converse equation $x_A\otimes_{A\rtimes_\un G}[q_\red^\mu]=1_{A\rtimes_\red G}$ follows word for word as
in \cite{Julg-Valette:Kamenable}*{Proposition 3.4} and we omit further details.
\end{proof}

\begin{remark} In case of a-T-menable groups a different line of argument could be used to obtain the
above proposition. In fact, it is shown in \cite{Higson-Kasparov:Etheory}*{Theorems 8.5 and 8.6} that the
elements $D\in KK^G(\mathcal B,\C)$ and $\eta\in KK^G(\C,\mathcal B)$ which implement the
validity of the strong Baum-Connes conjecture can  be chosen to be $KK^G$-equivalences.
Since for proper actions the full and reduced crossed products (and hence all exotic crossed products) coincide,
we can use the descents for $\rtimes_\mu$ and $\rtimes_{\un}$  to obtain the following chain of $KK$-equivalences for a given $G$\nb-\cstar{}algebra $(A,\alpha)$:
$$A\rtimes_{\mu}G\sim_{KK^G} (A\otimes \mathcal B)\rtimes_\mu G\cong (A\otimes \mathcal B)\rtimes_\un G\sim_{KK^G} A\rtimes_{\un}G.$$
One can deduce from this that the quotient maps
$$A\rtimes_{\un}G\onto A\rtimes_{\mu}G\onto A\rtimes_{\red}G$$
are  also $KK$-equivalences.
\end{remark}

For the special case of KLQ-functors, a similar line of argument as used above gives the following result:

\begin{theorem}\label{thm-KLQ}
Suppose that $\rtimes_{v_{KLQ}}$ is a KLQ-crossed product functor corresponding to the universally absorbing
representation $v: G\to \U\M(C_v^*(G))$. Suppose that the unit $1_G\in KK^G(\C,\C)$ is represented by
a cycle $(\H, \gamma, T)$ such that $\gamma$ is weakly contained in $v$. Then, for every system $(A,G,\alpha)$
the quotient map $q_v: A\rtimes_{\alpha,u}G\onto A\rtimes_{\alpha, v_{KLQ}}G$ is a $KK$-equivalence.
In particular, the quotient map $v: C^*(G)\to C_v^*(G)$  is a $KK$-equivalence.
\end{theorem}
\begin{proof} As in the proof of Theorem \ref{thm-Kamenable} we can represent the descent of $1_A\in KK^G(A,A)$
for the maximal crossed products by the cycle $(\H\otimes A\rtimes_{\alpha,u}G, \Psi, T\otimes 1)$
where $A\rtimes_{\alpha,u}G$ acts on $\H\otimes A\rtimes_{\alpha,u}G$ via the integrated form of the covariant homomorphism
$(1_\H\otimes i_A^u, \gamma\otimes i_G^u)$. Since $\gamma$ is weakly contained in $v$, this action can be written as
the composition of the representation $(1_{C_v^*(G)}\otimes i_A^u)\rtimes (v\otimes i_G^u)$ of $A\rtimes_{\alpha,u}G$
into $\M(C_v^*(G)\otimes A\rtimes_{\alpha,u}G)$ with
$\tilde\gamma \otimes \id_{A\rtimes_u G}: C_v^*(G)\otimes A\rtimes_{\alpha,u}G\to B(\H\otimes A\rtimes_{\alpha,u}G)$,
 where $\tilde\gamma: C_v^*(G)\to B(\H)$ denotes the unique factorisation to $C_v^*(G)$ of the integrated
 form of $\gamma$ -- which exists since $\gamma$ is weakly contained in $v$. By the definition of KLQ-functors,
the representation $(1_{C_v^*(G)}\otimes i_A^u)\rtimes (v\otimes i_G^u)$, and hence  also
$(1_\H\otimes i_A^u)\rtimes (\gamma\otimes i_G^u)$
factors through $A\rtimes_{\alpha, v_{KLQ}}G$, which gives an element
$x_A\in KK(A\rtimes_{\alpha, v_{KLQ}}G, A\rtimes_{\alpha,u}G)$ which is a $KK$-inverse
of the quotient map $q_v$.
\end{proof}

\begin{example} Suppose that $N$ is a closed normal subgroup of $G$ such that
$G/N$ is $K$-amenable. Let $v=\lambda_{G/N}\otimes u_G$ be the interior tensor product
of the regular representation of $G/N$ with the universal representation of $G$.
Then $v$ is universally absorbing and we may consider the corresponding KLQ-functor
$(A, G,\alpha)\mapsto A\rtimes_{v_{KLQ}}G$.
Since $G/N$ is $K$-amenable, there is a
representative $(\H,\gamma, T)$ of $1_{G/N}\in KK^{G/N}(\C,\C)$ such that $\gamma$ is weakly contained in
$\lambda_{G/N}$. Pulling this back to $G$ via the quotient map $q:G\to G/N$ shows that
$(\H,\gamma, T)$ is also a representative of $1_G\in KK^G(\C,\C)$ if we regard $\gamma$ as a representation of $G$.
Since $\gamma\prec\lambda_{G/N}= \lambda_{G/N}\otimes 1_G\prec \lambda_{G/N}\otimes u_G=v$
we see that $\gamma$ is weakly contained in $v$ and the theorem implies a $KK$-equivalence between
$A\rtimes_{\alpha,u}G$ and $A\rtimes_{\alpha, v_{KLQ}}G$ for any system $(A,G,\alpha)$.

More generally, suppose that $G$ is a locally compact group such that $1_G\in KK^G(\C,\C)$ is represented
by some Kasparov cycle $(\H,\gamma, T)$. Let $v=\gamma\otimes u_G$. Then $v$ is universally absorbing and the
theorem applies to the KLQ-functor $\rtimes_{v_{KLQ}}$.
\end{example}

\section{$L^p$ examples}\label{sec:lp}

Some of the most interesting examples of exotic crossed products come from conditions on the decay of matrix coefficients.  In this section, we use examples to explore some of the phenomena that can occur.  There is no doubt much more to say here.

For a locally compact group $G$ and $p\in [1,\infty]$, let $E_p{\sbe B(G)}$ denote the weak{*}-closure of the intersection $L^p(G)\cap B(G)$.  This is clearly an ideal in $B(G)$, and so gives rise to a KLQ-crossed product, and in particular a completion $C^*_{E_p}(G)$ of the group algebra.  Recall from \cite[Proposition 2.11 and Proposition 2.12]{Brown-Guentner:New_completions}\footnote{This reference only studies the discrete case, but the same proofs work for general locally compact groups.} that $C^*_{E_p}(G)$ is always the reduced completion for $p\leq 2$, and that $C^*_{E_p}(G)=C^*_\un(G)$ for some $p<\infty$ if and only if $G$ is amenable. Of course, since $E_\infty=B(G)$, we always get $C_{E_\infty}^*(G)=C^*_\un(G)$.

Similarly, let $E_0\sbe B(G)$ denote the weak{*-}closure of $C_0(G)\cap B(G)$, which is also an ideal. For countable discrete groups, Brown and Guentner \cite[Corollary 3.4]{Brown-Guentner:New_completions} have shown that $E_0=B(G)$ if and only if $G$ is a-T-menable.  In \cite{Jolissaint} Jolissaint extended this to general second countable\footnote{The techniques of Brown and Guentner, and of Jolissaint, in fact show that $E_0=B(G)$ if and only if $G$ has the \emph{Haagerup approximation property} in the sense of \cite[1.1.1 (2)]{CCJJV} without any second countability assumptions.  The Haagerup approximation property is equivalent to all the other possible definitions of a-T-menability in the second countable case, but not in general: see  \cite[Section 1.1.1 and Theorem 2.1.1]{CCJJV} for more details on this.} locally compact groups.

The following theorem records some properties of the algebras $C^*_{E_p}(G)$ for the free group on two generators.

\begin{theorem}\label{f2lp}
Let $G={\Free_2}$ be the free group on two generators.  Then one has the following results.
\begin{enumerate}
\item For $p\in [{2},\infty]$ the ideals $E_p$, or equivalently the $C^*$-algebras $C^*_{E_p}(G)$, are all different (in the sense that the identity on $C_c(G)$ does not extend to isomorphisms of these algebras for different parameters $p$).
\item The union of the ideals $E_p$ for $p\in [2,\infty)$ is weak*-dense in $B(G)=E_0=E_\infty$.
\item The $C^*$-algebras $C^*_{E_p}(G)$, $p\in [{2},\infty]$, all have the same $K$-theory (in the strong sense that the canonical quotient maps between them induce $KK$-equivalences).
\end{enumerate}
\end{theorem}

\begin{proof}
Properties (1) and (2) follow from Okayasu's work \cite{Okayasu:Free-group}.  Property (3) is a consequence of Theorem \ref{thm-Kamenable} since
${\Free_2}$ is a-T-menable and the  corresponding KLQ-functors are correspondence functors.
\end{proof}

It seems interesting to ask exactly which other groups have properties (1), (2) and (3) from Theorem \ref{f2lp}.  We now discuss this.

First, we look at property (1).   If a group $G$ has this property, then clearly $G$ must be non-amenable, otherwise all the spaces $E_p$, $p\in \{0\}\cup [1,\infty]$, are the same.  It is easy to see that property (1) for $\Free_2$ implies that this property holds for all discrete groups that contain ${\Free_2}$ as a subgroup.  It is thus quite conceivable that it holds for all non-amenable discrete groups.

Property (1) also holds for some non-discrete groups: it was observed recently  by Wiersma (see \cite{Wiersma}*{\S 7}), using earlier results of Elias Stein, that (1)  holds for $SL(2,\mathbb R)$. Indeed,  it is shown in \cite{Wiersma}*{Theorem 7.3} that $p\mapsto E_p=\overline{B(G)\cap L^p(G)}^{w*}$ is a bijection between $[2,\infty]$ and the set of  weak{*-}closed ideals $E\subseteq B(G)$ for $G=SL(2,\mathbb R)$.  However, property (1) does not hold {for} all non-amenable {locally compact groups} as the following example shows.

\begin{example}\label{ex:nodiff}
Let $G$ be a non-compact simple Lie group with finite center and rank at least two.  It is well-known that such a group is non-amenable.  Cowling \cite[page 233]{Cowling}\footnote{See \cite[Corollary 3.3.13]{Howe} for more detail in the special case when $G=SL(n,\Real)$, $n\geq 3$.} has shown that there exists $p\in (2,\infty)$ (depending on $G$) such that all non-trivial irreducible unitary representations of $G$ have all their matrix coefficients in $L^{p}(G)$.  Hence for all $q\in [p,\infty)$ we have that $C^*_{E_q}(G)=C^*_{E_{p}}(G)$.  Indeed, this follows as the sets of irreducible representations of these $C^*$-algebras are the same: they identify canonically with the set of all non-trivial irreducible unitary representations of $G$.
\end{example}

We now look at property (2).  As written, the property can only possibly hold for a-T-menable groups.  Simple a-T-menable Lie groups of rank one (\ie\ $SU(n,1)$ and $SO(n,1)$) have the property by already cited results of Cowling \cite[page 233]{Cowling}.  For discrete groups, one has the following sufficient condition for property (2) to hold.

\begin{lemma}\label{lem:e0}
Say $G$ is a finitely generated discrete group, and fix a word length $l:G\to \N$.  Assume moreover there exists a negative type function $\psi:G\to \Real_+$ such that $\psi(g)\geq c\cdot l(g)$ for some $c>0$ and all $g\in G$.  Then the union $\cup E_p$ is weak{*-}dense in $B(G)$.
\end{lemma}

\begin{proof}
Let $\phi$ be any element of $B(G)$.  For $t>0$ let $\phi_t(g)=e^{-t\psi(g)}$, which is positive type by Schoenberg's theorem.  For any fixed $t>0$ and all $p$ suitably large, the function $\phi_t$ will be in $l^p(G)$: indeed this follows from the estimate on $\psi$ and as there exists $b>0$ such that $|\{g\in G~:~l(g)\leq r\}|\leq e^{br}$ for any $r>0$ (the latter fact is an easy consequence of finite generation).  Hence the functions $\phi\cdot \phi_t$ are all in some $E_p$ (where $p$ will in general vary as $t$ does), and they converge pointwise, whence weakly, to $\phi$.
\end{proof}
The condition in the lemma is closely related to the \emph{equivariant compression} of $G$ \cite{Guentner-Kaminker: equi}, and has been fairly well-studied.  It is satisfied for example by finitely generated free groups, but also much more generally: for example, it can be easily deduced from \cite{Niblo-Campbell:cat} that any group acting properly cocompactly by isometries on a CAT(0) cube complex has this property, and from \cite[Section 2.6]{Bekka-delaHarpe-Valette} that any group acting properly cocompactly by isometries on a real hyperbolic space does.  Note, however, that \cite{Austin:comp} implies that the lemma only applies to a subclass of finitely generated a-T-menable groups.

Finally, we look at property (3).  By Theorem \ref{thm-Kamenable} it holds for all second countable $K$-amenable groups (in particular, all a-T-menable groups).  The following example shows that it fails in general.

\begin{example}\label{ex:spn1}
Say $G=SP(n,1)$, and $n\geq 3$.  Prudhon \cite{Prudhon:sp} has shown that there are infinitely many irreducible representations $(\pi,\H_\pi)$ of $G$ -- the so-called \emph{isolated series} representations -- such that there is a direct summand of $C^*(G)$ isomorphic to $\mathcal{K}(\H_\pi)$, and moreover that the corresponding direct summand $K_0(\mathcal{K}(\H_\pi))\cong \Z$ of the group $K_0(C^*(G))$ is in the kernel of the natural map $\lambda_*:K_0(C^*(G))\to K_0(C^*_\red(G))$.  It follows from the previously cited work of Cowling \cite[page 233]{Cowling} that any non-trivial irreducible representation of $SP(n,1)$ extends to $C^*_{E_p}(G)$ for some $p<\infty$.  Combining all of this, it follows that there exists $p\in (2,\infty)$ such that $C^*_{E_p}(G)$ has a direct summand isomorphic to $\K(\H)$ for some separable Hilbert space $\H$, and that the corresponding direct summand $K_0(\mathcal{K}(\H))\cong \Z$ of $K_0(C^*_{E_p}(G))$ is in the kernel of the map on $K$-theory induced by the natural quotient $C^*_{E_p}(G)\to C^*_{E_2}(G)$.  Hence property (3) fails for $SP(n,1)$ for $n\geq 3$.

One can deduce from work of Pierrot \cite{Pierrot:sl} that property (3) fails for $SL(4,\C)$ and $SL(5,\C)$ in a similar way.
\end{example}

The work of Prudhon and Pierrot cited above uses detailed knowledge of the representation theory of the groups involved.  It would be interesting to have more directly accessible examples, or techniques that showed similar results for discrete groups.  For example, it is natural to guess that similar failures of property (3) occur for discrete hyperbolic groups with property (T), and for lattices in higher rank simple Lie groups; proving such results seems to require new ideas, however.

\section{Morita compatibility and the minimal exact correspondence functor}\label{sec:minimal}

We now want to relate our results to the constructions of \cite{Baum-Guentner-Willett}, in which a new
version of the Baum-Connes conjecture is discussed. In fact, the authors consider a certain minimal exact and
Morita compatible (in a sense explained below)
crossed-product functor, which they call $\rtimes_{\E}$, and the new version of the Baum-Connes conjecture
formulated by Baum, Guentner and Willett in that paper  asserts that the Baum-Connes assembly map should always
be an isomorphism if we replace  reduced crossed products by  $\rtimes_{\E}$ crossed products.
We should note that if $G$ is exact, then $\rtimes_{\E}= \rtimes_{\red}$, so that in this case  the new conjecture coincides
with the old one. However, all known counter examples for the original conjecture
are due to the existence of non-exact groups and cannot be shown to be counterexamples to the reformulated conjecture. The most remarkable result of \cite{Baum-Guentner-Willett}
shows that many such examples actually do satisfy the reformulated conjecture.

Note that  Baum, Guentner and Willett also constructed a direct assembly map for the crossed-product functor $\rtimes_{\E}$
where they used the descent in $E$-theory instead of $KK$-theory, since at the time of writing it was not clear to them
whether the $KK$-theory descent exists
(\eg, see the open question \cite[Question 8.1 (vii)]{Baum-Guentner-Willett}). In this section we show that
this is indeed the case, so that $KK$-methods do apply to the reformulation of  the Baum-Connes conjecture.

We start the discussion by showing that the projection property is inherited
by the infimum of a collection of crossed-product functors which satisfy this property.
Recall from \cite[Lemma 3.7]{Baum-Guentner-Willett} that, starting with a collection of crossed-product functors
$\{\rtimes_\mu: \mu\in \Sigma\}$ (where we {\em do not} assume  that $\Sigma$ will be a set),
a new crossed-product functor $(A,\alpha)\mapsto A\rtimes_{\alpha, \mu_{\inf}}G$
can be constructed, which should be understood as the {\em infimum} of the functors in the collection
$\{\rtimes_\mu: \mu\in \Sigma\}$.
The crossed product $ A\rtimes_{\alpha, \mu_{\inf}}G$ is the unique quotient of $A\rtimes_{\alpha,\un}G$ such that
the set $(A\rtimes_{\alpha,\mu_{\inf}}G)\dach$ of equivalence classes of irreducible representations
 (denoted $S_{\mu_{\inf}}(A)$ in \cite{Baum-Guentner-Willett})
is given by the formula
\begin{equation}\label{eq-inf-dach}
(A\rtimes_{\alpha,\mu_{\inf}}G)\dach=\bigcap_{\mu\in \Sigma} (A\rtimes_{\alpha,\mu}G)\dach,
\end{equation}
where we view each dual space $(A\rtimes_{\alpha,\mu}G)\dach$ as a subset of $(A\rtimes_{\alpha,\un}G)\dach$
via composition with the quotient maps $q_{A,\mu}:A\rtimes_{\alpha,\un}G\to A\rtimes_{\alpha,\mu}G$.
 Alternatively, one may define
$A\rtimes_{\alpha, \mu_{\inf}}G$ as the quotient $(A\rtimes_{\alpha,\un}G)/I_{\mu_{\inf}}$ with
$$I_{\mu_{\inf}}=\overline{\sum_{\mu\in \Sigma} I_\mu},$$
the closed sum of all ideals $I_\mu:=\ker q_{A,\mu}$,  $\mu\in \Sigma$. Note that this does not cause
any set theoretic problems, since the collection of ideals $\{I_\mu: \mu\in \Sigma\}$ is a set.

It has been shown in \cite[Lemma 3.7]{Baum-Guentner-Willett} that $(A,\alpha)\mapsto A\rtimes_{\alpha, \mu_{\inf}}G$ is indeed
a crossed-product functor. Moreover, it is shown in \cite[Theorem 3.8]{Baum-Guentner-Willett} that taking the infimum of
 a collection of exact crossed-product functors will again give an exact crossed-product functor.
 We now show

 \begin{proposition}\label{prop-inf-corr}
Let  $\{\rtimes_\mu:\mu\in \Sigma\}$ be a collection of correspondence crossed-product functors.
Then its infimum $\rtimes_{\mu_{\inf}}$ will be a correspondence functor as well.
\end{proposition}

We need a lemma:

\begin{lemma}\label{lem-hereditary}
Suppose that $(A,\alpha)$ is a $G$\nb-algebra and $B\subseteq A$ is a $G$\nb-invariant hereditary subalgebra.
For each covariant representation $(\pi,U)$ of $(A,G,\alpha)$ on some Hilbert space (or Hilbert module) $\H$
let us write  $(\pi|_B, U)$ for the covariant representation of $(B,G,\alpha)$ which acts on the
subspace (or submodule) $\H|_B:=\pi(B)\H$ via the restriction $\pi|_B$ and the given unitary representation $U$.
Then
\begin{equation}\label{eq-hereditary-rep}
(B\rtimes_{\alpha,\un}G)\dach=\{\pi|_B\rtimes U: \pi\rtimes U\in (A\rtimes_{\alpha,\un}G)^\dach, \pi(B)\neq\{0\}\}.
\end{equation}
In particular, if $p\in \M(A)$ is a $G$\nb-invariant projection, then
\begin{equation}\label{eq-projection-rep}
(pAp\rtimes_{\alpha,\un}G)\dach=\{\pi|_{pAp}\rtimes U: \pi\rtimes U\in (A\rtimes_{\alpha,\un}G)^\dach, \pi(p)\neq 0\}.
\end{equation}
\end{lemma}
\medskip

\begin{proof} We first consider the ideal $I:=\overline{ABA}$. It is then well known that
$$(I\rtimes_{\alpha,\un}G)\dach=\{\pi|_I\rtimes U: \pi\rtimes U\in (A\rtimes_{\alpha,\un}G)\dach, \pi(I)\neq \{0\}\}.$$
Note that we then have $\pi(I)\H=\H$ if $\pi(I)\neq \{0\}$, since $\pi\rtimes U$ is irreducible and $\pi(I)\H$
is $\pi\rtimes U$-invariant.

Note that $\pi(ABA)=\pi(A)\pi(B)\pi(A)=\{0\}$ if and only if $\pi(B)=\{0\}$, which
follows by approximating $b\in B$ by $a_i b a_i$, where $(a_i)$ runs through a bounded approximate unit of $A$.
Using this and the above observation for the  ideal $I=\overline{ABA}$, we may now assume that
$B$ is full in the sense that $\overline{ABA}=A$ and to show that in this situation we have
$$(B\rtimes_{\alpha,\un}G)\dach=\{\pi|_B\rtimes U : \pi\rtimes U\in (A\rtimes_{\alpha,\un}G)\dach\}.$$
For this we observe that $(BA,\alpha)$ is a $G$\nb-equivariant $B-A$ equivalence  bimodule with
$B$- and $A$-valued inner products given by $_B\braket{c}{d}=cd^*$ and $\braket{c}{d}_A=c^*d$
for $c,d\in BA$. Thus, $BA$ induces a bijection between the
equivalence classes of  irreducible covariant representations
of  $(A,G,\alpha)$ to those of $(B,G,\alpha)$ by sending a representation $(\pi,U)$ to the
$BA$-induced representation $(\Ind^{BA}\pi, \Ind^{BA}U)$  on the Hilbert space $(BA)\otimes_{\pi}\H$
on which $\Ind^{BA}\pi$ is given by the left action of $B$ on the first factor $BA$ and
 $\Ind^{BA}U=\alpha\otimes U$
(\eg, see \cite{Combes:Crossed_Morita,Echterhoff:Morita_twisted}). Now observe that there is a unique unitary operator
$V:  BA\otimes_{\pi}\H\to \pi(B)\H$ given on
elementary tensors by sending $ba\otimes \xi$ to $\pi(ba)\xi$,  which clearly intertwines
$(\Ind^{BA}\pi, \Ind^{BA}U)$ with $(\pi|_B, U)$. This proves (\ref{eq-hereditary-rep}) and (\ref{eq-projection-rep}) then
follows from the fact that $\pi(pAp)\neq 0$ if and only if $\pi(p)\neq 0$.
\end{proof}

\begin{corollary}\label{cor-projection}
A given crossed-product functor
 $(A,\alpha)\to A\rtimes_{\alpha,\mu}G$ satisfies the projection property (hence is a correspondence functor)
if and only if
\begin{equation}\label{eq-projection-rep1}
(pAp\rtimes_{\alpha,\mu}G)^\dach=\{\pi|_{pAp}\rtimes U: \pi\rtimes U\in (A\rtimes_{\alpha,\mu}G)\dach, \pi(p)\neq 0\}.
\end{equation}
holds for every $G$\nb-invariant projection $p\in \M(A)$.
\end{corollary}
\begin{proof}
Let $\tilde{p}$ denote the image of $p$ in $\M(A\rtimes_\mu G)$ under the canonical map.
Then $\tilde{p}(A\rtimes_\mu G)\tilde{p}$ is a corner of
 $A\rtimes_{\alpha,\mu}G$ which implies that
$$(\tilde{p}(A\rtimes_\mu G)\tilde{p})\dach=\{(\pi\rtimes U)|_{\tilde{p}(A\rtimes_\mu G)\tilde{p}}: \pi\rtimes U\in (A\rtimes_{\alpha,\mu}G)\dach, \pi\rtimes U(\tilde{p})\neq 0\}. $$
A computation on the level of functions in $C_c(G, pAp)$  shows that
$(\pi\rtimes U)|_{\tilde{p}(A\rtimes_\mu G)\tilde{p}}$ is the integrated form
of $(\pi|_{pAp}, U)$.  Since $\pi\rtimes U(\tilde{p})=\pi(p)$ the homomorphism
$$\iota\rtimes G: pAp\rtimes_{\alpha,\mu}G \onto \tilde{p}(A\rtimes_{\mu}G)\tilde{p}\subseteq A\rtimes_{\alpha,\mu}G$$
will be faithful if and only if equation (\ref{eq-projection-rep1}) holds.
\end{proof}

\begin{proof}[Proof of Proposition \ref{prop-inf-corr}]
By Theorem \ref{thm-corr} it suffices to show that if all functors $\rtimes_\mu$ satisfy the
projection property, the same holds for $\rtimes_{\mu_{\inf}}$.
So we may assume that equation (\ref{eq-projection-rep1}) holds for each $\mu$.  But then
it also holds for $\mu_{\inf}$, since for every $G$\nb-invariant projection $p\in \M(A)$ we have
\begin{align*}
(pAp\rtimes_{\alpha,\mu_{\inf}}G)\dach &=
\bigcap_{\mu\in \Sigma} (pAp\rtimes_{\alpha,\mu}G)\dach \\
&=\bigcap_{\mu\in \Sigma}
\{\pi|_{pAp}\rtimes U: \pi\rtimes U\in (A\rtimes_{\alpha,\mu}G)\dach, \pi(p)\neq 0\}\\
&=\{\pi|_{pAp}\rtimes U: \pi\rtimes U\in (A\rtimes_{\alpha,\mu_{\inf}}G)\dach, \pi(p)\neq 0\}.
\end{align*}
Hence the result follows from Corollary \ref{cor-projection}.
\end{proof}

The following corollary gives a counterpart to the minimal exact Morita compatible
crossed-product functor $(A,\alpha)\mapsto A\rtimes_{\alpha,\E}G$ as constructed
in \cite{Baum-Guentner-Willett}. Since it is a correspondence functor, it allows
a descent in $KK$-theory and the results of  \S\ref{sec:descent} will apply.
We shall see below that, at least for
second countable groups
and on the category of separable
$G$\nb-algebras, the functor
$\rtimes_\E$ coincides with the minimal exact correspondence functor $\rtimes_{\E_\Cor}$ of

\begin{corollary}\label{cor-minimal-exact}
For every locally compact group $G$ there exists a smallest exact correspondence crossed-product functor
$(A,\alpha)\mapsto A\rtimes_{\alpha, \E_\Cor}G$.
\end{corollary}
\begin{proof}
We simply define $\rtimes_{\E_\Cor}$ as the infimum of the collection of all exact correspondence functors
$\rtimes_\mu$. By Proposition \ref{prop-inf-corr} it will be a correspondence functor and by
\cite[Theorem 3.8]{Baum-Guentner-Willett} it will be exact.
\end{proof}

The minimal exact correspondence functor $\rtimes_{\E_{\Cor}}$ enjoys
the following property:

\begin{corollary}\label{cor-min-corr}
Let $(D,\delta)$ be any unital $G$\nb-\cstar{}algebra. Then for each  $G$\nb-\cstar{}algebra $(A,\alpha)$
the inclusion $j_A: A\to A\otimes_{\max} D;\; a\mapsto a\otimes 1$  descends to an injective
\Star{}homomorphism $j_A\rtimes_{\E_{\Cor}}G:A\rtimes_{\E_{\Cor}}G\to (A\otimes_{\max}D)\rtimes_{\E_{\Cor}}G$.

In other words, we have $\rtimes_\mu =\rtimes_{\mu_D^{\max}}$ if  $\mu=\E_{\Cor}$.
If $D$ is exact, the same holds if we replace $\otimes_{\max}$ by $\otimes_{\min}$.
\end{corollary}
\begin{proof}
This follows from minimality of $\rtimes_{\E_\Cor}$ and the (easily verified) fact that $\rtimes_{\mu_D^\nu}$ is exact if $\rtimes_\mu$ is exact and either: $\nu=\max$; or both $\nu=\min$ and $D$ is exact.
\end{proof}

In order to see that the functor $\rtimes_{\E_\Cor}$ coincides with the ``minimal exact Morita compatible functor''
$\rtimes_\E$, we need to recall the notion of Morita compatibility as defined in \cite{Baum-Guentner-Willett} and then
compare it with the properties of our correspondence functors.
If $G$ is a locally compact group we write $\K_G:=\K(\ell^2(\N))\otimes \K(L^2(G))$ equipped with
the $G$\nb-action $\Ad \Lambda$, where $\Lambda=1\otimes\lambda:G\to \U(\ell^2(\N)\otimes L^2(G))$
and $\lambda$ denotes the left
regular representation of $G$. Then there is a canonical isomorphism of maximal crossed products
$$\Phi: (A\otimes\K_G)\rtimes_{\alpha\otimes \Ad\Lambda,\un}G\stackrel{\cong}{\longrightarrow} (A{\rtimes}_{\alpha,\un}G)\otimes\K_G$$
given by the integrated form of the covariant homomorphism $(i_A\otimes \id_{\K_G}, i_G\otimes \Lambda)$
of $(A\otimes\K_G, G, \alpha\otimes \Ad\Lambda)$ into $\M((A{\rtimes_{\alpha,\un}}G)\otimes\K_G)$.
Then Baum, Guentner and Willett defined a crossed-product functor to be {\em Morita compatible}
if the composition of $\Phi$ with the quotient map $(A{\rtimes_{\alpha,\un}}G)\otimes\K_G\to (A{\rtimes_{\alpha,\mu}}G)\otimes\K_G$
factors through an isomorphism
$$\Phi_\mu: (A\otimes\K_G)\rtimes_{\alpha\otimes \Ad\Lambda,\mu}G\stackrel{\cong}{\longrightarrow} (A{\rtimes_{\alpha,\mu}}G)\otimes\K_G.$$

\begin{proposition}\label{prop-Morita}
Suppose that $G$ is a locally compact group and that $(A,\alpha)\mapsto A\rtimes_{\alpha,\mu}G$ is
a crossed-product functor on the category of $\sigma$-unital \cstar{}algebras. Then the following
three conditions are equivalent:
\begin{enumerate}
\item $\rtimes_{\mu}$ is Morita compatible as defined above.
\item $\rtimes_\mu$ has the full projection property.
\item $\rtimes_\mu$ is strongly Morita compatible.
\end{enumerate}
\end{proposition}

Note that the assumption that all $G$-algebras are $\sigma$-unital is only used in the proof (1) $\Rightarrow$ (2). All other
statements hold in full generality.

\begin{proof} The proof of (2) $\Leftrightarrow$ (3) is exactly the same as the proof of (2) $\Leftrightarrow$ (4) in Lemma \ref{lem-Morita}
and we omit it here. The proof of (3) $\Rightarrow$ (1) follows from Corollary \ref{cor-ext}.

(1) $\Rightarrow$ (2): Let $(A,\alpha)$ be a $\sigma$-unital $G$\nb-algebra and let $p\in \M(A)$ be a $G$\nb-invariant full projection.
We need to show that the inclusion $\iota: pAp\hookrightarrow A$ descends to give a faithful  inclusion
$\iota\rtimes G: pAp\rtimes_{\mu}G\hookrightarrow A\rtimes_{\mu}G$. This will be the case if and only if
$$(\iota\rtimes G)\otimes\id_{\K_G}: pAp\rtimes_{\mu}G\otimes\K_G\to A\rtimes_\mu G\otimes \K_G$$
is faithful, and we shall now deduce from Morita compatibility that this is the case.

Let $\alpha\otimes \Ad\Lambda$ be the corresponding action of $G$ on $A\otimes\K_G$.
It follows  from \cite[Corollary~2.6]{Mingo-Phillips:Triviality} that there exists a $G$\nb-invariant  partial isometry $v\in \M(A\otimes \K_G)$
such that $v^*v=1_A\otimes 1_{\K_G}$ and $vv^*=p\otimes 1_{\K_G}$. This implies that
the mapping
$$\Ad v^*: pAp\otimes \K_G\to A\otimes \K_G;  \;(a\otimes k)\mapsto v^*(a\otimes k)v; $$
is a $G$\nb-equivariant isomorphism. Hence it induces an isomorphism
$$\Ad v^*\rtimes G: (pAp\otimes \K_G)\rtimes_{\mu}G\stackrel{\cong}{\longrightarrow} (A\otimes \K_G)\rtimes_{\mu}G.$$
Let  $\tilde{v}$ denote the image of $v$ in $\M((A\rtimes_{\alpha,\mu}G)\otimes\K_G)$ under the composition
 $\Phi_\mu\circ
i_{A\otimes\K_G}$, where $i_{A\otimes\K_G}:\M(A\otimes\K_G)\to \M((A\otimes\K_G)\rtimes_{\mu}G))$ denotes the canonical map.
Then $\tilde{v}$ is a partial isometry with $\tilde{v}^*\tilde{v}=1_{A\rtimes G\otimes\K}$ and $\tilde{v}\tilde{v}^*=\tilde{p}\otimes 1_{\K_G}$
where $\tilde{p}$ denotes the image of $p$ in $\M(A\rtimes_\mu G)$.
Consider the diagram
$$
\xymatrix{
(pAp\rtimes_{\alpha,\mu} G)\otimes \K_G \ar[r]^{\iota\rtimes G\otimes \id_\K}& \tilde{p}(A\rtimes_{\alpha,\mu} G)\tilde{p}\otimes \K_G
\ar[r]^{\Ad{\tilde{v}^*}}& (A\rtimes_{\alpha,\mu} G)\otimes \K_G\\
(pAp\otimes\K_G)\rtimes_{\alpha\otimes\Ad\Lambda,\mu}G \ar[u]^{\Phi_\mu} \ar[rr]^{\Ad v^*\rtimes G} &  &
(A\otimes\K_G)\rtimes_{\alpha\otimes\Ad\Lambda,\mu}G \ar[u]^{\Phi_\mu}
}
$$
One easily checks on the generators that this diagram commutes. All maps but $\iota\rtimes G\otimes \id_\K$ are known
to be isomorphisms, hence $\iota\rtimes G\otimes \id_\K$ must be an isomorphism as well.
\end{proof}

\begin{lemma}\label{extsep}
Let $G$ be a second countable group, and say $\rtimes_\mu$ is a crossed-product functor, defined on the category of {\em separable} $G$\nb-\cstar{}algebras.  For each $G$\nb-\cstar{}algebra $A$, let $(A_i)_{i\in I}$ be the net of $G$-invariant separable sub $C^*$-algebras of $A$ ordered by inclusion.  Define $A\rtimes_{\mathrm{sep},\mu}G$ by
$$
A\rtimes_{\mathrm{sep},\mu}G:=\lim_{i\in I}A_i\rtimes_\mu G.
$$
Then $A\rtimes_{\mathrm{sep},\mu}G$ is a completion of $C_c(G,A)$ with the following properties.
\begin{enumerate}
\item $(A,\alpha)\mapsto A\rtimes_{\alpha,\mathrm{sep},\mu}G$ is  a crossed-product functor for $G$.
\item If $\rtimes_\mu$ has the ideal property, the same holds for $\rtimes_{\mathrm{sep},\mu}$.
\item If $\rtimes_\mu$ is a correspondence  functor, the same holds for $\rtimes_{\mathrm{sep},\mu}$.
\item If $\rtimes_\mu$ is exact, the same holds for $\rtimes_{\mathrm{sep},\mu}$.
\item If $\rtimes_\mu$ is Morita compatible, then the same holds for $\rtimes_{\mathrm{sep},\mu}$.
\item If $\rtimes_\nu$ is any crossed-product functor extending $\rtimes_\mu$ to all $G$\nb-\cstar{}algebras, then $\rtimes_{\mathrm{sep},\mu}\geq \rtimes_\nu$.
\end{enumerate}
\end{lemma}

\begin{proof}
Note first that as $G$ is second countable any $f\in C_c(G,A)$ has image in some second countable $G$-invariant subalgebra of $A$, and thus $A\rtimes_{\mathrm{sep},\mu}G$ contains $C_c(G,A)$ as a dense subset as claimed.  It is automatically smaller than the maximal completion, and it is larger than the reduced as the canonical maps
$$
A_i\rtimes_\mu G \to A_i\rtimes_\red G \to A\rtimes_\red G
$$
give a compatible system, whence there is a map $A\rtimes_{\textrm{sep},\mu}G \to A\rtimes_\red G$ (which is equal to the identity on $C_c(G,A)$) by the universal property of the direct limit.
\medskip
\\
(1) Functoriality follows from functoriality of the direct limit construction.
\medskip
\\
(2) If $I$ is an ideal in $A$, then for each separable $G$-invariant $C^*$-subalgebra $A_i$ of $A$ as in the definition of $\rtimes_{\mathrm{sep},\mu}$, define $I_i=I\cap A_i$.  Then the directed system $(I_i)_{i\in I}$ is cofinal in the one used to define $I\rtimes_{\mathrm{sep},\mu}G$, whence
$$
I\rtimes_{\mathrm{sep},\mu}G=\lim_i(I_i\rtimes_\mu G).
$$
The result now follows as each map $I_i\rtimes_\mu G\to A_i\rtimes_\mu G$ is injective by the ideal property for $\rtimes_\mu$, and injectivity passes to direct limits.
\medskip
\\
(3) We prove the projection property.  Let $p$ be a $G$-invariant projection in the multiplier algebra of $A$.  The argument is similar to that of part (2).  The only additional observation needed is that if we set $B_i$ to be the $C^*$-subalgebra of $A$ generated by $A_ip$ and $A_i$, then the net $(B_i)_{i\in I}$ is cofinal in the net defining $A\rtimes_{\mathrm{sep},\mu}G$.  Moreover, $p$ is in the multiplier algebra of each $B_i$, whence each of the injections $pB_ip\to B_i$ induces an injection on $\rtimes_\mu$ crossed products by the projection property for $\rtimes_\mu$.  The result follows now on passing to the direct limit of the inclusions
$$
pB_ip\rtimes_\mu G \to B_i \rtimes_\mu G.
$$
(4) This is similar to parts (2) and (3).  Given a short exact sequence
$$
0\to I\to A \to B\to 0
$$
with quotient map $\pi:A\to B$, let $(A_i)$ be the net of separable $G$-invariant subalgebras  of $A$, and consider the net of short exact sequences
$$
0\to A_i\cap I \to A_i \to \pi(A_i)\to 0.
$$
The nets $(A_i\cap I)$ and $(\pi(A_i))$ are cofinal in the nets defining the $\rtimes_{\mathrm{sep},\mu}$ crossed products for $I$ and $B$ respectively.  On the other hand, exactness of $\mu$ gives that all the sequences
$$
0\to (A_i\cap I )\rtimes_\mu G\to A_i \rtimes_\mu G \to \pi(A_i)\rtimes_\mu G\to 0
$$
are exact.  The result follows as exactness passes to direct limits.
\medskip
\\
(5) This again follows a similar pattern: the only point to observe is that the collection of separable $G$-invariant $C^*$-algebras of $A\otimes \mathcal{K}_G$ of the form $A_i\otimes \mathcal{K}_G$, where $A_i$ is a separable $G$-invariant $C^*$-subalgebra of $A$, is cofinal in the direct limit defining $(A\otimes \mathcal{K}_G)\rtimes_{\mathrm{sep},\mu}G$.
\medskip
\\
(6) From functoriality of $\rtimes_\nu$, there is a compatible system of \Star{}homomorphisms
$$
A_i\rtimes_{\mu} G =A_i\rtimes_{\nu} G \to A\rtimes_\nu G.
$$
From the universal property of the direct limit there is a \Star{}homomorphism $A\rtimes_{\mathrm{sep},\mu}G \to A\rtimes_\nu G$ which is clearly the identity on $C_c(G,A)$, completing the proof.
\end{proof}

\begin{example}\label{sepex}
Every non-amenable second countable group has a crossed product $\rtimes_{\mu}$ for which
the crossed-product functor $\rtimes_{\mathrm{sep},\mu}$ (in which we restrict $\rtimes_\mu$ to the category
of separable $G$-algebras) does not coincide with $\rtimes_{\mu}$.
 For this, we use the construction of \S\ref{subsec:counter}.  Let $\H$ be a Hilbert space with uncountable Hilbert space dimension.
Let $\mathcal  S$ denote the set of all separable subalgebras of $\K(\H)$, all  equipped with the trivial $G$-action.
We define a crossed-product functor $\rtimes_{\mu}$ by defining $A\rtimes_{\mu}G$ as the completion
of $C_c(G,A)$ by the norm
$$\| f\|_\mu:=\max\big\{ \|f\|_\red, \sup\{\|\phi\circ f\|_{B\rtimes_{\un}G} : B\in\mathcal S, \phi\in \mathrm{Hom}^G(A,B)\}\big\},$$
where $\mathrm{Hom}^G(A,B)$ denotes the set of $G$-equivariant \Star{}homomorphisms from $A$ to $B$.
As explained in \S\ref{subsec:counter}, this is a functor.

Now consider $A=\K(\H)$ with the trivial $G$-action. Then
$\K(\H)\rtimes_{\mu}G=\K(\H)\rtimes_{\red}G\cong \K(\H)\otimes C_\red^*(G)$, since there are no nontrivial homomorphisms
from $\K(\H)$ into any of its separable subalgebras. On the other hand we have
$B_i\rtimes_{\mu}G\cong B\rtimes_{\un}G$ for every separable subalgebra $B\subseteq \K(\H)$, which
implies $\K(\H)\rtimes_{\mathrm{sep},\mu}G\cong \K(\H)\rtimes_{\un}G\cong \K(\H)\otimes C^*(G)$.
\end{example}

\begin{corollary}
The following functors from the category of separable $G$\nb-\cstar{}algebras to the category of \cstar{}algebras are the same.
\begin{enumerate}
\item The restriction to separable $G$\nb-algebras of the minimum $\rtimes_{\E_\Cor}$ over all exact correspondence functors defined for all
$G$\nb-\cstar{}algebras.
\item The restriction to separable $G$\nb-algebras of the minimum $\rtimes_{\E}$ over all exact Morita compatible functors defined for all
$G$\nb-\cstar{}algebras.
\item The minimum over all exact correspondence functors defined for separable $G$\nb-\cstar{}algebras.
\item The minimum over all exact Morita compatible functors defined for separable $G$\nb-\cstar{}algebras.
\end{enumerate}
\end{corollary}

\begin{proof}
Label the crossed products (defined on separable $G$\nb-\cstar{}algebras) appearing in the points above as $\rtimes_1$, $\rtimes_2$, $\rtimes_3$ and $\rtimes_4$.  By (2) $\Rightarrow$ (1) of Proposition \ref{prop-Morita} (which holds in full generality) we clearly have
$$
\rtimes_1\geq \rtimes_2\geq \rtimes _4, \text{ and } \rtimes_1\geq \rtimes_3\geq \rtimes_4.
$$
It thus suffices to show that $\rtimes_4\geq \rtimes_1$.  Indeed, let $\rtimes_{\mathrm{sep},4}$ be the extension given by Lemma \ref{extsep}.  Then $\rtimes_{\mathrm{sep},4}$ is an exact and Morita compatible functor on the category of $\sigma$-unital $G$\nb-\cstar{}algebras, whence it is an exact correspondence functor on this category by Proposition \ref{prop-Morita}.  Hence in particular $\rtimes_4$ was actually an exact correspondence functor on the category of separable $G$\nb-algebras to begin with, and so $\rtimes_{\mathrm{sep},4}$ is an exact correspondence functor on the category of all $G$\nb-algebras.  It is thus one of the functors that $\rtimes_1$ is the minimum over, so $\rtimes_1\leq \rtimes_{\mathrm{sep},4}$ on the category of all $G$\nb-algebras, and in particular $\rtimes_1\leq \rtimes_4$ on the category of separable $G$\nb-algebras as required.
\end{proof}

\begin{remark} The above corollary shows that (at least for second countable groups and separable $G$-algebras) the
reformulation of the Baum-Connes conjecture in \cite{Baum-Guentner-Willett} is equivalent to the statement that
$$ \as^{\E_\Cor}_{(A,G)}: K_*^{\top}(G;A)\mapsto K_*(A\rtimes_{\alpha, \E_\Cor}G)$$
is always an isomorphism, where $\rtimes_{\E_\Cor}$ is the minimal exact correspondence
functor. By the results of \S\ref{sec:descent} we can use the full force of equivariant $KK$-theory
to study this version of the conjecture.
\end{remark}

\section{Remarks and questions}\label{sec:questions}

\subsection{Crossed products associated to ${UC_{b}}(G)$}\label{sec:cub}

We only know one general construction of exact correspondence functors.  Fix a unital $G$\nb-algebra $(D,\delta)$.  Let $(A,\alpha)$ be a $G$\nb-algebra.  The \Star{}homomorphism
$$
A\to A\otimes_{{\max}} D,\quad a\mapsto a\otimes 1
$$
induces an injective \Star{}homomorphism
$$
C_c(G,A)\to (A\otimes_{\max} D)\rtimes_{\alpha\otimes \delta,\un}G.
$$
Recall from Section \ref{subsec-tensor} that the crossed product $A\rtimes_{\un^{\max}_D}G$ is by definition the closure of the image of $C_c(G,A)$ inside $(A\otimes D)\rtimes_{\alpha\otimes \delta,\un}G$.

The crossed product $\rtimes_{\un_D^{\max}}$ is always exact by \cite[Lemma 5.4]{Baum-Guentner-Willett}, and always a correspondence functor by Corollary \ref{cor-tensor}.  Here we will show that the collection
$$
\{\rtimes_{\un_D^{\max}}~:~(D,\delta) \text{ a $G$\nb-algebra}\}
$$
of exact correspondence functors -- which are the \emph{only} exact correspondence functors we know for general $G$ -- has a concrete minimal element.

\begin{lemma}\label{small}
Let $UC_b(G)$ denote the \cstar{}algebra of bounded, left-uniformly continuous complex-valued functions on $G$, equipped with the $G$\nb-action induced by translation.

Then for any unital $G$\nb-algebra $D$, $\rtimes_{\un_D^{\max}}\geq \rtimes_{\un^{\max}_{UC_b(G)}}$.
\end{lemma}

\begin{proof}
Consider any state $\phi:D\to \C$, and note that for any $d\in D$, the function $d_\phi:G\to \C$ defined by $d_\phi:g\mapsto \phi(\delta_g(d))$ is bounded and uniformly continuous, hence an element of $UC_b(G)$.  Consider the map
$$
\Phi:D\to UC_b(G),\quad d\mapsto d_\phi.
$$
This is clearly equivariant, unital and positive.  Moreover, positive maps of $C^*$-algebras with commutative codomains are automatically completely positive: indeed, post-composing with multiplicative linear functionals reduces this to showing that a state is a completely positive map to $\C$, and this follows from the GNS-construction which shows in particular that a state is a compression of a \Star{}homomorphism by a one-dimensional projection.  Hence $\Phi$ is completely positive.  Now consider the diagram
$$
\xymatrix{ (A\otimes_{\max}D)\rtimes_{\un}G \ar[rr]^-{(1\otimes \Phi)\rtimes_\un G}& & (A\otimes_{\max}UC_b(G))\rtimes_{\un}G \\
 A\rtimes_{\un_D^{\max}}G \ar[u]^{\iota_D} & & A\rtimes_{\un_{UC_b(G)}^{\max}}G \ar[u]^{\iota_{UC_b(G)}} },
$$
where the vertical maps are the injections given by definition of the crossed products on the bottom row, and the top line exists by Theorem \ref{thm-corr} and the fact that $1\otimes \Phi$ is completely positive.  As $\Phi$ is unital, the composition
$$
(1\otimes \Phi)\rtimes_\un G\circ \iota_D: A\rtimes_{\un_D^{\max}}G \to  \iota_{UC_b(G)}(A\rtimes_{\un_{UC_b(G)}^{\max}}G)
$$
identifies with a \Star{}homomorphism from $A\rtimes_{\un_D^{\max}}G$ to  $A\rtimes_{\un_{UC_b(G)}^{\max}}G$ that extends the identity map on $C_c(G,A)$.  This gives the desired conclusion.
\end{proof}

The crossed product $\rtimes_{\un_{UC_b(G)}^{\max}}$ has another interesting property.  Indeed, recall the following from \cite[Section 3]{ad:amen}.

\begin{definition}\label{ameninf}
A locally compact group $G$ is \emph{amenable at infinity} if it admits an amenable action on a compact topological space.
\end{definition}

\begin{lemma}\label{lem-ameninf}
Say $G$ is amenable at infinity.  Then for any $G$\nb-algebra $(A,\alpha)$,
$$
A\rtimes_{\alpha,\un_{UC_b(G)}^{\max}}G=A\rtimes_{\alpha,\red}G.
$$
\end{lemma}

\begin{proof}
If $G$ is amenable at infinity, then \cite{ad:amen}*{Proposition 3.4} implies that the action of $G$ on the spectrum $X$ of $UC_b(G)$ is amenable.  Hence for any $G$\nb-algebra $A$, the tensor product $A\otimes_{\max} UC_b(G)$ is a $G$-$C(X)$ algebra for an amenable $G$\nb-space $X$.  The result follows from \cite[Theorem 5.4]{ad:amen}, which implies that
$$
(A\otimes_{\max} UC_b(G))\rtimes_\un G=(A\otimes_{\max} UC_b(G))\rtimes_\red G. \eqno\qedhere
$$
\end{proof}

A group that is amenable at infinity is always exact \cite[Theorem 7.2]{ad:amen}.  For discrete groups, the converse is true by \cite{Ozawa:examen}, but this is an open question in general; nonetheless, many exact groups, for example all almost connected groups \cite[Proposition 3.3]{ad:amen}, are known to be amenable at infinity.

To summarise, we have shown that $\rtimes_{\un_{UC_b(G)}^{\max}}$ is an exact correspondence functor; that it is minimal among a large family of exact correspondence functors; and that for many (and possibly all) exact groups, it is equal to the reduced crossed product.  The following question is thus very natural.

\begin{question}\label{ques}
Is the crossed product $\rtimes_{\un^{\max}_{UC_b(G)}}$ equal to the minimal exact correspondence crossed product $\rtimes_{\E_\Cor}$?
\end{question}

There are other natural questions one could ask about  $\rtimes_{\un^{\max}_{UC_b(G)}}$: for example, is it a KLQ functor `in disguise'?  One could also ask this about any of the functors $\rtimes_{\un^{\max}_D}$.

\subsection{Questions about $L^p$ functors}

Recall from Section \ref{sec:lp} that $E_p$ denotes the weak{*}-closure of $B(G)\cap L^p(G)$ in $B(G)$ for $p\in [2,\infty]$, and $E_0$ is the weak{*}-closure of $B(G)\cap C_0(G)$.  The following questions about these ideals and the corresponding group algebras and crossed products seem natural and interesting.
\begin{enumerate}
\item Say $G$ is an exact group.  Are all of the functors $\rtimes_{E_p}$ exact?  More generally, if $G$ is an exact group, are all KLQ-crossed products
exact?  Note that Example \ref{ex-ideal-property} shows that any non-amenable exact group admits a non-exact crossed-product functor.
\item Are there non-exact groups and $p\in\{0\}\cup(2,\infty)$ such that $\rtimes_{E_p}$ is exact?
\item For which locally compact groups does the canonical quotient $q:C^*_{E_0}(G)\to C^*_{\red}(G)$ induce an isomorphism on $K$-theory (one could also ask this for other $p$)?  This cannot be true in general by Example \ref{ex:spn1}, but we do not know any discrete groups for which it fails.
\item For which groups $G$ are all the exotic group \cstar{}algebras $C^*_{E_p}(G)$ different?  Example \ref{ex:nodiff} shows that this cannot be true for general non-amenable groups, but it could in principal be true for all non-amenable discrete groups.
\item For which groups is $E_0$ the weak{*}\nb-closure of $\cup_{p<\infty} E_p$?  This holds for all simple Lie groups with finite center by results of Cowling \cite{Cowling}, and we showed it holds for a fairly large class of discrete groups in Lemma \ref{lem:e0}.  Conceivably, it could hold in full generality.  More generally, one can ask: is $E_q$ is the weak{*-} closure of $\cup_{p<q}E_p$ for all $q\in [2,\infty)$?  Similarly, is it true that $E_p=\cap_{q>p}E_q$ for all $p\in[2,\infty)$?  Okayasu \cite[Corollary 3.10]{Okayasu:Free-group} has shown that this is true for ${\Free_2}$ and Cowling, Haagerup and Howe \cite{chh:l2} have shown that $E_2=\cap_{p>2}E_p$ is always true.  Not much else seems to be known here, however.
\end{enumerate}

\end{document}